\documentclass[a4paper]
{cedram-smai-jcm}

\usepackage{tikz}
\usepackage{subcaption}
\usepackage{amssymb, amsmath, amsthm,stmaryrd}
\usepackage{mathtools}
\usepackage{todonotes}
\usepackage{enumerate}
\usepackage{enumitem}

\newtheorem{assum}{Assumption}

\DeclareMathOperator{\dis}{dist}
\newcommand{\tnorm}[1]{\vert\hspace{-0.3mm}\Vert#1\Vert\hspace{-0.3mm}\vert}
\providecommand{\abs}[1]{\left\lvert#1\right\rvert}
\providecommand{\norm}[1]{\left\lVert#1\right\rVert}
\DeclareMathOperator{\dist}{dist}
\renewcommand{\leq}{\leqslant}
\renewcommand{\geq}{\geqslant}

\newcommand{\dPhi}{ \underline{A} }
\newcommand{\dX}{\mathrm{d}x}
\newcommand{\dY}{\mathrm{d}y}
\newcommand{\dS}{\mathrm{d}S}
\newcommand{\jump}[1]{\llbracket#1\rrbracket}
\DeclarePairedDelimiterX{\inp}[2]{(}{)}{#1, #2}
\newcommand{\tangular}[1]{ \llbracket\kern-0.5ex|#1|\kern-0.5ex\rrbracket} 
\newcommand{\avg}[1]{ \{\!\!\{#1\} \!\!\} }
\newcommand{\calL}{\mathcal{L}}
\newcommand{\Ext}{\mathcal{E}}
\newcommand{\ID}{\mathrm{id}}
\newcommand{\mesh}{\mathcal{T}_h}
\newcommand{\facets}{\mathcal{F}}
\newcommand{\el}{T}
\newcommand{\BrokenOmega}{\Omega_{1} \cup \Omega_{2}}
\newcommand{\BrokenOmegah}{\Omega_{1,h} \cup \Omega_{2,h}}
\newcommand{\ActiveOmega}{\Omega^{+}}
\newcommand{\lset}{ \phi }
\newcommand{\OmegaLin}{ \Omega^{\text{lin}} }
\newcommand{\GammaLin}{ \Gamma^{\text{lin}} }
\newcommand{\Cutel}{ \mathcal{T}^{\Gamma} } 
\newcommand{\CutelExt}{ \mathcal{T}^{\Gamma}_{+}} 
 
\newcommand{\Cutdom}{ \Omega^{\Gamma} } 
\newcommand{\CutdomExt}{ \Omega^{\Gamma}_{+}}

\newcommand{\curvedFes}[1]{ V_{h,#1} }
 
\newcommand{\FES}[2]{ V_{h,#1}^{#2}}
\newcommand{\normal}{ n } 
\newcommand{\Ptan}[1]{ P_{#1} }
\providecommand{\StabCIP}[2]{ J_{ \mathrm{CIP} } \left( #1, #2 \right) }
\providecommand{\StabGLS}[2]{ J_{ \mathrm{GLS} } \left( #1, #2 \right) }

\providecommand{\StabTikh}[2]{ J_{ \alpha } \left( #1, #2 \right) }
\providecommand{\StabNitsche}[2]{ J_{ \mathrm{H} }^{ \Gamma_h } \left( #1, #2 \right) }
\providecommand{\StabNablaGamma}[2]{ J_{ \mathrm{T} }^{ \Gamma_h } \left( #1, #2 \right) }
\providecommand{\StabNablaNormal}[2]{ J_{ \mathrm{N} }^{ \Gamma_h } \left( #1, #2 \right) }

\providecommand{\StabIFall}[2]{ J^{ \Gamma_h } \left( #1, #2 \right) }
\newcommand{\InterpBackgroundMesh}{I_{h}}

\newcommand{\InterpUnfitted}{ I_h^{\Gamma} }
\newcommand{\InterpBackgroundMeshZ}{I_{h}^{0}}
\newcommand{\InterpUnfittedZ}{ \InterpBackgroundMeshZ }
\newcommand{\ExtMeshTrafo}{E}

\title[Unique continuation over an interface]{Unique continuation for an elliptic interface problem using unfitted isoparametric finite elements}

\author[E. Burman]{\firstname{Erik} \lastname{Burman}}
\address{Department of Mathematics, University College London, United Kingdom}
\email{e.burman@ucl.ac.uk} 
\thanks{This work was funded EPSRC grant EP/V050400/1.}

\author[J. Preuss]{\firstname{Janosch} \lastname{Preuss}}
\address{Department of Mathematics, University College London, United Kingdom}
\email{j.preuss@ucl.ac.uk}

  

\begin{document}

\begin{abstract}
We study unique continuation over an interface using a stabilized unfitted finite element method tailored to the conditional stability of the problem.
The interface is approximated using an isoparametric transformation of the background mesh and 
the corresponding geometrical error is included in our error analysis.
To counter possible destabilizing effects caused by non-conformity of the discretization and cope with the interface conditions, we introduce adapted regularization terms. 
This allows to derive error estimates based on conditional stability. 
The necessity and effectiveness of the regularization is illustrated in numerical experiments. 
We also explore numerically the effect of the heterogeneity in the coefficients on the ability to reconstruct the solution outside the data domain. For Helmholtz equations we find that a jump in the flux impacts the stability of the problem significantly less than the size of the wavenumber.
\end{abstract}

\maketitle

\section{Introduction}\label{sec:intro}

\subsection{Motivation}\label{ssection:motivation}
Recently, there has been an increasing interest in the development of numerical methods for unique continuation problems, see e.g.\ \cite{BHL18a,BNO19,BB20,BFMO21,BDE21,MM22,DMS22}.
These methods are based on so called conditional stability estimates which represent a quantitative 
form of the unique continuation property for solutions of certain partial differential equations (PDEs).
In addition, many physical applications involve interfaces over which material parameters can exhibit jump discontinuities, e.g.\ in seismic wave propagation. 
Despite the practical relevance of these problems, numerical methods for unique continuation involving interfaces seem to be unavailable in the literature.
\par 
On the other hand, the development of numerical methods, e.g.\ finite elements methods (FEMs), for well-posed problems involving interfaces is well advanced. 
So called unfitted FEMs are very appealing to treat these problems, see e.g. \cite{HH02,ZCC11,M12,BC15,WX17,LR17,BE18HHO,YJF20,CLX21}, since they allow the use of simple meshes which are independent of the possibly complex geometry of the interface.
However, the analysis and implementation of unfitted FEMs imposes some challenges (stability of variational formulations, usually only an implicit description of the geometry is available, ...) which are poorly understood in the context of ill-posed problems. 
In this paper, we make some progress in this direction by developing an unfitted FEM for a unique continuation problem involving an interface.
\subsection{The problem setting}
We define the problem setting as sketched in Figure~\ref{fig:problem-overview-unfitted}. 
Let $\Omega \subset \mathbb{R}^d, d \in \{2,3\}$, be a bounded, connected and open set split by a smooth interface $\Gamma$ into two connected components $\Omega_1$ and $\Omega_2$ such that $\Gamma = \partial \Omega_1$ 
and $\Omega = \Omega_1 \cup \Gamma \cup \Omega_2$. 
Given the functions $u_i$ on $\Omega_i$, we identify $u$ on $\BrokenOmega$ with the pair $(u_1,u_2)$. 
\begin{figure}[htbp]
\centering
\includegraphics[scale=0.19]{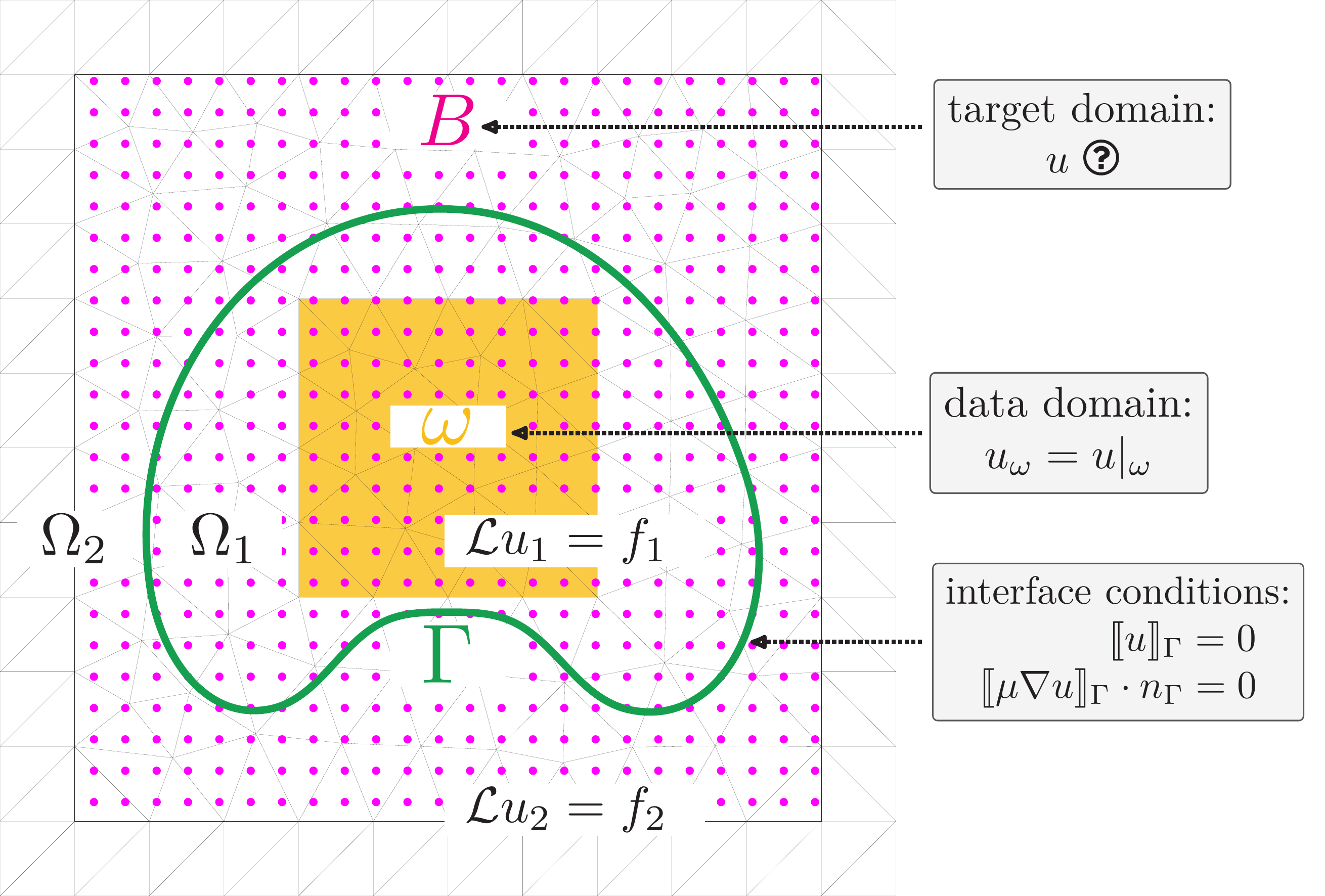}
\caption{Sketch of geometry and data for unique continuation over an interface. } 
\label{fig:problem-overview-unfitted}
\end{figure}
We consider the elliptic interface problem 
\begin{align}
	\calL u = f & \text{ in } \BrokenOmega, \label{eq:PDE_bulk} \\
	\jump{u}_{\Gamma} = 0 & \text{ on } \Gamma, \label{eq:jump_IF} \\ 
	\jump{\mu \nabla u }_{\Gamma} \cdot n_{\Gamma} = 0 & \text{ on }  \Gamma, \label{eq:normal_jump_IF}
\end{align}
with $ \calL u := - \nabla \cdot (\mu \nabla u) - \rho u $  for a quantity $\mu$ that is discontinuous across $\Gamma$ and constant on each $\Omega_i$ with $\mu_i := \mu|_{\Omega_i} >0 $.  
Similarly, $\rho$ is assumed to be piecewise constant such that $\rho_i :=  \rho|_{\Omega_i}  \in \mathbb{R}$. 
Here we are primarily interested in the case of the Helmholtz equation for which $\rho_i = k_i^2 $ for some $k_i > 0$ representing 
the wavenumber, but the case of diffusion-reaction equations for which $\rho_i \leq 0$ is covered as well.
The jump over the interface $\Gamma$ is defined for $v = (v_1,v_2) \in H^1(\BrokenOmega) $ by $\jump{v}_{ \Gamma} := v_1 -  v_2$ and analogously for vector-valued functions. 
We fix the unit normal vector $\normal_{\Gamma}$ of $\Gamma$ to point from $\Omega_1$ into $\Omega_2$. \par 
Note that no boundary data on $\partial \Omega$ is assumed to be given, which renders this problem ill-posed. 
To recover a weak form of stability, we will assume that measurements 
\begin{equation}\label{eq:measurements}
u_{\omega} = u|_{\omega} 
\end{equation}
of a solution of (\ref{eq:PDE_bulk})-(\ref{eq:normal_jump_IF}) in a subset $\omega \subset \Omega$ are available. More precisely, we assume that $\omega$ is compactly contained in either $\Omega_1$ or $\Omega_2$.
The objective is then to extend the solution from $\omega$ into a larger set $B \subset \Omega$ across the interface. 
Possible applications where this task occurs include inverse problems in subsurface flows, seismology and biomedical applications.
From \cite[Theorem 1.1]{CW20}, see Appendix~\ref{A:cond-stab} for details, we obtain that the
following conditional stability estimate of H\"older-type holds.
\begin{coro}\label{cor:conditional_stabiliy_estimate}
Let $\omega \subset B \subset \Omega$ such that $B \setminus \omega$ does not touch the boundary of $\Omega$ and $\omega$ is 
contained either in $\Omega_1$ or $\Omega_2$. Then there exists a $\tau \in (0,1)$ such that for all $u \in H^1(\Omega)$ we have
\begin{equation}\label{eq:cond_stab}
\norm{ u }_{B} \leq  C \left( \norm{ u  }_{\omega} + \norm{ \calL u }_{ H^{-1}(\Omega)}  \right)^{\tau}
\left( \norm{ u  }_{\Omega} + \norm{ \calL u }_{ H^{-1}(\Omega)}  \right)^{1 - \tau}.
\end{equation}
\end{coro}
Here we introduced for $M \subset \mathbb{R}^d$ and $u,v \in L^2(M)$ the notation
\[ (u,v)_{M} := \int\limits_{M} u v \; \dX, \qquad \norm{v}_{M} :=  \sqrt{ (v,v)_{M} }.   \]
We point out that the constant $C$ and the parameter $\tau$ in (\ref{eq:cond_stab}) depend on bounds of the coefficients of the operator $\calL$ and on certain geometrical properties. The explicit form of this dependence is in general not known. 
What can be characterized though for the case when $\omega$ is a ball contained in $B$ is the rate at which $C$ and $\tau$ degenerate as the distance between $B$ and the boundary 
of $\Omega$ diminishes. We refer the reader to \cite[Theorem 1.1]{CW20} for the details.

\subsection{Content and structure of the article}
In this article we propose a stabilized unfitted FEM for the numerical solution of (\ref{eq:PDE_bulk})-(\ref{eq:measurements}). 
To cope with the fact that the problem merely enjoys conditional stability, see Corollary~\ref{cor:conditional_stabiliy_estimate}, we use the stabilization framework introduced in \cite{B13}, which has been applied successfully to several different problems, see e.g.\ \cite{BNO19,BNO20,BNO22,BP22,BDE21}, using geometrically \textit{fitted} discretizations. 
It turns out that some additional stabilization terms on the interface are necessary 
to allow for the use of an \textit{unfitted} discretization. \par 
Unfitted methods are usually based on an implicit description of the geometry, e.g.\ by levelset functions. Robust and accurate numerical integration on implicitly defined domains is challenging, which renders achieving high-order geometric accuracy with unfitted methods a difficult endeavor. 
Here, we adopt an approach developed in \cite{L_CMAME_2016,LR17,LR18} based on an isoparametric mapping of the background mesh. 
 In this framework the exact interface $\Gamma$ is represented numerically by an interface $\Gamma_h$ such that $\dist(\Gamma, \Gamma_h) \leq C h^{q+1}$ holds for a parameter $q \in \mathbb{N}$ which controls the geometric accuracy and a constant $C >0 $ independent of the mesh width $h$.
In contrast to well-posed problems, we find that the geometric error (stemming from $\Gamma \neq \Gamma_h$) exerts an additional destabilizing effect which has to be countered by a suitable regularization term, see e.g.\ Figure~\ref{fig:non-convex-diffusion-Stab} and Figure~\ref{fig:convex-relerr-exact-data}.  \par 
As our main result (see Theorem~\ref{thm:L2B-error-estimate}) we show that the following error estimate for the extension of the solution given on the data domain $\omega$ to the target domain $B$ holds:
\[
 \norm{ u - u_h \circ \Phi_h^{-1}  }_{ B } \leq C  h^{\tau q} \left(  \norm{u}_{q} + h^{-q} \delta  \right).
\]
Here, $u_h$ is the numerical solution in an unfitted finite element space of order $p \geq q \geq 1$ and $\Phi_h$ is a continuous and piecewise smooth transformation mapping the exact to the approximate geometry. 
The term $\delta $ represent a data perturbation, $\tau \in (0,1)$ is the H\"older exponent from 
Corollary~\ref{cor:conditional_stabiliy_estimate} and $\norm{u}_{q}$ some Sobolev norm of the exact solution.
Comparing our theoretical result with \cite{BNO19,BNO20,BNO22,BP22}, we conclude that (granted suitable stabilization is added) unfitted methods allow to solve unique continuation problems as good as their fitted counterparts. This is also corroborated by a numerical comparsion with fitted methods presented in Section~\ref{ssection:fitted-vs-unfitted}. \par
We summarize the three main contributions of this article:
\begin{itemize}
\item We solve a unique continuation problem which involves an \textit{interface}.
\item We use an \textit{unfitted} method to discretize an \textit{ill-posed} problem and 
provide a fairly complete error analysis complemented by numerical experiments.
\item In particular, our analysis allows arbitrary polynomial orders and includes \textit{geometric errors} 
based on the technique from \cite{L_CMAME_2016,LR17,LR18}.  
\end{itemize}
Let us also take the opportunity to mention open problems that are beyond the scope of this article:
\begin{itemize}
\item In reference \cite{BNO19} unique continuation for the constant coefficient Helmholtz equation has been considered. 
Under a certain convexity condition on the geometry of data and target domain conditional stability estimates similar to Corollary~\ref{cor:conditional_stabiliy_estimate} have been 
derived which are robust in the wavenumber, see e.g.\ \cite[Lemma 2]{BNO19}. It is not clear whether such results can be established for 
		the interface problem considered here. As a result, we cannot characterize the dependence of our final Theorem~\ref{thm:L2B-error-estimate} on the geometrical configuration explicitly. However, we point out that all the results preceding Theorem~\ref{thm:L2B-error-estimate} are independent of the stability properties of the continuous problem (in particilar of the geometric configuration) since Corollary~\ref{cor:conditional_stabiliy_estimate} is only applied to establish the final result.
\item The error estimate in Theorem~\ref{thm:L2B-error-estimate} is not robust in the contrast of the coefficients $\mu_i$ and $\rho_i$. 
This stems on the one hand from the fact that the conditional stability estimate in Corollary~\ref{cor:conditional_stabiliy_estimate} is not robust in the coefficients either 
		and on the other hand from our decision not to keep track of the dependence of the constants on bounds of the coefficients in the error analysis preceding Theorem~\ref{thm:L2B-error-estimate}. Actually we expect that even for the pure Poisson problem (i.e.\ $\rho_i = 0$) it is unlikely that error estimates that are robust in the diffusion coefficients can be obtained in the ill-posed setting considered here. This is inferred from the analysis for the well-posed case performed in reference \cite{BGSS17} which showed that 
achieving robustness requires to handle a certain flux related quantity over which we lack control in the ill-posed setting.
\end{itemize}
The remainder of this article is structured as follows. 
In Section~\ref{section:IF-iso} we explain how the geometry is resolved using the isoparametric mapping technique from \cite{L_CMAME_2016,LR17}. 
Some additional results for this technique are derived in this article which may be of independent interest.
For example, we derive useful relations between derivatives on the exact and approximate interface and show how the isoparametric mapping can be combined with a Galerkin-least squares stabilization. These results could be of interest in various settings, e.g.\ for convection-dominated problems, fluid problems involving interfaces or coupled bulk-surface problems. In Section~\ref{section:iso-FEM} an unfitted isoparametric FEM for the problem (\ref{eq:PDE_bulk})-(\ref{eq:measurements}) is introduced whose numerical analysis is carried out in Section~\ref{section:error-analysis}.  
Numerical experiments are presented in Section~\ref{section:numexp}. We conclude in Section~\ref{section:conclusion}.

\section{Resolving the interface by an isoparametric mapping}\label{section:IF-iso}
This section deals with the geometrical approximation of the interface.
In Subsection~\ref{ssection:exact-geom}-\ref{ssection:iso-map} three different levels
of geometric accuracy will be introduced. 
We start in Subsection~\ref{ssection:exact-geom} with the exact geometry, 
proceed in Subsection~\ref{ssection:p1-geom} with the simplest approximation by a piecewise linear reference configuration and finish in Subsection~\ref{ssection:iso-map} with the higher-order accurate approach from \cite{L_CMAME_2016,LR17} based on an isoparametric mapping.
In Subsection~\ref{ssection:iso-basic-properties} and Subsection~\ref{ssection:geom-err-iso} the basic properties of this mapping are recalled and some further extensions are derived.
Please note that we have outsourced the proofs of this section into Appendix~\ref{B:geom}.

\subsection{The exact geometry}\label{ssection:exact-geom}

It is very common in unfitted FEMs to assume that the interface can be described by a smooth levelset function $\phi$ such that $\Gamma := \{ x \in \Omega \mid \lset(x) = 0  \}$. 
The corresponding subdomains are then given by $\Omega_{1} := \{  \lset < 0 \}$ and $ \Omega_{2} := \{  \lset > 0 \}$. 
We introduce operators for restricting to these subdomains as follows:
\begin{equation}\label{eq:Restrictions-phys}
(R_1 u)(x) :=
\left\{ \begin{array}{cc} u(x) &  x \in  \Omega_1,\\
 0 &   \text{ else, }  \end{array}\right.
\quad 
(R_2 u)(x) :=
\left\{ \begin{array}{cc} u(x) &  x \in  \Omega_2,\\
 0 &   \text{ else. }  \end{array}\right.
\end{equation}
Conversely, we need to be able to extend a function defined merely on a subdomain smoothly to all of $\Omega$. 
To this end, we utilize that for Lipschitz domains, see \cite[Theorem 5, Chapter VI]{S70}, a linear and continuous Sobolev extension operator 
\begin{equation}\label{eq:def-Sobolev-Extension}
\Ext_i:W^{m,p}(\Omega_i) \rightarrow W^{m,p}(\mathbb{R}^d), 
\quad \text{such that } \Ext_i |_{\Omega_i} = \ID,
\; i =1,2,
\end{equation}
exists for all $m \in \mathbb{N}_0$ and $1 \leq p \leq \infty$. 
Here, $W^{m,p}(\Omega_i)$ are the standard Sobolev spaces which we denote by $H^m(\Omega_i) $ when $p=2$ as usual. 
For $u \in W^{m,p}(\BrokenOmega) $ it is convenient to use the notation $\Ext u = ( \Ext u_1  , \Ext u_2 )$. \\ 
Further, we will have to consider derivatives along and normal to the interface. 
As in the introduction let $\normal_{\Gamma}$ be the unit normal vector on $\Gamma$ pointing from $\Omega_1$ into $\Omega_2$. 
Based on $\normal_{\Gamma}$ we define the tangential projection $\Ptan{\Gamma} := I - \normal_{\Gamma} (\normal_{\Gamma})^T$ and surface gradient $\nabla_{\Gamma} := \Ptan{\Gamma} \nabla$. 
\subsection{The piecewise linear reference geometry}\label{ssection:p1-geom}
Let us first introduce a  quasi-uniform and simplicial triangulation $\mesh$ of $\Omega$, which is assumed to be independent of $\Gamma$. 
We use the notation $\lesssim$ (and $\sim$ for $\lesssim$ and $\gtrsim$) to denote that an inequality 
holds with a constant independent of $h$ and how the interface intersects the triangulation $\mesh$.
The constant may however depend on the coefficients of the differential operator $\calL$ in (\ref{eq:PDE_bulk}).
In this work we would like to focus attention on how the geometrical resolution of the 
interface influences the reconstruction. To eliminate other sources of geometrical 
errors we assume (as in the sketch given in Figure~\ref{fig:problem-overview-unfitted}):
\begin{assum}\label{assum:mesh-fits-Omega-omega}
Let $\Omega$ and $\omega$ be polygonal domains which are exactly fitted by $\mesh$. 
\end{assum}
We denote by $\FES{p}{}$ the standard $H^1$-conforming finite element space on $\mesh$ based on piecewise polynomials of order $p$ and set $\FES{p}{0} := \{ v_h \in  \FES{p}{} \mid v_h|_{\partial \Omega} = 0  \}$. 
Let $\hat{\lset}$ be the piecewise linear nodal interpolation of $\lset$ into $\FES{1}{}$.
Based on $\hat{\lset}$ we define the piecewise linear reference interface $\GammaLin := \{  x \in \Omega \mid \hat{\lset}(x) = 0  \}$, respectively the subdomains $\OmegaLin_{1} := \{  \hat{\lset} < 0 \}$ and $ \OmegaLin_{2} := \{  \hat{\lset} > 0 \}$. 
An illustration of some of these geometrical quantities is given in Figure~\ref{fig:Stab-sketch-mod}. 
Further, the intersection of an element with $\GammaLin$ is denoted by $\Gamma_{\el} := \el \cap \GammaLin $.  We combine all the elements cut by the interface into the set $\Cutel :=  \{ \el \in \mesh \mid \Gamma_{\el} \neq \emptyset \}  $ with corresponding domain $\Cutdom := \{ x \in T \mid T \in \Cutel \}$.\\ 
We now turn our attention to the elements in the bulk.  
Let $ \el_i := \el \cap \OmegaLin_{i}  $ denote the portion of the element in $\OmegaLin_{i} $ 
and define the active mesh $\mesh^{i} := \{ \el \in \mesh \mid \el_i \neq \emptyset \} $ and corresponding active domain $\ActiveOmega_i :=  \{ x \in \el : \el \in \mesh^{i}  \}$.  
As for the physical domains we define corresponding restriction operators which give rise to the so called Cut-FEM space:
\begin{equation}\label{eq:Cutfes}
(R_i^+ u)(x) :=
\left\{ \begin{array}{cc} u(x) &  x \in  \Omega_i^+,\\
 0 &   \text{ else, }  \end{array}\right. 
 \qquad 
\FES{p}{\Gamma} :=  R_1^+ \FES{p}{} \oplus  R_2^+ \FES{p}{}. 
\end{equation}
Elements $v_h \in \FES{p}{\Gamma} $ have the form $v_h = (v_{h,1},v_{h,2})$ with  $v_{h,i} \in R_{i}^{+} \FES{p}{} $.
Note that the introduced description of the geometry is only second order accurate, i.e.\ $\dist(\Gamma, \Gamma_h)  \lesssim \mathcal{O}(h^{2})$, which limits the overall accuracy of derived numerical method unless some form of correction is applied, see e.g.\ \cite{BHL18,BBCL18}.
\subsection{The isoparametric mapping}\label{ssection:iso-map}
The higher-order description of the geometry from \cite{L_CMAME_2016,LR17} is  based on an interpolant $\lset_h \in \FES{q}{}, q \in \mathbb{N}$ of the levelset function in a higher-order finite element space. To circumvent that for $q \geq 2$ the geometry description given by $\lset_h$ is only implicit, a mapping $\Theta_h \in [\FES{q}{} ]^d$ is constructed based on $\lset_h$ and $\hat{\lset}$ which maps the piecewise linear reference geometry to a higher order accurate approximation of the exact geometry, i.e.\ we have 
\[
\hat{\lset} \approx \lset_h \circ \Theta_h \; \text{and for } \Gamma_h := \Theta_h( \GammaLin) \text{ it holds: }  \dist(\Gamma, \Gamma_h) \lesssim \mathcal{O}(h^{q+1}).
\]
This means that $\Gamma_h$ defined as the image of the piecewise linear reference interface under the mapping $\Theta_h$ and  $\Omega_{i,h} := \Theta_h( \OmegaLin_{i} )$ approximate the exact interface and subdomains with order $\mathcal{O}(h^{q+1})$ and thanks to  
\begin{equation}\label{eq:integral-trafo}
\int\limits_{\Omega_{i,h}}{f  \; dx} = \int\limits_{ \Omega^{\text{lin}}_{i}}{ f \circ \Theta_{h} \abs{ \det{ \left( D \Theta_{h}  \right) }} \; dy}
\end{equation}
it suffices to compute integrals on the piecewise linear reference configuration which is well-understood, see e.g.\ \cite[Section 5]{BC15} for details. 
We also define $\Omega_{i,h}^+ := \Theta_h( \Omega_i^+ )$ and remark that the regularity $\lset \in C^{q+2}(U)$ where $U$ is a neighborhood of $\Gamma$ 
is assumed to guarantee the approximation properties of the mapping $\Theta_h$, see \cite[Section 2.1]{LR17}. \\ 
For the analysis (not the implementation) we will need another mapping $\Psi$ which maps the piecewise linear reference geometry back to the exact geometry, in particular \ $\Psi(\GammaLin) = \Gamma$ holds. The transformation $\Theta_h$ is invertible for sufficiently small $h$ (see \cite[Section 3.5]{LR17}) and we can define $\Phi_h = \Psi \circ \Theta_h^{-1}$, which then fulfills  $\Phi_h(\Gamma_h) = \Gamma$ and has the smoothness property $\Phi_h \in [C(\Omega)]^d \cap [C^{q+1}( \Theta_h(\mesh))]^d$, where $C^{q+1}( \Theta_h(\mesh)) := \oplus_{\el \in \mesh} C^{q+1}( \Theta_h(\el)) $, see \cite{LR18}. 
For a detailed description on how the mappings $\Theta_h$ and $\Psi$ are constructed we refer to \cite[Section 3]{LR17}. Here, we only recall the fact that both mappings are small perturbations of the identity whose action is localized in the vicinity of the interface.
In particular, as $\omega$ is compactly contained in one of the subdomains it follows that for sufficiently small $h$ it holds
\begin{equation}\label{eq:omega-trafo-id}
\Phi_{h} |_{\omega} = \ID.
\end{equation}
Let us close this subsection by defining some geometrical quantities on the discrete interface $\Gamma_h$. 
Let $\normal_{\Gamma_h}$ denote the unit normal vector on $\Gamma_h$. 
We also define the discrete tangential projection $\Ptan{\Gamma_h} := I - \normal_{\Gamma_h} (\normal_{\Gamma_h})^T$ and surface gradient $\nabla_{\Gamma_h} := \Ptan{\Gamma_h} \nabla$.  

\subsection{Basic properties of the mappings}\label{ssection:iso-basic-properties}
As noted above, the mesh transformation is essentially a small perturbation of the identity. A quantitative version of this statement has been given in \cite[Lemma 5.5]{LR17} which is recalled below. 
\begin{lemm}\label{lem:Trafo-Grad-bounds}
Let $\dPhi := D\Phi_h$. 
The following holds 
\begin{align}
\norm{\Theta_h - \Psi }_{\infty,\Omega} + h \norm{D (\Theta_h - \Psi) }_{\infty,\Omega} &\lesssim h^{q+1}, \label{eq:Bound-Theta-Psi-Grad} \\ 
\norm{\Phi_h -\ID}_{\infty,\Omega} + h \norm{ \dPhi - I }_{\infty,\Omega} & \lesssim h^{q+1}. \label{eq:Bound-Phi-Grad}  
\end{align}
\end{lemm}
Here, $\norm{\cdot}_{\infty,M}$ denotes the $L^{\infty}$-norm on $M \subset \mathbb{R}^d$.
From the proof of \cite[Lemma 5.10]{LR17} we also recall the following estimates:   
\begin{equation}\label{eq:DPhi_estimates}
\norm{\dPhi^{-1} -I}_{\infty,\Omega} \lesssim h^q, \quad
\norm{ \det(\dPhi) -1}_{\infty,\Omega} \lesssim h^q, \quad 
\norm{  C_{\Phi_h} }_{\infty,\Omega} \lesssim h^q,
\end{equation}
where $C_{\Phi_h} := (\dPhi^T \dPhi)^{-1} \det(\dPhi) - I $.
Notice also that this implies $ \norm{\det(A^{-1})}_{\infty,\Omega} \lesssim 1$ for $h$ sufficiently small.
We will additionally require bounds for higher derivatives of $\Phi_h$ which were stated in \cite[Remark 5.6]{LR17} without proof.
\begin{lemm}[Higher order bounds for $\Phi_h$]\label{lem:Phi-HigherOrderBounds}
For $l = 2, \cdots, q+1$ it holds that
\begin{enumerate}[label=(\alph*)]
\item $\max_{ T \in \mesh  } \norm{D^l ( \Psi - \Theta_h ) }_{\infty,T} \lesssim h^{q+1-l}  $. 
\item  $\max_{ T \in \mesh  }  \norm{ D^l ( \Phi_h - \ID ) }_{\infty,\Theta_h(T)} = \max_{ T \in \mesh  }  \norm{ D^l \Phi_h  }_{\infty,\Theta_h(T)}  \lesssim  h^{q+1-l}.  $   
\end{enumerate} 
\end{lemm}
\begin{proof}
Given in Appendix~\ref{B:geom} like the other proofs of this section.
\end{proof}
In the analysis we will frequently have to transform between exact and discrete interface.
To this end, it is important to understand how the respective normal vectors and associated projections are related.
We recall from \cite[eq. (A.20)]{LR17} that the transformation $\Phi_h$ induces the relations 
\begin{equation}\label{eq:normal-trafo-Phi}
\normal_{\Gamma} \circ \Phi_h = \frac{ \dPhi^{-T} \normal_{\Gamma_h}  }{ \norm{\dPhi^{-T} \normal_{\Gamma_h}  }_2    }, \quad 
\normal_{\Gamma_h} \circ \Phi_h^{-1} =  \frac{ \dPhi^{T} \normal_{\Gamma}  }{ \norm{\dPhi^{T} \normal_{\Gamma}  }_2    }.
\end{equation}
Using these relations and (\ref{eq:Bound-Phi-Grad}) it is easy to show:
\begin{lemm}[Normal and tangential projection]\label{lem:normal_tan_proj}
For $h$ sufficiently small the following estimates hold uniformly on $\Gamma$, respectively $\Gamma_h$.
\begin{enumerate}[label=(\alph*)]
\item
On $\Gamma$ we have 
\[
\norm{ \dPhi^{T} \normal_{\Gamma} }_{2} \gtrsim 1 
\text{ and } \abs{ \norm{ \dPhi^{T} \normal_{\Gamma} }_{2 } - 1 } \lesssim h^q, 
\]
and on  $\Gamma_h$: 
\[
\norm{ \dPhi^{-T} \normal_{\Gamma_h} }_{2} \gtrsim 1 
\text{ and } \abs{ \norm{ \dPhi^{-T} \normal_{\Gamma_h} }_{2 } - 1 } \lesssim h^q. 
\]
\item The perturbation of the exact normal vector is bounded by: 
\[
\norm{ \normal_{\Gamma}  - \normal_{\Gamma_h} \circ \Phi_h^{-1} }_2  \lesssim h^q \text{ on } \Gamma,  \quad	\norm{ \normal_{\Gamma} \circ \Phi_h - \normal_{\Gamma_h} }_2  \lesssim h^q \text{ on } \Gamma_h. \] 
\item As a result, the tangential projections satisfy: 
\[  \norm{ \Ptan{\Gamma_h} \circ \Phi_h^{-1} - \Ptan{ \Gamma } }_{2} \lesssim h^q,  
\text{ on } \Gamma, \quad
\norm{ \Ptan{\Gamma} \circ \Phi_h - \Ptan{ \Gamma_h } }_{2} \lesssim h^q,  \text{ on } \Gamma_h. 
\]
\end{enumerate}
\end{lemm}
We remark that results similar to Lemma~\ref{lem:normal_tan_proj} already appear in \cite[Lemma 3.3]{GLR18} and \cite[Lemma 17]{YL_PHD_2023}.

\subsection{Geometry errors for mapped functions}\label{ssection:geom-err-iso}
Since $\Phi_h$ maps the discrete to the exact geometry, precomposition with this map allows to pull back functions defined on the exact geometry to the discrete geometry and vice versa. 
The following two lemmas discuss how various norms of these functions are related.
\begin{lemm}[norm equivalence]\label{lem:norm-equiv}
Let $\tilde{v} := v \circ \Phi_h$ for $v \in H^1( \Omega_{i} ) \cap H^2(\Psi(\mesh^{i} )) $. Then it holds that 
\begin{enumerate}[label=(\alph*)]
\item $\norm{ \tilde{v} }_{  \Omega_{i,h}   } \sim \norm{ v }_{  \Omega_{i}  } $ and 
$\norm{ \nabla \tilde{v} }_{  \Omega_{i,h}   } \sim \norm{ \nabla v }_{  \Omega_{i}  } $. 
\item For the second derivatives we have for $\nu,\mu \in \{1,\ldots,d\}$: 
\[
\sum\limits_{ \el \in \mesh^{i} }  \int\limits_{ \Theta_h(T) } \abs{ \partial_{y_{\nu}} \partial_{y_{\mu}} \tilde{v}  }^2 \; \dY
\lesssim \sum\limits_{ \abs{\alpha} \leq 2  } \sum\limits_{ \el \in \mesh^{i} }  \int\limits_{ \Psi(T) } \abs{ D_x^{\alpha} v  }^2 \; \dX. 
\]
\end{enumerate}
\end{lemm}
While the norms involving the full gradients on $\Gamma$ and $\Gamma_h$ are equivalent, the corresponding statement for the normal and tangential derivatives only holds up to a geometrical error proportional to $h^q$. 
In view of Lemma~\ref{lem:normal_tan_proj} this has to be expected.
\begin{lemm}[Derivatives at the interface]\label{lem:Deriv-IF}
Let $ \tilde{v} := v \circ \Phi_h$ for $v \in H^1(\Gamma)$. Then 
\begin{enumerate}[label=(\alph*)] 
\item We have $\norm{ v}_{\Gamma} \sim \norm{ \tilde{v} }_{\Gamma_h} $ and
 $\norm{\nabla v}_{\Gamma} \sim \norm{ \nabla \tilde{v} }_{\Gamma_h} $.
\item For the normal derivates: 
\[ 
\norm{ \nabla \tilde{v} \cdot  \normal_{ \Gamma_h  }  }_{ \Gamma_h} \lesssim h^q  \norm{ \nabla v }_{\Gamma} + 
\norm{ \nabla v \cdot \normal_{ \Gamma  }  }_{ \Gamma }, 
\quad
\norm{ \nabla v \cdot  \normal_{ \Gamma  }  }_{ \Gamma} \lesssim h^q  \norm{ \nabla \tilde{v} }_{\Gamma_h} + 
\norm{ \nabla \tilde{v} \cdot \normal_{ \Gamma_h  }  }_{ \Gamma_h}.
\] 
\item And for tangential derivatives:
\[
\norm{ \nabla_{\Gamma_h} \tilde{v} }_{ \Gamma_h} \lesssim h^q  \norm{ \nabla v }_{\Gamma} + 
\norm{ \nabla_{\Gamma} v  }_{ \Gamma },
\quad 
\norm{ \nabla_{\Gamma} v   }_{ \Gamma} \lesssim h^q  \norm{ \nabla \tilde{v} }_{\Gamma_h} + 
\norm{ \nabla_{\Gamma_h} \tilde{v}   }_{ \Gamma_h}.
\quad 
\] 
\end{enumerate}
\end{lemm}
Finally, we will have to quantify the extent to which a function $f \circ \Phi_h^{-1}(x)$ differs from 
$f(x)$ when compared at the same point $x$. 
Clearly, smoothness assumptions need to be imposed on $f$ if one wants to 
guarantee that the difference is small. 
We remark that part (a) and (b) of the following lemma have been used before in the literature, see e.g.\ \cite[in proof of Lemma 12]{L_GUFEMA_2017}. 
\begin{lemm}[Pullback]\label{lem:pullback}
Let $M \subset \Omega$ and $U,V \subset \Omega$ be open sets that are sufficiently large to contain all line segments between points in $ \Phi_h^{-1}(M)$ and $ M$, respectively between $ \Phi_h(M) $ and $ M$.
\begin{enumerate}[label=(\alph*)] 
\item For $f \in W^{1,\infty}(U) $ we have $\norm{ f \circ \Phi_h^{-1} -f }_M \lesssim  h^{q+1} \sqrt{\abs{M}} \norm{\nabla f}_{\infty,U}$, where $\abs{M}$ denotes the Lebesgue measure of $M$.
\item  For $f \in W^{2,\infty}(U) $ we have 
\[ \norm{ \nabla ( f \circ \Phi_h^{-1} - f ) }_M \lesssim  h^{q+1} \sqrt{\abs{M}} \norm{ f}_{W^{2,\infty}(U) } + h^q \norm{ \nabla{f} }_{M}. \]
\item For $M = \Theta_h(\el)$ with $\el \in \mesh$ and $f \in W^{3,\infty}(V) $ it holds that  
\[ \norm{ f \circ \Phi_h - f }_{ H^2(M) } \lesssim h^{q-1} \sqrt{ \abs{M} } \norm{ f }_{W^{3,\infty}(V)}.
\]
\item Moreover, for $f \in H^2( \Psi(T))$ we have 
\begin{align*}
  \norm{ \partial_{x_{\nu}} \partial_{x_{\mu}} (f \circ \Phi_h) \circ \Phi_h^{-1}  - \partial_{x_{\nu}} \partial_{x_{\mu}} f  }_{ \Psi(T) }
\lesssim
   h^{q} \norm{f}_{ H^2( \Psi(T)) } +  h^{q-1 } \norm{ \nabla f}_{  \Psi(T) }. 
\end{align*}
\end{enumerate}
\end{lemm}

\section{Isoparametric unfitted FEM}\label{section:iso-FEM}
In this section we introduce a stabilized FEM for the solution of (\ref{eq:PDE_bulk})-(\ref{eq:measurements}).
Based on the technique to resolve the interface discussed in Section~\ref{section:IF-iso} we start in Subsection~\ref{ssection:iso-FES} by introducing isoparametric FEM spaces. 
In Subsection~\ref{ssection:method} we define the actual discrete variational formulation, discuss stabilization terms and record some basic properties including stability and bounds on geometric errors.
The final Subsection~\ref{ssection:unfitted-interp} deals with interpolation into unfitted isoparametric FEM spaces. 
The results in this section are independent of the stability properties of the continuous problem, in particular equation (\ref{eq:cond_stab}) is never 
used.

\subsection{Isoparametric finite element spaces}\label{ssection:iso-FES} 
Following \cite{LR17} we define the following isoparametric FEM spaces: 
For $\clubsuit \in \{ \Gamma, 0\}$:   
\begin{equation}\label{eq:def-fes-Theta-Gamma}
\curvedFes{\Theta}^{\clubsuit} := \{ v_h \circ \Theta_h^{-1} \mid  v_h \in \FES{p}{\clubsuit} \}, 
\quad 
\curvedFes{\Phi}^{\clubsuit} := \{ v_h \circ \Phi_h^{-1} \mid  v_h \in \curvedFes{\Theta}^{\clubsuit}  \},
\end{equation}
\begin{equation}\label{eq:def-fes-Psi}
\curvedFes{\Psi}^{\clubsuit} := \{ v_h \circ \Psi^{-1} \mid  v_h \in  \FES{p}{\clubsuit} \}. 
\end{equation}
Note that $\curvedFes{\Theta}^{\Gamma}$ is an isoparametric version of the Cut-Finite element space defined in  (\ref{eq:Cutfes}), while $\curvedFes{\Theta}^{0}$ is based on a standard finite element space on the background mesh with homogeneous Dirichlet boundary conditions on $\partial \Omega$.
Since the mappings $\Psi$ and $\Theta_h$ reduce to the identity away from the interface, it follows that functions in the transformed spaces vanish on $\partial \Omega $ as well. \\ 
Let us also note that some of the tools frequently used in the analysis involving the spaces $\FES{p}{\clubsuit}$ easily carry over to their curved counterparts. 
Indeed, as $w_h \in \curvedFes{\Theta}^{\clubsuit}$ is of the  form $w_h = v_h \circ \Theta_h^{-1}$ with $v_h \in \FES{p}{\clubsuit}  $  and 
we have from \cite[Lemma 3.14]{LR17} the relations
\begin{equation}\label{eq:scaling_p1-geom}
\norm{ \nabla w_h }_{ \Theta_h( \el) } \sim \norm{\nabla v_h}_{ T} \text{ and }  \norm{w_h}_{\Theta_h( \el)} \sim \norm{v_h}_{ T  }
\end{equation}
we can transfer standard inverse inequalities from $v_h$ to $w_h$, e.g.
\begin{align}\label{eq:inverse-ieq-gradient}
\norm{ \nabla w_h }_{ \Theta_h( \el) } \sim \norm{\nabla v_h}_{ T} \lesssim h^{-1} \norm{v_h}_{ T  } \sim h^{-1} \norm{w_h}_{\Theta_h( \el)}. 
\end{align}
Since $\norm{w_h}_{ \Theta_h( \Gamma_{\el} ) } \sim \norm{ v_h }_{  \Gamma_{\el} } $, we also obtain in analogy to the case with piecewise linear geometry, see \cite{HH02}, that
\begin{equation}\label{eq:trace-inequality-dfm}
\norm{w_h}_{ \Theta_h( \Gamma_{\el} ) }^2 \lesssim h^{-1} \norm{w_h}_{ \Theta_h(\el) }^2 + h^{1} \norm{ \nabla w_h}_{ \Theta_h(\el) }^2.   
\end{equation}
We remark that (\ref{eq:scaling_p1-geom}) and (\ref{eq:trace-inequality-dfm}) are valid for general $w = v \circ  \Theta_h^{-1}$ with $v \in H^1(T)$.
\subsection{ Stabilized method for unique continuation   }\label{ssection:method} 
Let us start by defining the bilinear forms usually employed for the unfitted discretization 
of problem (\ref{eq:PDE_bulk})-(\ref{eq:normal_jump_IF}) when homogeneous Dirichlet boundary conditions on $\partial \Omega$ are given. 
In the bulk we define 
\begin{equation}\label{eq:bfi_int}
a(u,v) := \sum\limits_{i=1,2} (\mu_i \nabla u_i, \nabla v_i)_{ \Omega_i } - (\rho_i u_i,  v_i)_{ \Omega_i }, 
\end{equation}
and analogously on the discrete geometry
\begin{equation}\label{eq:bfi_h_int}
a_h(u,v) := \sum\limits_{i=1,2} (\mu_i \nabla u_i, \nabla v_i)_{ \Omega_{i,h} } - (\rho_i u_i,  v_i)_{ \Omega_{i,h} }. 
\end{equation}
On the interface we define  
\begin{equation}\label{eq:Nitsche-terms-def}
N^{c}(u,v) := \int\limits_{\Gamma} \avg{-\mu \nabla v} \cdot \normal_{\Gamma} \jump{ u } \; \dS, \quad \avg{ \mu \nabla v } \cdot \normal_{\Gamma}  := \kappa_1 \mu_1 \nabla v_1 \cdot n_{\Gamma} + \kappa_2 \mu_2 \nabla v_2 \cdot n_{\Gamma}   
\end{equation}
for convex weights $\kappa_1 + \kappa_2 = 1$ in the numerical flux.
Note that $ N^{c}(u,v) = 0$ for the exact solution $u$ of (\ref{eq:PDE_bulk})-(\ref{eq:normal_jump_IF}) holds. 
On the discrete geometry we define $N^{c}_h$ analogously where the integration is now over $\Gamma_h$ and the exact normal is replaced by $\normal_{\Gamma_h}$.
We combine bulk and interface terms into
\begin{equation}\label{def:A-bil}
A(u,v) := a(u,v) + N^{c}(u,v), \quad 
A_h(u,v) := a_h(u,v) + N_h^{c}(u,v).
\end{equation}
Notice that $N_h^{c}(u,v)$ is the only term on the interface remaining in a standard unfitted Nitsche formulation \cite{HH02} when $v$ is continuous across $\Gamma$. \par 
Next we proceed to the definition of the linear forms. 
Let $f_{i} = \calL u_i \in W^{1,\infty}( \Omega_{i} )$ be the exact right hand side of (\ref{eq:PDE_bulk}). 
In applications exact data is usually not available. 
We merely have access to perturbed data 
\[ \tilde{f}_i := f_i + \delta f_i, \quad  \delta f_i \in L^2(\Omega_i). \]
For implementation, we have to extend the data $\tilde{f}_i $ given on $\Omega_i$ to 
$\Omega_{i,h}$. 
To this end, we use the Sobolev extension operator from (\ref{eq:def-Sobolev-Extension}). 
Hence,
\[ \Ext_i \tilde{f}_i = \Ext_i f_i  +  \Ext_i \delta f_i := f_{i,h} + \delta f_{i,h}  
\]
is well defined on $\BrokenOmegah$.
Note that
\[
 \norm{  f_{i,h} }_{ W^{1,\infty}( \Omega_{i,h} ) }  \lesssim \norm{  f_{i} }_{ W^{1,\infty}( \Omega_{i} )  }, \quad 
\norm{  \delta f_{i,h} }_{ \Omega_{i,h} }  \lesssim \norm{  \delta f_i }_{ \Omega_i } 
\]
holds by continuity of the extension.
We write $ f_{h} := ( f_{1,h} , f_{2,h} )$ and the same for  $\delta f_h$ and $\Ext \tilde{f}$. 
Then we define the linear forms
\[
\ell(w) : =  (f ,w )_{ \BrokenOmega }, \; 
\ell_h(w_h) : =  ( f_{h},w_{h})_{ \BrokenOmegah  }, \; 
\tilde{\ell}_h(w_h) : =  ( \Ext \tilde{f} ,w_{h} )_{ \BrokenOmegah  }.
\]
We proceed analogously for the data domain.
Recall that $u_{\omega} := u|_{\omega}$ for $u$ being the exact solution of (\ref{eq:PDE_bulk})-(\ref{eq:normal_jump_IF}). We assume perturbed data 
\[
\tilde{u}_{\omega} := u_{\omega} + \delta u_{\omega}, \quad \delta u_{\omega} \in L^2(\omega)
\]
is given.
The strength of the entire data perturbation is measured in the norm 
\begin{equation}\label{eq:delta-def}
\delta := \norm{ \delta u_{\omega} }_{ \omega } + \norm{ \delta f }_{\BrokenOmega}. 
\end{equation}
To define a stable method for the unique continuation problem, we follow 
the approach introduced in \cite{B13}. 
The idea is to formulate the ill-posed problem as a discrete optimization problem consisting 
of a data fidelity term: $\norm{u_h - \tilde{u}_{\omega}  }_{\omega}$, a term to enforce the PDE constraint: $A_h(u_h,z_h) - \tilde{\ell}_h(z_h) $, where $z_h$ represents a Lagrange multiplier, 
and suitable stabilization terms to guarantee unique solvability and incorporate a priori 
knowledge of the solution.
The stabilized FEM is then obtained from the first order optimality condition. 
We refer the interested reader to Remark~\ref{rem:Lagrangian} for a more detailed derivation and proceed now directly to the discrete variational formulation defined as follows.
Find $(u_h,z_h) \in \curvedFes{\Theta}^{\Gamma} \times  \curvedFes{\Theta}^{0}  $ such that 
\begin{align}
& B_h[(u_h,z_h),(v_h,w_h)] = ( \tilde{u}_{\omega},v_h)_{\omega} + \tilde{\ell}_{h}(w_h) + \sum\limits_{i=1,2} \sum\limits_{ \el \in \mesh^{i} } h^2 ( \Ext_i \tilde{f}_i , \calL v_{h,i} )_{  \Theta_h( \el_i) } \nonumber \\ 
	&= (u_{\omega} ,v_h )_{\omega} + \ell_h(w_h) +   \sum\limits_{i=1,2} \sum\limits_{ \el \in \mesh^{i} } h^2 (  f_{i,h}, \calL v_{h,i} )_{  \Theta_h( \el_i) }  + g(v_h,w_h)  \label{eq:discr_stab_var_form}
\end{align}
with 
\begin{align}
& g(v_h,w_h) := (\delta u_{\omega}  ,v_h )_{\omega} + (\delta f_h, w_h )_{\BrokenOmegah}  +   \sum\limits_{i=1,2} \sum\limits_{ \el \in \mesh^{i} } h^2 (  \delta f_{i,h}, \calL v_{h,i} )_{  \Theta_h( \el_i) }, \nonumber \\ 
& B_h[(u_h,z_h),(v_h,w_h)] = A_h(v_h,z_h) + s_h(u_h,v_h) + (u_h,v_h)_{ \omega } \nonumber \\ 
	&  \phantom{ B_h[(u_h,z_h),(v_h,w_h)]   }  + A_h(u_h,w_h) - s_h^{\ast}(z_h,w_h). \label{eq:stab_bfi}
\end{align}
We remark that as $z_h$ is continuous across the interface, the variables $z_{h,i}$ which appear in the definition of $A_h(v_h,z_h)$ (see e.g.\ (\ref{eq:bfi_h_int}))  are simply the restriction of $z_h$ to the respective subdomains.  
It remains to define the primal $s_h(\cdot,\cdot)$ and dual $s_h^{\ast}(\cdot,\cdot)$ stabilization terms.  
We commence by introducing the constituents of the former. 
A fairly standard component of methods based on the framework in \cite{B13} are Galerkin-least squares:
\[
	\StabGLS{u_h}{v_h} := \sum\limits_{i=1,2} \sum\limits_{ \el \in \mesh^{i} } h^2 ( \calL u_{h,i} , \calL v_{h,i} )_{  \Theta_h( \el_i) }
\]
and continuous interior penalty terms:
\begin{equation}\label{eq:J-CIP}
	\StabCIP{u_h}{v_h} := \sum\limits_{i=1,2} \sum\limits_{ F\in  \facets^i } h \! \! \int\limits_{\Theta_h(F) } \! \! \mu_i \jump{ \nabla u_{h,i} }_{ \Theta_h(F) } \cdot \normal \jump{ \nabla v_{h,i} }_{ \Theta_h(F) } \cdot \normal  \; \dS_{ \Theta_h(F) }.
\end{equation}
These terms are needed to control the $H^{-1}$-norm of the PDE residual, see the estimate of the terms $\mathrm{I}_1$ and $\mathrm{I}_2$ in Lemma~\ref{lem:weak-conv}.  
Here, $\facets^i := \{ F = T_a \cap T_b \mid T_a,T_b \in \mesh^i, T_a \neq T_b, F \nsubseteq \partial \Omega   \}$ is used to denote the facets of the active mesh not lying on $\partial \Omega$ (see Figure~\ref{fig:Stab-sketch-mod} for a sketch), $\jump{v}_{F} := v|_{T_a} -  v|_{T_b}$ represents the jump over a facet $F = T_a \cap T_b$ and $n$ denotes the outward pointing normal vector of $T_a$.
We remark that it would be sufficient to carry out the integration in (\ref{eq:J-CIP}) merely 
over $\Theta_h(F \cap \el_i)$, i.e.\ the part of the facet lying in $\Omega_{i,h}$. 
However, the formulation in (\ref{eq:J-CIP}) is more convenient for implementation as 
integrals over cut facets are avoided.  \par 
\begin{figure}[htbp]
\centering
\includegraphics[scale=0.35]{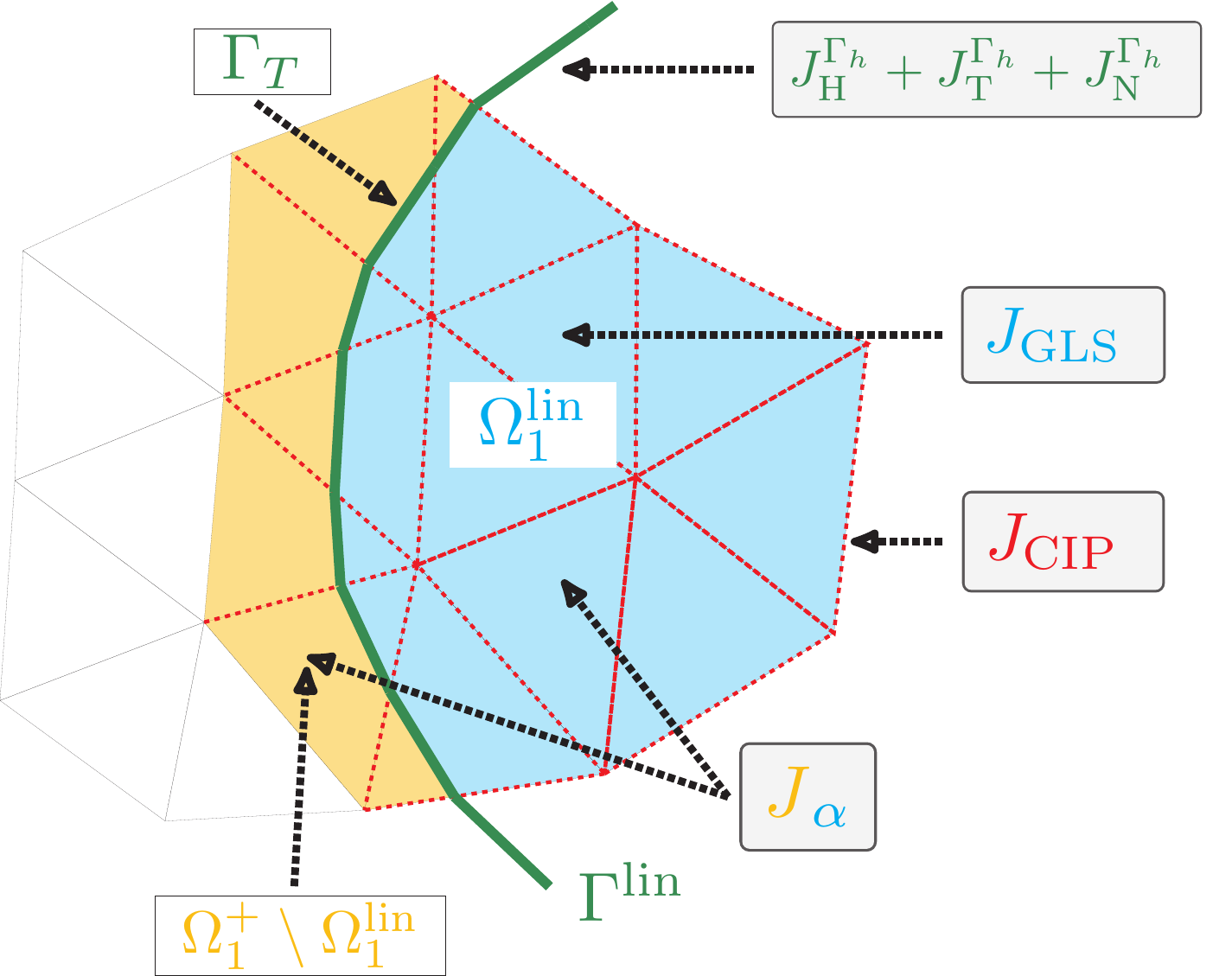}
\caption{Domain of definition of stabilization terms on the piecewise linear reference geometry ($q=1$). Facets in $\facets^1$ are indicated by red dashed lines.  } 
\label{fig:Stab-sketch-mod}
\end{figure}
Next we introduce stabilization terms on the discrete interface. 
The first two terms are strong analogs of the classical Nitsche terms commonly found 
in unfitted methods for interface problems \cite{HH02,LR17} and are required to control the terms $\mathrm{I}_3$ and $\mathrm{I}_6$ in Lemma~\ref{lem:weak-conv}.  
For $\bar{\mu} = (\mu_1+\mu_2)/2$ we define
\begin{align*}
	\StabNitsche{u_h}{v_h} &:= \frac{ \bar{\mu} }{h}  \int\limits_{ \Gamma_h } \jump{ u_h }_{ \Gamma_h } \jump{ v_h }_{ \Gamma_h } \dS_{\Gamma_h}, \\ 
 \StabNablaNormal{u_h}{v_h} &:=  h \int\limits_{ \Gamma_h } \jump{ \mu \nabla u_h }_{ \Gamma_h } \cdot \normal_{\Gamma_h} \jump{  \mu \nabla v_h  }_{ \Gamma_h } \cdot \normal_{\Gamma_h}  \; \dS_{\Gamma_h}.
\end{align*}
Additionally, we will stabilize the jump of the tangential gradient 
\[
	\StabNablaGamma{u_h}{v_h} :=  h \bar{\mu} \int\limits_{ \Gamma_h } \jump{ \nabla_{\Gamma_h } u_h }_{ \Gamma_h } \jump{ \nabla_{\Gamma_h } v_h }_{ \Gamma_h } \dS_{\Gamma_h}.
\]
The need for this penalty arises due to the nonconformity $\curvedFes{\Theta}^{\Gamma} \nsubseteq  H^1(\Omega)$ as we will see later in Lemma~\ref{lem:corr_IF_jumps}.
We combine the stabilization on the interface into 
\begin{equation}\label{eq:IF-stab-comb}
	\StabIFall{\cdot}{\cdot} := \gamma_{\mathrm{IF}} \left( \StabNitsche{\cdot}{\cdot} + \StabNablaNormal{\cdot}{\cdot} +  \StabNablaGamma{\cdot}{\cdot} \right), 
\end{equation}
where $\gamma_{\mathrm{IF}}$ is some positive parameter which we will set equal to one for the analysis. 
Whereas the previous terms would be consistent on the exact geometry for sufficiently smooth solutions of (\ref{eq:PDE_bulk})-(\ref{eq:normal_jump_IF}), 
we now introduce a weakly inconsistent Tikhonov term:
\begin{equation}\label{eq:def-Tikh}
\StabTikh{u_h}{v_h} := h^{2q} \sum\limits_{i=1,2}  \alpha_1 ( u_{h,i}, v_{h,i} )_{ \Omega_{i,h}^+ } + \alpha_2 (\nabla u_{h,i}, \nabla v_{h,i} )_{ \Omega_{i,h}^+ } , 
\end{equation}
for some $\alpha_i >0, i=1,2$.
A weaker version of this stabilization only involving the $L^2$-norm has 
been utilized previously in the context of unique continuation problems in e.g.\ \cite{BDE21,BP22}. 
Here we require additionally control over the gradient to compensate for geometric errors introduced 
by imperfect approximation of the interface, i.e.\ $\dist(\Gamma, \Gamma_h) \lesssim \mathcal{O}(h^{q+1})$.
In the analysis this becomes apparent in the final estimate of the proof of Lemma~\ref{lem:corr_IF_jumps} whereas we refer to Figure~\ref{fig:non-convex-diffusion-Stab} for 
a numerical illustration.
We remark that it is possible to localize these geometric errors and consequently the corresponding stabilization to a band of elements around the interface by exploiting the locality of $\Theta_h$. 
As this measure did not lead to a significant improvement of the numerical results, we decided 
to present the technically simpler global stabilization in this paper. \\
We combine the introduced terms into the primal stabilization: 
\begin{align}
s_h(u_h,v_h) &:= \tilde{s}_h(u_h,v_h)  + \StabGLS{u_h}{v_h}, \nonumber \\ 
\tilde{s}_h(u_h,v_h) &:= \StabCIP{u_h}{v_h}  + \StabTikh{u_h}{v_h} +
\StabIFall{u_h}{v_h}.  \label{eq:primal_stab_withoutGLS}
\end{align}
The dual stabilization is simply given by
\begin{equation}\label{eq:dual_stab}
s^{\ast}_h(z_h,w_h) := (\mu \nabla z_h, \nabla w_h)_{ \Omega }.
\end{equation}
The dual stabilization needs to be sufficiently strong to allow for certain continuity estimates, see e.g.\ Lemma~\ref{lem:Ah-continuity}. Using 
integration by parts one could think about replacing the stabilization of the gradient by a combination of $\StabCIP{\cdot}{\cdot}, \StabGLS{\cdot}{\cdot}$ and 
a penalty of the normal derivative on $\Gamma_h$ and $\partial \Omega$. Note however that the presence of $\StabCIP{\cdot}{\cdot}$ would introduce additional couplings in the matrix. 
Hence, we prefer to use the simpler choice (\ref{eq:dual_stab}) in this article.

\begin{rema}[Derivation of variational formulation]\label{rem:Lagrangian}
Let us consider the following Lagrangian:
\begin{align*}
L(u_h,z_h) := \frac{1}{2} \norm{u_h - \tilde{u}_{\omega}  }_{\omega}^2 +  A_h(u_h,z_h) - \tilde{\ell}_h(z_h) + \frac{1}{2} \tilde{s}_h(u_h,u_h) 
+ \frac{1}{2} \StabGLS{u_h - \Ext u }{ u_h - \Ext u } -\frac{1}{2} s^{\ast}_h(z_h,z_h),
\end{align*}
where $ \Ext u  := ( \Ext_1 u_1  , \Ext_2 u_2  )$ is a Sobolev extension of the exact solution of (\ref{eq:PDE_bulk})-(\ref{eq:normal_jump_IF}).  
To search for critical points of $L$ we take the first order optimality conditions to arrive at 
\begin{align*}
0 &= \frac{\partial L}{\partial z_h}(w_h) = A_h(u_h,w_h) - \tilde{\ell}_h(w_h) - s^{\ast}_h(z_h,w_h), \quad \forall w_h \in  \curvedFes{\Theta}^{0}, \\
	0 & = \frac{\partial L}{\partial u_h}(v_h) =  (u_h - \tilde{u}_{\omega},v_h)_{ \omega }  + A_h(v_h,z_h) + s_h(u_h,v_h) 
            - \StabGLS{\Ext u }{ v_h },  \quad \forall v_h \in \curvedFes{\Theta}^{\Gamma}. 
\end{align*}
Adding these equations leads to the variational formulation (\ref{eq:discr_stab_var_form}) upon making the replacement 	
\[
\StabGLS{\Ext u }{ v_h } =  \sum\limits_{i=1,2} \sum\limits_{ \el \in \mesh^{i} } h^2 ( \calL \Ext_i u_i , \calL v_{h,i} )_{  \Theta_h( \el_i) }  \rightarrow    \sum\limits_{i=1,2} \sum\limits_{ \el \in \mesh^{i} } h^2 ( \Ext_i \tilde{f}_i , \calL v_{h,i} )_{  \Theta_h( \el_i) }. 
\]
Notice that in practice the exact solution $u_i$ is of course unknown so we cannot use $\calL \Ext_i u_i$ directly as data. 
However, restricted to $\Omega_i$ we have $(\calL \Ext_i u_i)|_{\Omega_i} = \calL u_i = f_i$ as $u_i$ fulfills (\ref{eq:PDE_bulk}). 
Hence, in the noise-free case one would replace $\calL \Ext_i u_i$ by the data extension $ \Ext_i f_i = f_{i,h} $. 
This will later be justified in Lemma~\ref{lem:GLS} which shows that the corresponding consistency errors converge at optimal order. 
In the presence of noise we then simply replace $ \Ext_i f_i$ by $\Ext_i \tilde{f}_i$, which leads to the formulation (\ref{eq:discr_stab_var_form}). 
\end{rema}
\begin{rema}[Discretization space for dual variable]
For the unique continuation problem the dual variable $z_h$ simply approximates\footnote{We will quantify this statement later in Proposition~\ref{prop:conv_discr_error_tnorm}.} zero and is therefore taken from the space $\curvedFes{\Theta}^{0}$ which solely contains functions that are continuous over the interface.
However, for control problems as considered in \cite{BFMO23} the dual variable approximates the control. Following \cite{GLL90}, the control variable of minimal $L^2$-norm solves the homogeneous adjoint problem which inherits the discontinuous coefficients of the original problem. 
Consequently, the latter then has to be discretized using a subspace of $\curvedFes{\Theta}^{\Gamma}$. 
We have extended our error analysis to cover this case and arrived at the same final results. 
Here we present the slightly simpler version using a continuous variable $z_h$, yet we point out that the mentioned extension could pave the way to include interfaces in the control problem considered in \cite{BFMO23}. 
\end{rema}

\subsubsection{Norms} 
For the error analysis we introduce the norms 
\begin{equation}\label{eq:stab_norms}
\norm{v_h}_{s_h} := s_h(v_h,v_h)^{1/2}, \; v_h \in  \curvedFes{\Theta}^{\Gamma}, \quad  
\norm{w_h}_{s^{\ast}_h} := s_h^{\ast}(w_h,w_h)^{1/2}, \; w_h \in  \curvedFes{\Theta}^{0}.   
\end{equation}
For $\alpha_1 >0$ it is clear that $\norm{\cdot}_{s_h}$ is indeed a norm on $\curvedFes{\Theta}^{\Gamma}$.
On the product space $\curvedFes{\Theta}^{\Gamma} \times \curvedFes{\Theta}^{0} $ we define the norm
\begin{equation}\label{eq:tnorm}
\tnorm{ (v_h,w_h) }^2  := \norm{v_h}_{s_h}^2 + \norm{v_h}_{\omega}^2 + \norm{ w_h }_{s_h^{\ast}}^2. 
\end{equation}
Note in passing that combining Friedrich's inequality with continuity of the data extension yields 
\begin{align}
 g(v_h,w_h) & \lesssim \norm{\delta u_{\omega}  }_{\omega} \norm{v_h }_{\omega} + \norm{ \delta f_h   }_{\BrokenOmegah} \norm{w_h }_{\Omega}   \nonumber \\
	& + \StabGLS{ v_{h} }{  v_{h} }^{1/2} h \norm{ \delta f_h }_{\BrokenOmegah} \lesssim  \delta  \tnorm{ (v_h,w_h) } \label{eq:perturbation-bound}. 
\end{align}
Moreover, we will frequently make use of the common notation 
\begin{equation}\label{eq:snorms-IF}
\norm{v_h}_{\frac{1}{2},h,\Gamma_h}^2 := h^{-1} \norm{v_h}_{\Gamma_h}^2,
\qquad
\norm{v_h}_{- \frac{1}{2},h,\Gamma_h}^2 := h \norm{v_h}_{\Gamma_h}^2.
\end{equation}
For measuring smoothness of the exact solution let us also define  
\begin{equation}\label{eq:def_q_norm}
\norm{u}_{q} := \norm{u}_{ H^{q+1}(\BrokenOmega) } + \norm{u}_{ W^{3,\infty}(\BrokenOmega) } .
\end{equation}
By Sobolev embedding theorems the second term in (\ref{eq:def_q_norm}) involving the infinity norm can be omitted for $q \geq 4$ in dimensions $d \leq 3$.

\subsubsection{Stability, consistency and continuity}
We record some elementary properties of the introduced variational forms. 
\begin{lemm}[Consistency]\label{lem:consistency_A}
Let $u$ be a solution of (\ref{eq:PDE_bulk})-(\ref{eq:normal_jump_IF}). Then 
\[ A(u,v_h) = \ell(v_h), \; \text{ for all } v_h \in \curvedFes{\Phi_h}^{0}.  \]
\end{lemm}
\begin{proof} 
Follows analogously to \cite[Lemma 5.7]{LR17}. 
\end{proof}
Further note that 
\[
B_h[(u_h,z_h),(u_h,-z_h)]
= \norm{ u_h }_{s_h}^2 + \norm{ u_h}_{ \omega}^2 + \norm{ z_h }^2_{ s_h^{\ast} } 
= \tnorm{ (u_h, z_h) } \tnorm{ (u_h, -z_h) }. 
\]
Hence, 
\begin{equation}\label{eq:inf-sup}
\sup_{ (v_h,w_h) \in \curvedFes{\Theta}^{\Gamma} \times  \curvedFes{\Theta}^{0}  } \frac{ B_h[ (u_h,z_h), (v_h , w_h) ] }{ \tnorm{ (v_h, w_h) }  } \gtrsim  \tnorm{ (u_h, z_h) }. 
\end{equation}
The bilinear form $A_h$ is also continuous in the defined norms.
\begin{lemm}[Continuity]\label{lem:Ah-continuity}
For $v \in H^1(\BrokenOmegah) +  \curvedFes{\Theta}^{\Gamma} $ and $w_h \in \curvedFes{\Theta}^{0} $ we have  
\begin{align*}
	A_h(v,w_h) \lesssim \norm{w_h}_{s^{\ast}_h} \left( \norm{ \sqrt{ \mu } \nabla v  }_{  \BrokenOmegah } + \norm{v}_{\BrokenOmegah }  
+ \norm{ \jump{ v } }_{\frac{1}{2},h,\Gamma_h } \right). 
\end{align*}
\end{lemm}
\begin{proof}
	Using (\ref{eq:trace-inequality-dfm}), $w_h \in \curvedFes{\Theta}^{0}$ (in particular it is univalued and smooth inside elements) and finally (\ref{eq:inverse-ieq-gradient}) we obtain\footnote{Here, we denoted by $w_{h,i}$ the traces coming from different sides of the interface which may be different even for $w_h \in \curvedFes{\Theta}^{0}$, e.g. when the interface coincides with inter element boundaries.  }
\begin{align*}
	\sum\limits_{i=1,2} \sum\limits_{ \el \in  \Gamma_{\el} } h \norm{ \nabla w_{h,i}  }_{ \Theta_h(\Gamma_{\el})}^2 \lesssim \sum\limits_{\el \in \mesh} \left( \norm{\nabla w_h}^2_{\Theta_h(\el) } + h^2  \norm{\nabla w_h}^2_{ H^1(\Theta_h(\el)) }  \right) \lesssim \norm{w_h}_{s^{\ast}_h}^2. 
\end{align*}
Hence, the claim follows in view of 
\[
N_h^{c}(v, w_h) \lesssim \left( \sum\limits_{i=1,2} \sum\limits_{ \el \in  \Gamma_{\el} } h \norm{ \nabla w_{h,i}  }_{ \Theta_h(\Gamma_{\el})}^2 \right)^{1/2}  \norm{ \jump{ v } }_{\frac{1}{2},h,\Gamma_h }. 
\]
\end{proof}

\subsubsection{Geometry errors for bilinear and linear forms}\label{ssection:geom-bfi} 
We recall \cite[Lemma 5.10]{LR17} in which the geometric errors for the standard unfitted Nitsche formulation have been analyzed.
\begin{lemm}\label{lem:bil-PDE-geom-error-Christoph}
For $u \in  H^1(\BrokenOmega )$ and $w_h  \in \curvedFes{\Theta}^{0}$ set $\tilde{u}:= u \circ \Phi_h$ and $ \tilde{w}_h = w_h \circ \Phi_h^{-1}$. 
\begin{enumerate}[label=(\alph*)]
\item The geometrical consistency error for the bilinear form in the bulk is bounded by
\begin{align*}
	\abs{ a(u, \tilde{w}_h) - a_h(\tilde{u}, w_h ) } \lesssim h^q  \sum\limits_{i=1,2} \norm{ \tilde{u}_i }_{ H^1( \Omega_{i,h} ) }  \norm{ w_{h} }_{ H^1( \Omega ) } \lesssim h^q \norm{u}_{H^1(\BrokenOmega)}  \norm{w_h}_{ s_h^{\ast} }. 
\end{align*}
\item And for the linear form  
\begin{align*}
\abs{ \ell(\tilde{w}_h) - \ell_h( w_h )  } \lesssim h^q \norm{f}_{ W^{1,\infty}( \BrokenOmega ) }  \norm{ w_h }_{ \Omega }  \lesssim h^q \norm{u}_{ W^{3,\infty}( \BrokenOmega ) } \norm{w_h}_{ s_h^{\ast} }, 
\end{align*}
where the second inequality assumes that $u$ solves (\ref{eq:PDE_bulk}). 
\end{enumerate}
\end{lemm}

\subsection{Unfitted interpolation}\label{ssection:unfitted-interp}

To derive error estimates in the next Section~\ref{section:error-analysis}, we need to be able to compare the exact solution of (\ref{eq:PDE_bulk})-(\ref{eq:normal_jump_IF}) with an approximant in our unfitted curved finite element space.
Since geometry data in applications often has moderate resolution, we focus here on the case of a possibly underresolved geometry. 
Hence, we introduce 
\begin{assum}
Let $ 1 \leq q \leq p$.
\end{assum}
Under this assumption it follows thanks to \cite[Theorem 3.12]{LR17} that $\Psi$ induces a family of (curved) finite elements in the sense of Bernardi \cite{B89} to which the interpolation theory developed in the latter paper can be applied. 
We then obtain from \cite[Corollary 4.1]{B89} the existence of an interpolation operator $\InterpBackgroundMesh:H^{\ell}(\Omega) \rightarrow \curvedFes{\Psi}$
with optimal approximation properties. For $v \in H^{\ell}(\Omega)$ it holds that 
\begin{equation}\label{eq:InterpBackgroundMesh}
\left(  \sum\limits_{ T \in \mesh }  \abs{ v - \InterpBackgroundMesh v  }_{  H^{m} (\Psi(T))  }^2 \right)^{1/2}  
	\lesssim h^{\ell-m} \norm{v}_{ H^{\ell}(\Omega)  } 
\end{equation}
for $ 0 \leq m \leq  \ell \leq q+1$.
To define an unfitted interpolation operator, we use standard techniques, see e.g. \cite{HH02,R08,LR17}.
We define 
\begin{equation}\label{eq:unfitted_interp}
\InterpUnfitted: L^2(\BrokenOmega) \rightarrow \curvedFes{\Psi}^{\Gamma}, \quad
\InterpUnfitted u = ( R_1^+ \InterpBackgroundMesh  \Ext_1 R_1 u, R_2^+ \InterpBackgroundMesh \Ext_2 R_2 u ).
\end{equation}
Sometimes we will refer to the components by $(\InterpUnfitted u)_i := R_i^+ \InterpBackgroundMesh \Ext_i R_i u$. Since  $(\InterpUnfitted u)_i$ will in the following only appear on domains which are part of of the active mesh $\mesh^i$ on which the restriction $R_i^+$ amounts to the identity, we will drop $R_i^+$  from now on to ease notation. 
We denote by $\InterpBackgroundMeshZ:H^{1}_0(\Omega) \rightarrow \curvedFes{\Psi}^{0}$ 
the interpolation operator into the standard (curved) spaces with homogeneous Dirichlet boundary conditions on $\partial \Omega$, see \cite[Theorem 5.1]{B89}. 
We then obtain the usual interpolation results.
\begin{lemm}[Unfitted interpolation]\label{lem:approx_unfitted}
Let $u \in H^{\ell}(\BrokenOmega) $ and $ 1 \leq m \leq  \ell \leq q+1$.  
\begin{enumerate}[label=(\alph*)]
\item We have: 
\[ \norm{  (u -  \InterpUnfitted u) \circ \Phi_h }_{ L^2( \BrokenOmegah ) } \sim \norm{ u -  \InterpUnfitted u }_{ L^2( \BrokenOmega ) } \lesssim   h^{\ell } \norm{ u }_{ H^{\ell }(\BrokenOmega) }. 
\] 
\item We have: 
\[ \norm{\nabla ( u \circ \Phi_h -  \InterpUnfitted u \circ \Phi_h ) }_{ L^2( \BrokenOmegah ) } \sim \norm{\nabla ( u -  \InterpUnfitted u ) }_{ L^2( \BrokenOmega ) } \lesssim   h^{\ell -1 } \norm{ u }_{ H^{\ell }(\BrokenOmega) }.
		\]
\item For $i=1,2$ it holds 
\[ \norm{  (\Ext_i R_i u - \InterpBackgroundMesh \Ext_i R_i u) \circ \Phi_h }_{\frac{1}{2},h,\Gamma_h } \lesssim  h^{\ell - 1 } \norm{ u }_{ H^{\ell }(\Omega_i) }. \]
\item For $i=1,2$ and $m \geq 2$ it holds 
\[
\norm{  \nabla (  \Ext_i R_i u - \InterpBackgroundMesh \Ext_i R_i u) \circ \Phi_h }_{-\frac{1}{2},h,\Gamma_h } \lesssim  h^{\ell -1 } \norm{ u }_{ H^{\ell }(\Omega_i ) } .
\]
\item For $i=1,2$ and $m \geq 2$ we have 
\[
\sum\limits_{ \el \in \mesh^{i} } h^2 \norm{ ( \Ext_i R_i u -  \InterpBackgroundMesh \Ext_i R_i u ) \circ \Phi_h }_{H^2( \Theta_h(T)) }^2  \lesssim  h^{2(\ell -1) } \norm{ u }_{ H^{\ell }(\Omega_i ) }^2 .
\]
\end{enumerate}
\end{lemm}
\begin{proof}
The proofs of statements (a)-(d) are standard, see \cite{HH02,R08} for the case 
of a piecewise linear reference geometry and \cite{LR17,LR18} for higher order 
involving the mapping.
We only show statement (e), which follows by applying Lemma~\ref{lem:norm-equiv} and the interpolation result (\ref{eq:InterpBackgroundMesh}): 
\begin{align*}
& \sum\limits_{ \el \in \mesh^{i} } h^2 \norm{ ( \Ext_i R_i u -  \InterpBackgroundMesh \Ext_i R_i u ) \circ \Phi_h }_{H^2( \Theta_h(T)) }^2  
 \lesssim \sum\limits_{ \el \in \mesh^{i} } h^2 \norm{  \Ext_i R_i u -  \InterpBackgroundMesh \Ext_i R_i u  }_{H^2( \Psi_h(T)) }^2 \\
& \lesssim h^{2(\ell -1) } \norm{ u }_{ H^{\ell }(\Omega_i ) }^2 \; \text{ for } 2 \leq m \leq \ell \leq q+1.
\end{align*}
\end{proof}
As a consequence of the interpolation bounds, the continuity result of Lemma~\ref{lem:Ah-continuity},
and the geometry error estimate from Lemma~\ref{lem:bil-PDE-geom-error-Christoph} we obtain:
\begin{lemm}\label{lem:lem:bil-PDE-geom-error-interp}
Let $u \in  H^{q+1}(\BrokenOmega )$ solve (\ref{eq:PDE_bulk})-(\ref{eq:normal_jump_IF}) and 
$w_h  \in \curvedFes{\Theta}^{0}$. 
Then
\[
\abs{ A(u, w_h \circ \Phi_h^{-1}) - A_h( \InterpUnfitted u \circ \Phi_h , w_h ) } \lesssim h^{q} \norm{u}_{H^{q+1}(\BrokenOmega)} \norm{w_h}_{ s_h^{\ast} }. 
\]
\end{lemm}
\begin{proof}
Since $\Phi_h( \Omega_{i,h} ) = \Omega_i$ we have $ (\Ext_i R_i u) \circ \Phi_h |_{ \Omega_{i,h} } = (R_i u) \circ \Phi_h |_{ \Omega_{i,h} }$ and similarly on $\Gamma_h$. Using this along with $N^{c}(u, w_h \circ \Phi_h^{-1}) = N_h^{c}( u \circ \Phi_h , w_h ) = 0$ 
	for $u$ solving (\ref{eq:jump_IF}) yields
\begin{align*}
& \abs{ A(u, w_h \circ \Phi_h^{-1}) - A_h( \InterpUnfitted u \circ \Phi_h , w_h ) }  \\
& \leq \abs{ a(u, w_h \circ \Phi_h^{-1}) - a_h( u \circ \Phi_h , w_h ) }
  + \abs{ A_h( \Ext u \circ \Phi_h -   \InterpUnfitted u \circ \Phi_h , w_h ) } \\ 
 &  := \mathrm{I}_1 + \mathrm{I}_2. 
\end{align*}
By Lemma~\ref{lem:bil-PDE-geom-error-Christoph} we have $\mathrm{I}_1 \lesssim h^q \norm{u}_{ H^1(\BrokenOmega ) } \norm{w_h}_{ s_h^{\ast} }   $. 
To estimate $\mathrm{I}_2$ we use continuity of $A_h(\cdot,\cdot)$, see Lemma~\ref{lem:Ah-continuity}, and the interpolation estimates from Lemma~\ref{lem:approx_unfitted}:  
\begin{align*}
	\mathrm{I}_2 & \lesssim \norm{w_h}_{ s_h^{\ast} } \Big( \sum\limits_{i=1,2} \norm{    (\Ext_i R_i u  - \InterpBackgroundMesh  \Ext_i R_i u) \circ \Phi_h ) }_{ H^1( \Omega_{i,h } ) } \\
	& + \norm{  (\Ext_i R_i u - \InterpBackgroundMesh \Ext_i R_i u) \circ \Phi_h }_{\frac{1}{2},h,\Gamma_h }  \Big) \lesssim h^q \norm{w_h}_{ s_h^{\ast} } \norm{ u }_{ H^{q+1}(\BrokenOmega ) }.
\end{align*} 
\end{proof}
We now derive several results related to the consistency of the stabilization.
By this we mean that if $u$ solves (\ref{eq:PDE_bulk})-(\ref{eq:normal_jump_IF}),
then inserting the interpolant $\InterpUnfitted u \circ \Phi_h $ into the stabilization terms 
should at most lead to an error of $\mathcal{O}(h^q)$. 
The first lemma deals with the terms on the interface.
\begin{lemm}[Consistency of interface stabilization]\label{lem:IF-stab-interp}
Let $u \in H^{q+1}(\BrokenOmega)$ solve (\ref{eq:PDE_bulk})-(\ref{eq:normal_jump_IF}). Then 
\begin{enumerate}[label=(\alph*)]
\item $ \StabNablaNormal{\InterpUnfitted u \circ \Phi_h }{ \InterpUnfitted u \circ \Phi_h }^{1/2} \lesssim h^q \norm{u}_{ H^{q+1}(\BrokenOmega) }. $  
\item $ \StabNablaGamma{\InterpUnfitted u \circ \Phi_h }{ \InterpUnfitted u \circ \Phi_h }^{1/2} \lesssim h^q \norm{u}_{ H^{q+1}(\BrokenOmega) }. $  
\item $ \StabNitsche{\InterpUnfitted u \circ \Phi_h }{ \InterpUnfitted u \circ \Phi_h }^{1/2} \lesssim h^q \norm{u}_{ H^{q+1}(\BrokenOmega) }. $ 
\end{enumerate}
\end{lemm}
\begin{proof} \hfill
\begin{enumerate}[label=(\alph*)]
\item From Lemma~\ref{lem:Deriv-IF} (b) and $ \jump{ \mu \nabla u }_{\Gamma} \cdot \normal_{\Gamma}  =  \jump{ \mu \nabla \Ext u }_{\Gamma} \cdot \normal_{\Gamma} = 0$ we obtain 
\begin{align*}
& \StabNablaNormal{ \InterpUnfitted u \circ \Phi_h }{ \InterpUnfitted u \circ \Phi_h } = h \norm{ \jump{ \mu \nabla  \InterpUnfitted u \circ \Phi_h } \cdot \normal_{ \Gamma_h }   }_{ \Gamma_h }^2   \\ 
	& \lesssim h^{2(q+\frac{1}{2})} \sum\limits_{i=1,2} \norm{ \nabla (\InterpUnfitted u)_i }_{\Gamma}^2 + h \norm{ \jump{ \mu \nabla ( \InterpUnfitted u - \Ext u) } \cdot \normal_{ \Gamma }   }_{ \Gamma }^2  \\ 
& \lesssim  h^{2(q+\frac{1}{2})} \sum\limits_{i=1,2} \norm{ \nabla (\InterpUnfitted u)_i }_{\Gamma}^2 + h \sum\limits_{i=1,2} \norm{  \nabla ( \InterpUnfitted u - \Ext u)_i  }_{ \Gamma }^2 
  := \mathrm{I}_1 + \mathrm{I}_2.  
\end{align*}
From Lemma~\ref{lem:Deriv-IF} (a) and Lemma~\ref{lem:approx_unfitted} (d) we obtain 
\begin{align*}
\mathrm{I}_2 \lesssim \sum\limits_{i=1,2}  \norm{ \nabla (  \Ext_i R_i u - \InterpBackgroundMesh \Ext_i R_i u) \circ \Phi_h }_{-\frac{1}{2},h,\Gamma_h }^2 
\lesssim h^{2q} \norm{u}_{ H^{q+1}(\BrokenOmega) }^2.
\end{align*} 
Similarly, $ \mathrm{I}_1 \lesssim h^{2 (q+\frac{1}{2}) } \norm{u}_{ H^{q+1}(\BrokenOmega) }^2$, which concludes the proof of (a).
\item Since $u \in H^{q+1}(\BrokenOmega)$ with $q \geq 1$ we are allowed to take the tangential derivative of equation (\ref{eq:jump_IF}) to obtain $\jump{ \nabla_{\Gamma} u}_{\Gamma} = 0$. 
Then using Lemma~\ref{lem:Deriv-IF} (c) yields
\begin{align*}
& \StabNablaGamma{\InterpUnfitted u \circ \Phi_h }{ \InterpUnfitted u \circ \Phi_h } = h \bar{\mu}  \norm{ \jump{ \nabla_{ \Gamma_h }  \InterpUnfitted u \circ \Phi_h } }_{ \Gamma_h }^2  \\
	& \lesssim  h^{2(q+\frac{1}{2})} \sum\limits_{i=1,2} \norm{ \nabla (\InterpUnfitted u)_i }_{\Gamma}^2 + h \norm{ \jump{ \nabla_{ \Gamma } \InterpUnfitted u } }_{ \Gamma }^2 \\
	& \sim h^{2(q+ \frac{1}{2})} \sum\limits_{i=1,2} \norm{ \nabla (\InterpUnfitted u)_i }_{\Gamma}^2 + h \norm{ \jump{ \nabla_{ \Gamma } (\InterpUnfitted u - \Ext u ) } }_{ \Gamma }^2 
 \lesssim \mathrm{I}_1 + \mathrm{I}_2.
\end{align*}
The terms $\mathrm{I}_1$ and $ \mathrm{I}_2$ are the same as the ones already estimated in (a).
\item As $\Phi_h(\Gamma_h) = \Gamma$ we have $0= \jump{u}_{\Gamma} = \jump{ \Ext u}_{\Gamma} =  \jump{ \Ext u \circ \Phi_h}_{\Gamma_h}  $. 
This allows to skip the transformation step to the exact interface that was necessary in part (a)-(b). 
Hence, by applying Lemma~\ref{lem:approx_unfitted} (c): 
\begin{align*}
& \StabNitsche{ \InterpUnfitted u \circ \Phi_h }{ \InterpUnfitted u \circ \Phi_h } =
\frac{\bar{\mu}}{ h } \norm{ \jump{  \InterpUnfitted u \circ \Phi_h   }_{\Gamma_h }  }_{ \Gamma_h }^2  =  \frac{\bar{\mu}}{ h }  \norm{ \jump{  \InterpUnfitted u \circ \Phi_h - \Ext u \circ \Phi_h  }_{\Gamma_h }  }_{ \Gamma_h }^2 \\
& \lesssim \sum\limits_{i=1,2} \norm{  (\Ext_i R_i u - \InterpBackgroundMesh \Ext_i R_i u) \circ \Phi_h }_{\frac{1}{2},h,\Gamma_h }^2 
  \lesssim h^{2q} \norm{u}_{ H^{q+1}(\BrokenOmega) }^2.
\end{align*}
\end{enumerate}
\end{proof} 
Next we consider the least squares term.
\begin{lemm}[Least-squares term]\label{lem:GLS}
	Let $u \in H^{q+1}(\BrokenOmega) \cap W^{3,\infty}(\BrokenOmega) $ solve (\ref{eq:PDE_bulk})-(\ref{eq:normal_jump_IF}). 
Then it holds that 
\[
\sum\limits_{i=1,2} \sum\limits_{ \el \in \mesh^{i} } h^2 \norm{ f_{i,h} - \calL(  \InterpUnfitted  u \circ  \Phi_h )   }^2_{ \Theta_h(T_i) } 
	\lesssim h^{2q} \norm{u}_{q}^2.  
\]
\end{lemm}
\begin{proof}
We have  
\begin{align*}
& \sum\limits_{i=1,2} \sum\limits_{ \el \in \mesh^{i} } h^2 \norm{ f_{i,h} - \calL(  \InterpUnfitted  u \circ  \Phi_h )   }^2_{ \Theta_h(T_i) } \lesssim \sum\limits_{i=1,2} \sum\limits_{ \el \in \mesh^{i} } h^2 \norm{ f_{i,h} - \calL(  \Ext_i u_i  \circ  \Phi_h )   }^2_{ \Theta_h(T_i) }   \\ 
&  
+ \sum\limits_{i=1,2} \sum\limits_{ \el \in \mesh^{i} } h^2 \norm{  \calL(  ( \Ext_i R_i u   -  \InterpBackgroundMesh \Ext_i R_i u)  \circ  \Phi_h )   }^2_{ \Theta_h(T) }  \\
& \lesssim \sum\limits_{i=1,2} \sum\limits_{ \el \in \mesh^{i} } h^2 \norm{ f_{i,h} - \calL(  \Ext_i u_i  \circ  \Phi_h )   }^2_{ \Theta_h(T_i) } 
	+ h^{2q} \norm{u}_{ H^{q+1}(\BrokenOmega) }^2,
\end{align*}
where Lemma~\ref{lem:approx_unfitted} (e) has been applied. To treat the remaining term further, we employ Lemma~\ref{lem:norm-equiv} (a) and the triangle inequality to split
\begin{align*}
& \sum\limits_{i=1,2} \sum\limits_{ \el \in \mesh^{i} } h^2 \norm{ f_{i,h} - \calL(  \Ext_i u_i  \circ  \Phi_h )   }^2_{ \Theta_h(T_i) }   \\ 
& \lesssim \sum\limits_{i=1,2} \sum\limits_{ \el \in \mesh^{i} } h^2 \norm{ f_{i,h} \circ \Phi_h^{-1} - \calL(  \Ext_i u_i  \circ  \Phi_h ) \circ \Phi_h^{-1}  }^2_{ \Psi_h(T_i) } \\
& \lesssim  \sum\limits_{i=1,2} \sum\limits_{ \el \in \mesh^{i} } h^2 \norm{ f_{i,h} \circ \Phi_h^{-1} - f_{i,h}  }^2_{ \Psi_h(T_i) }
+ \sum\limits_{i=1,2} \sum\limits_{ \el \in \mesh^{i} } h^2 \norm{ f_{i,h}  -  \calL(  \Ext_i u_i  \circ  \Phi_h ) \circ \Phi_h^{-1}   }^2_{ \Psi_h(T_i) } \\
& := \mathrm{J}_1 + \mathrm{J}_2. 
\end{align*}
From Lemma~\ref{lem:pullback} (a) we obtain
\[
\mathrm{J}_1 \lesssim h^{2(q+2)} \sum\limits_{i=1,2} \norm{ \nabla f_{i,h} }_{\infty,\Omega }^2 
\lesssim h^{2(q+2)} \norm{ f }_{W^{1,\infty}( \BrokenOmega ) }^2  
\lesssim h^{2(q+2)} \norm{u}_{ W^{3,\infty}( \BrokenOmega ) }^2 
\]
by our smoothness assumption on the data extension. 
To treat $ \mathrm{J}_2$ we note that as  $\Psi(T_i) \subset \Omega_i$ by construction it holds 
\[
f_{i,h}|_{ \Psi(T_i) }  = f_i = \calL R_i u =  (\calL \Ext_i R_i u)|_{ \Psi(T_i) }  . 
\]
Hence, we obtain from Lemma~\ref{lem:pullback} (d) 
\begin{align*}
\mathrm{J}_2  &= \sum\limits_{i=1,2} \sum\limits_{ \el \in \mesh^{i} } h^2 \norm{ \calL (\Ext_i R_i u )  -  \calL(  \Ext_i u_i  \circ  \Phi_h ) \circ \Phi_h^{-1}   }^2_{ \Psi_h(T_i) } \lesssim 
	h^{2q} \norm{u}_{H^2(\BrokenOmega)}^2,
\end{align*}
which concludes the argument.
\end{proof}
We also have to estimate the continuous interior penalty term. 
The proof is similar to \cite[Lemma 13]{L_GUFEMA_2017} where a related estimate for a facet-based ghost penalty stabilization, which is a localized form of continuous interior penalty (possibly including higher order jumps), has been established.
To this end, an interpolation operator which maps directly into $ \curvedFes{\Theta}^{\Gamma} $ was constructed.  
Note that our interpolator $\InterpUnfitted$ maps to $ \curvedFes{\Psi}^{\Gamma}$ instead so that we first compose with $\Phi_h$ to obtain $\InterpUnfitted  u \circ  \Phi_h \in  \curvedFes{\Theta}^{\Gamma} $. 
\begin{lemm}[Continuous interior penalty]\label{lem:CIP}
Let $u \in H^{q+1}(\BrokenOmega) \cap W^{3,\infty}(\BrokenOmega)$ solve (\ref{eq:PDE_bulk})-(\ref{eq:normal_jump_IF}). 
Then it holds that
	\[ \StabCIP{ \InterpUnfitted  u \circ  \Phi_h }{ \InterpUnfitted  u \circ  \Phi_h }^{1/2}  \lesssim h^{q} \norm{u}_{q}. \]
\end{lemm}
\begin{proof}
Since  $\Ext_i R_i u \in H^2(\Omega)$ we have 
\begin{align*}
& \StabCIP{ \InterpUnfitted  u \circ  \Phi_h }{ \InterpUnfitted  u \circ  \Phi_h }  =
\sum\limits_{i=1,2} \sum\limits_{ F \in  \facets^i } h \! \! \! \! \int\limits_{\Theta_h(F) } \! \! \! \! \mu_i \norm{ \jump{ \nabla ( \Ext_i R_i u  -   (\InterpBackgroundMesh \Ext_i R_i u) \circ \Phi_h   )  }_{ \Theta_h(F) }  \cdot \normal }^2 \! \!  \dS
\\ 
& \lesssim \sum\limits_{i=1,2} \sum\limits_{ \el \in \mesh^{i} } h \norm{ \nabla ( \Ext_i R_i u - (\InterpBackgroundMesh \Ext_i R_i u) \circ \Phi_h )  }_{ \partial \Theta_h( \el ) }^2   \\ 
& \lesssim \sum\limits_{i=1,2} \sum\limits_{ \el \in \mesh^{i} }  \norm{  \Ext_i R_i u - (\InterpBackgroundMesh \Ext_i R_i u) \circ \Phi_h  }_{ H^1( \Theta_h( \el) ) }^2 + h^2 \norm{  \Ext_i R_i u - (\InterpBackgroundMesh \Ext_i R_i u) \circ \Phi_h   }_{ H^2( \Theta_h(T) ) }^2 \\ 
& := \mathrm{I}_1 + \mathrm{I}_2.  
\end{align*}
From Lemma~\ref{lem:pullback} (c) and Lemma~\ref{lem:approx_unfitted} (e) we obtain 
\begin{align*}
\mathrm{I}_2 & \lesssim \sum\limits_{i=1,2} \sum\limits_{ \el \in \mesh^{i} } h^2 \norm{ \Ext_i R_i u - (\Ext_i R_i u) \circ \Phi_h }_{ H^2(\Theta_h( \el) ) }^2 + 
h^2 \norm{ (\Ext_i R_i u - \InterpBackgroundMesh \Ext_i R_i u) \circ \Phi_h }_{ H^2(\Theta_h( \el) ) }^2 \\
& \lesssim  h^{2q} \norm{u}_{W^{3,\infty}(\BrokenOmega)}^2 + h^{2q}  \norm{u}_{H^{q+1}(\BrokenOmega)}^2. 
\end{align*} 	
The term $\mathrm{I}_1$ is estimated similarly by appealing to Lemma~\ref{lem:pullback} (b) and Lemma~\ref{lem:approx_unfitted} (a)-(b). 
\end{proof}
Let us define 
\begin{align*}
& S_h( \InterpUnfitted  u \circ  \Phi_h ) :=  \sum\limits_{i=1,2} \sum\limits_{ \el \in \mesh^{i} } h^2 \norm{ f_{i,h} - \calL(  \InterpUnfitted  u \circ  \Phi_h )   }^2_{ \Theta_h(T_i) } 
+  \StabTikh{ \InterpUnfitted  u \circ  \Phi_h  }{\InterpUnfitted  u \circ  \Phi_h  }   \\
&  + \StabCIP{ \InterpUnfitted  u \circ  \Phi_h }{ \InterpUnfitted  u \circ  \Phi_h } 
  +\StabIFall{ \InterpUnfitted u \circ \Phi_h  }{ \InterpUnfitted u \circ \Phi_h } 
\end{align*}
Then by combining the previous results we obtain:
\begin{coro}\label{cor:interp-in-stab-bound}
Let $u \in H^{q+1}(\BrokenOmega) \cap W^{3,\infty}(\BrokenOmega) $ solve (\ref{eq:PDE_bulk})-(\ref{eq:normal_jump_IF}). 
Then
\[
S_h( \InterpUnfitted  u \circ  \Phi_h )  \lesssim h^{2q} \norm{u}_{q}^2. 
\]
\end{coro}
\begin{proof}
This follows from Lemma~\ref{lem:IF-stab-interp}, Lemma~\ref{lem:GLS}, Lemma~\ref{lem:CIP} and 
\begin{align*}
\StabTikh{ \InterpUnfitted  u \circ  \Phi_h  }{\InterpUnfitted  u \circ  \Phi_h  }   
	\lesssim h^{2q} \norm{u}_{H^1(\BrokenOmega)}^2.
\end{align*}
\end{proof}

\section{Error analysis}\label{section:error-analysis}
Thanks to the results established in the previous two sections, the usual error analysis
for methods based on the framework in \cite{B13} can be applied. 
First we establish convergence of the error in the stabilization norm $\tnorm{\cdot}$, see Corollary~\ref{cor:conv-tnorm}.
This result is independent of the stability properties of the continuous problem, i.e.\ it can be established without using equation (\ref{eq:cond_stab}). 
The stabilization norm $\tnorm{\cdot}$ is however too weak to provide a meaningful convergence measure for practical applications. 
To deduce $L^2$-convergence in the target domain we need to utilize the conditional stability estimate (\ref{eq:cond_stab}). 
This allows us to establish Theorem~\ref{thm:L2B-error-estimate}, which is the only result in this article depending on the stability properties of the continuum problem.

We start by analyzing the discretization error.
\begin{prop}\label{prop:conv_discr_error_tnorm}
	Let $u \in H^{q+1}(\BrokenOmega) \cap W^{3,\infty}(\BrokenOmega) $ be the exact solution of (\ref{eq:PDE_bulk})-(\ref{eq:measurements}). Let $(u_h,z_h) \in \curvedFes{\Theta}^{\Gamma} \times \curvedFes{\Theta}^{0}  $ be the solution of (\ref{eq:discr_stab_var_form}). Then 
\[ 
\tnorm{ (u_h - \InterpUnfitted u \circ \Phi_h, z_h) } \lesssim h^q \norm{u}_{q} + \delta.
\]
\end{prop}
\begin{proof}
In view of the inf-sup condition (\ref{eq:inf-sup}), it suffices to show that 
\[ B_h[(u_h - \InterpUnfitted u \circ \Phi_h,z_h),(v_h,w_h)] \lesssim 
	\tnorm{ (v_h,w_h) } ( h^q \norm{u}_{q} + \delta) 
	\] 
for all $(v_h,w_h) \in \curvedFes{\Theta}^{\Gamma} \times \curvedFes{\Theta}^{0} $. 
	Using first order optimality (\ref{eq:discr_stab_var_form}) and consistency (Lemma~\ref{lem:consistency_A}) yields 
\begin{align*}
& B_h[(u_h - \InterpUnfitted u \circ \Phi_h,z_h),(v_h,w_h)] =  
	(u- \InterpUnfitted u \circ \Phi_h, v_h)_{\omega}  + \ell_{h}(w_h) + g(v_h,w_h)  \\    
	& + \sum\limits_{i=1,2} \sum\limits_{ \el \in \mesh^{i} } h^2 ( f_{i,h} , \calL v_{h,i} )_{  \Theta_h( \el_i) } - s_h( \InterpUnfitted u \circ \Phi_h, v_h) - A_h( \InterpUnfitted u \circ \Phi_h , w_h ) \\
	& = \underbrace{(u- \InterpUnfitted u \circ \Phi_h, v_h)_{\omega}}_{I_1} + \underbrace{[ \ell_h(w_h) - \ell( w_h \circ \Phi_h^{-1} )  ]}_{I_2} + \underbrace{[A(u, w_h \circ \Phi_h^{-1} ) - A_h( \InterpUnfitted u \circ \Phi_h , w_h )]}_{I_3} \\
	& + \underbrace{ \left[ \sum\limits_{i=1,2} \sum\limits_{ \el \in \mesh^{i} } h^2 ( f_{i,h} - \calL \InterpUnfitted u \circ \Phi_h , \calL v_{h,i} )_{  \Theta_h( \el_i) }  - \tilde{s}_h( \InterpUnfitted u \circ \Phi_h ,v_h)  \right]}_{I_4} + \underbrace{g(v_h,w_h)}_{I_5}.
\end{align*} 
We consider the terms $\mathrm{I}_j$ separately.
\begin{itemize}
\item Recall that we assume $\Phi_{h} |_{\omega} = \ID$ and $\omega \subset \Omega_i$ for exactly one $i \in \{1,2\}$ so that from the interpolation estimates 
in Lemma~\ref{lem:approx_unfitted} we obtain
\[
	\mathrm{I}_1  \lesssim \norm{v_h}_{\omega} \norm{ u -  \InterpUnfitted  u }_{\omega} 
		\lesssim h^{q+1} \norm{v_h}_{\omega} \norm{u}_{H^{q+1}(\BrokenOmega)}. 
\]
\item From Lemma~\ref{lem:bil-PDE-geom-error-Christoph} and Lemma~\ref{lem:lem:bil-PDE-geom-error-interp} we have 
\[ \mathrm{I}_2 + \mathrm{I}_3 \lesssim h^q \norm{u}_{q} \norm{w_h}_{ s_h^{\ast} }. \]
\item For the second to last term we apply Corollary~\ref{cor:interp-in-stab-bound} to obtain
\begin{align*}
\mathrm{I}_4 \lesssim  S_h( \InterpUnfitted  u \circ  \Phi_h )^{1/2} \norm{v_h}_{s_h} 
\lesssim h^q \norm{u}_{q}  \norm{v_h}_{s_h}. 
\end{align*}
\item Now only the perturbation remains which has already been estimated in (\ref{eq:perturbation-bound}).
\end{itemize}
Combining these estimates yields the claim.
\end{proof}
\begin{rema}[Regularity assumptions on $u$]
	Recall that in several of the proofs\footnote{See e.g.\ Lemma~\ref{lem:bil-PDE-geom-error-Christoph} (b) and Lemma~\ref{lem:GLS}. In Lemma~\ref{lem:CIP} instead the requirement $u \in W^{3,\infty}(\BrokenOmega)$ appears to be genuine. } of the previous section the norms $\norm{f}_{ W^{1,\infty}( \BrokenOmega ) }$ have been replaced by the norms $\norm{u}_{ W^{3,\infty}( \BrokenOmega ) }$ using that $u$ solves (\ref{eq:PDE_bulk}). Note that the requirement $f \in W^{1,\infty}( \BrokenOmega )$ is already needed when performing a Strang-type error analysis of the approximation of a smooth boundary by a polygonal domain (i.e. even without the isoparametric mapping) for well-posed problems. Hence, the regularity assumptions employed here seem fairly natural.
\end{rema}
\begin{coro}\label{cor:conv-tnorm}
Let $u \in H^{q+1}(\BrokenOmega) \cap W^{3,\infty}(\BrokenOmega) $ be the exact solution of (\ref{eq:PDE_bulk})-(\ref{eq:measurements}). 
Let $(u_h,z_h) \in \curvedFes{\Theta}^{\Gamma} \times  \curvedFes{\Theta}^{0} $ be the solution of (\ref{eq:discr_stab_var_form}). Then 
\[
\tnorm{ (u_h - \Ext u \circ \Phi_h, z_h) } \lesssim h^q \norm{u}_{q} + \delta.
\]
\end{coro}
\begin{proof}
By the triangle inequality 
\[
\tnorm{ (u_h - \Ext u \circ \Phi_h, z_h) } 
\lesssim 
\tnorm{ (u_h - \InterpUnfitted u \circ \Phi_h, z_h) } 
+ 
\tnorm{ ( (\InterpUnfitted u  -  \Ext u) \circ \Phi_h, 0) }. 
\]
The first term has been estimated in Proposition~\ref{prop:conv_discr_error_tnorm} while
the approximation error is easily treated by appealing to the interpolation results given in Lemma~\ref{lem:approx_unfitted}.
\end{proof}
Corollary~\ref{cor:conv-tnorm} has to be combined with the conditional stability 
estimate from Corollary~\ref{cor:conditional_stabiliy_estimate} to deduce convergence rates in the target domain. 
To this end, we would like to apply the latter to $u - u_h \circ \Phi_h^{-1}$.
Unfortunately, this is not immediately possible as $u_h \circ \Phi_h^{-1} \notin H^1(\Omega)$ due to the use of an unfitted discretization.
This can be fixed by adding a corrector function which removes possible jumps across the interface.
The next lemma shows that the norm of this correction can be controlled by the stabilization.
\begin{lemm}\label{lem:corr_IF_jumps}
Let $u \in H^{q+1}(\BrokenOmega) \cap W^{3,\infty}(\BrokenOmega) $ be the exact solution of (\ref{eq:PDE_bulk})-(\ref{eq:measurements}). 
Let $(u_h,z_h) \in \curvedFes{\Theta}^{\Gamma} \times \curvedFes{\Theta}^{0}$ be the solution of (\ref{eq:discr_stab_var_form}). 
Let $\varphi \in H^1(\Omega_1)$ be the solution of 
\begin{equation}\label{eq:PDE_jump_correction}
	\left\{ \begin{array}{rcll} & -\Delta \varphi = 0, \quad  &\text{ in } \Omega_1, \\
	& \varphi = - \jump{u_h \circ \Phi_h^{-1} } \quad  & \text{ on } \Gamma.   \end{array}\right.
\end{equation}
Then 
\[
\norm{\varphi}_{ H^1(\Omega_1) }  \lesssim h^q \norm{u}_{q} + \delta.
\]
\end{lemm}
\begin{proof}
It follows from the Lax-Milgram lemma that 
\[
\norm{\varphi}_{ H^1(\Omega_1) }  \lesssim \norm{  \jump{u_h \circ \Phi_h^{-1} } }_{ H^{1/2}(\Gamma) }.
\]
We have to estimate the jump term on the right hand side using convergence of $u_h -\Ext u \circ \Phi_h$ in the $\tnorm{\cdot}$-norm, see Corollary~\ref{cor:conv-tnorm}. 
To achieve this, there are two main obstacles to overcome. 
Firstly, we have to eliminate the $H^{1/2}$-norm. 
To this end, we apply a Gagliardo-Nirenberg inequality on $\Gamma$ 
and use $ \jump{u}_{\Gamma} = \jump{ \Ext u}_{\Gamma} = 0$ and $ \jump{ \nabla_{\Gamma} u}_{\Gamma} = \jump{ \nabla_{\Gamma} \Ext u}_{\Gamma} = 0$ to obtain
\begin{align*}
\norm{   \jump{u_h \circ \Phi_h^{-1} }  }_{ H^{1/2}(\Gamma) }^2 & \lesssim \norm{ \jump{u_h \circ \Phi_h^{-1} } }_{L^2(\Gamma)}  \norm{ \jump{u_h \circ \Phi_h^{-1} } }_{H^1(\Gamma)}  \\
& \lesssim h^{-1} \norm{ \jump{u_h \circ \Phi_h^{-1} - \Ext u } }_{ \Gamma }^2 + h \norm{ \jump{ \nabla_{\Gamma} ( u_h \circ \Phi_h^{-1} - \Ext u) } }_{ \Gamma }^2.  
\end{align*}
Secondly, we have to transform from the exact interface to its approximation $\Gamma_h$ on 
which the discretization is defined. 
According to Lemma~\ref{lem:Deriv-IF} (c), this will lead to an additional  geometrical error on 
the discrete interfaces involving the full gradient, that is 
\begin{align*}
\norm{   \jump{u_h \circ \Phi_h^{-1} }  }_{ H^{1/2}(\Gamma) }^2 & \lesssim  
h^{-1} \norm{ \jump{u_h - \Ext u \circ \Phi_h } }_{ \Gamma_h }^2 + h \norm{ \jump{ \nabla_{\Gamma_h} ( u_h  - \Ext u \circ \Phi_h ) } }_{ \Gamma_h }^2 \\ 
	& +  h^{2(q+\frac{1}{2})} \sum\limits_{i=1,2} \norm{ \nabla ( u_{h,i} - \Ext_i u_i \circ \Phi_h  ) }_{\Gamma_h}^2 \\
	& \lesssim \tnorm{ (u_h - \Ext u \circ \Phi_h, z_h) }^2  +  h^{2(q+\frac{1}{2})} \sum\limits_{i=1,2} ( \norm{ \nabla  u_{h,i}  }_{\Gamma_h}^2 + \norm{ \nabla \Ext_i u_i  }_{\Gamma}^2 )  \\
& \lesssim (h^q \norm{u}_{q} + \delta)^2 + h^{2(q+\frac{1}{2})} \sum\limits_{i=1,2} \norm{ \nabla  u_{h,i}  }_{\Gamma_h}^2. 
\end{align*}
Here we used that the jumps on the interface are controlled by the stabilization and Corollary~\ref{cor:conv-tnorm}.
We will bound the remaining term in terms of the weak $H^1$-penalty using the trace\footnote{Here it is essential that $\StabTikh{\cdot}{\cdot}$ is defined on the active mesh $ \Omega_{i,h}^+ $ and not merely on $\Omega_{i,h}$.} (\ref{eq:trace-inequality-dfm}) and inverse inequalitites (\ref{eq:inverse-ieq-gradient}): 
\begin{align*}
	h \norm{ \nabla u_{h,i}}_{ \Gamma_h}^2 & \lesssim \sum\limits_{\el \in \Cutel} \norm{u_{h,i} }_{ H^1(\Theta_h(\el)) }^2 + h^2 \norm{u_{h,i}}_{ H^2(\Theta_h(\el)) }^2 
	\lesssim \sum\limits_{\el \in \Cutel} \norm{u_{h,i} }_{ H^1(\Theta_h(\el)) }^2. 
\end{align*}	
Hence, 
\begin{align*}
	& h^{2(q+\frac{1}{2})} \sum\limits_{i=1,2} \norm{ \nabla  u_{h,i}  }_{\Gamma_h}^2 \lesssim  \StabTikh{ u_h }{ u_h } \\
& \lesssim \StabTikh{ u_h - \Ext u \circ \Phi_h  }{ u_h - \Ext u \circ \Phi_h }
	+  h^{2q} \sum\limits_{i=1,2} \norm{  \Ext_i u_i \circ \Phi_h }_{H^1(\Omega)}^2  \lesssim (h^q \norm{u}_{q} + \delta)^2
\end{align*}
by Corollary~\ref{cor:conv-tnorm}. 
\end{proof}
To benefit from the application of the conditional stability estimate, we have to control 
the terms on the right hand side of (\ref{eq:cond_stab}). 
In particular, we have to control the residual in the $H^{-1}$-norm, which is ensured by the following lemma. 
\begin{lemm}\label{lem:weak-conv}
Let $u \in H^{q+1}(\BrokenOmega) \cap W^{3,\infty}(\BrokenOmega) $ be the exact solution of (\ref{eq:PDE_bulk})-(\ref{eq:measurements}). 
Let $(u_h,z_h) \in \curvedFes{\Theta}^{\Gamma} \times \curvedFes{\Theta}^{0} $ be the solution of (\ref{eq:discr_stab_var_form}). Then for any $w \in H^1_0(\Omega)$ it holds that  
\[
a(u-u_h \circ \Phi_h^{-1}, w) \lesssim (h^q \norm{u}_{q} + \delta) \norm{w}_{H^1(\Omega)}.
\]
\end{lemm}
\begin{proof}
	Let us define the shorthand $w_h :=  w \circ \Phi_h, \tilde{w}_h := w_h \circ \Phi_h^{-1} $  and $\tilde{u} = u \circ \Phi_h$.
Clearly, we have to deduce the claim eventually from Corollary~\ref{cor:conv-tnorm} for which 
we have to transform to the approximate geometry. Thus, the first step in this proof is to show that: 
\begin{equation}\label{eq:from-a-to-ah}
a(u-u_h \circ \Phi_h^{-1}, w) \lesssim a_h(\tilde{u} - u_h, w_h) + (h^q \norm{u}_{q} + \delta) \norm{w}_{H^1(\Omega)}  
\end{equation}
holds true. 
To this end, note that $w =\tilde{w}_h$, so we have 
\[
a(u-u_h \circ \Phi_h^{-1}, w)  =  [a(u, \tilde{w}_h ) - a_h(\tilde{u}, w_h)] + [ a_h(u_h,w_h) - a(u_h \circ \Phi_h^{-1}, \tilde{w}_h ) ] + a_h(\tilde{u} - u_h, w_h ). 
\]
	To estimate the first bracket we simply appeal to Lemma~\ref{lem:bil-PDE-geom-error-Christoph} (a). 
To control the second bracket, the term $\StabTikh{\cdot}{\cdot}$ is essential. By definition (\ref{eq:def-Tikh}) of this term, Lemma~\ref{lem:bil-PDE-geom-error-Christoph} (a) and Proposition~\ref{prop:conv_discr_error_tnorm} we deduce 
\begin{align*}
& \abs{ a_h(u_h,w_h) - a(u_h \circ \Phi_h^{-1}, \tilde{w}_h ) }  \lesssim
	h^q  \norm{w}_{H^1(\Omega)} \sum\limits_{i=1,2} \norm{ u_{h,i} }_{ H^1( \Omega_{i,h})  }   \\
	& \lesssim h^q  \norm{w}_{H^1(\Omega)} \left( h^{-q} \StabTikh{ u_h - \InterpUnfitted u \circ \Phi_h }{ u_h - \InterpUnfitted u \circ \Phi_h }^{1/2}  + \norm{\InterpUnfitted u \circ \Phi_h}_{ H^1(\BrokenOmegah) }  \right) \\
	& \lesssim \norm{w}_{H^1(\Omega)} ( h^q  \norm{u}_{q} + \delta),
\end{align*}
which shows (\ref{eq:from-a-to-ah}). 
The next step is to bound $a_h(\tilde{u} - u_h, w_h )$ by terms which can be controlled by the stabilization. We claim that 
\begin{equation}\label{eq:bound-ah-stab}
a_h(\tilde{u} - u_h, w_h ) \lesssim h^q \norm{w}_{H^1(\Omega)} \norm{u}_{q} + \sum\limits_{j=1}^{6}\mathrm{I}_j, 
\end{equation}
where 
\[
\mathrm{I}_1 := \sum\limits_{i=1,2} \sum\limits_{ \el \in \mesh^{i} }  ( \calL ( \Ext_i R_i u \circ \Phi_h  - u_{h,i}) ,  w \circ \Phi_h - ( \InterpBackgroundMeshZ w ) \circ \Phi_h )_{  \Theta_h( \el_i) }
\]
\[
	\mathrm{I}_2 :=  \sum\limits_{i=1,2} \sum\limits_{ F\in  \facets^i } h \int\limits_{\Theta_h(F \cap \el_i)  } \!\!\!\!\! \mu_i \jump{ \nabla ( \Ext_i R_i u \circ \Phi_h  - u_{h,i})  }_{ \Theta_h(F) } \cdot \normal \;  (w \circ \Phi_h - ( \InterpBackgroundMeshZ w ) \circ \Phi_h ) \; \dS_{ \Theta_h(F) }
\]
\[
\mathrm{I}_3 := N^{c}_h(u_h,\InterpUnfittedZ w \circ \Phi_h ), \quad
\mathrm{I}_4 := - s_h^{\ast}(z_h, \InterpUnfittedZ w \circ \Phi_h  ), \quad 
\mathrm{I}_5 := - ( \delta f_h , \InterpUnfittedZ w \circ \Phi_h )_{\BrokenOmegah} 
\]
\[
\mathrm{I}_6 := \left( \jump{ \mu \nabla( \Ext u \circ \Phi_h - u_h  ) } \cdot \normal_{\Gamma_h}, w \circ \Phi_h  -  \InterpUnfittedZ w \circ \Phi_h   \right)_{ \Gamma_h  }. \;
 \]
To establish (\ref{eq:bound-ah-stab}), we first use consistency 
\[
\ell( \InterpUnfittedZ w ) = a(u ,\InterpUnfittedZ w),
\]
which follows from Lemma~\ref{lem:consistency_A}. Hence, 
\begin{align}
a_h(\tilde{u} - u_h, w_h ) &= a_h(\tilde{u} - u_h, w \circ \Phi_h - \InterpUnfittedZ w \circ \Phi_h ) - a_h(u_h ,  \InterpUnfittedZ w \circ \Phi_h ) \nonumber  \\ 
	& + [  a_h( \tilde{u},  \InterpUnfittedZ w \circ \Phi_h ) -  a(u, \InterpUnfittedZ w) ] + \ell( \InterpUnfittedZ w )  \label{eq:weak-conv_aux}. 
\end{align}
By appealing to Lemma~\ref{lem:bil-PDE-geom-error-Christoph} (a)  we obtain 
\begin{align*}
\abs{ a_h( \tilde{u},  \InterpUnfittedZ w \circ \Phi_h ) -  a(u, \InterpUnfittedZ w) } &= \abs{ a(u, (\InterpUnfittedZ w \circ \Phi_h ) \circ \Phi_h^{-1} ) - a_h(\tilde{u}, \InterpUnfittedZ w \circ \Phi_h) } \\
& \lesssim h^q \norm{u}_{H^1(\BrokenOmega)} \norm{w}_{H^1(\Omega)}. 
\end{align*}
	Next we add the first order optimality condition 
\begin{align*}
& A_h(u_h, \InterpUnfittedZ w \circ \Phi_h) + \mathrm{I}_4  - \ell_h( \InterpUnfittedZ w \circ \Phi_h ) + \mathrm{I}_5  = 0, 
\end{align*}
which follows from (\ref{eq:discr_stab_var_form}), to the right hand side of (\ref{eq:weak-conv_aux}).
Using that 
\[
\abs{ \ell( \InterpUnfittedZ w ) - \ell_h( \InterpUnfittedZ w \circ \Phi_h ) } 
\lesssim h^q \norm{u}_{q} \norm{   \InterpUnfittedZ w \circ \Phi_h  }_{\BrokenOmegah } \lesssim h^q \norm{u}_{q} \norm{w}_{H^1(\Omega)},
\]
as ensured by Lemma~\ref{lem:bil-PDE-geom-error-Christoph} (b) and 
\[
- a_h(u_h ,  \InterpUnfittedZ w \circ \Phi_h ) + A_h(u_h, \InterpUnfittedZ w \circ \Phi_h) 
= \mathrm{I}_3 
\]
yields 
 \begin{align}
a_h(\tilde{u} - u_h, w_h )  \lesssim h^q \norm{u}_{q} \norm{w}_{H^1(\Omega)} + a_h(\tilde{u} - u_h, w \circ \Phi_h - \InterpUnfittedZ w \circ \Phi_h )  + \mathrm{I}_3  + \mathrm{I}_4 + \mathrm{I}_5. \label{eq:weak-conv_aux2}  
 \end{align}
 We recognize that the claimed inequality (\ref{eq:bound-ah-stab}) slowly begins to appear.
 The final step is an element-wise integration by parts of $a_h(\tilde{u} - u_h, w \circ \Phi_h - \InterpUnfittedZ w \circ \Phi_h ) $, which will 
 in particular lead to a boundary term on $\partial \Omega_{i,h} \setminus \partial \Omega$ that can be written as $\mathrm{I}_6$. 
 Hence, integration by parts yields 
\begin{align*}
	a_h(\tilde{u} - u_h, w \circ \Phi_h - \InterpUnfittedZ w \circ \Phi_h ) =  \mathrm{I}_1 + \mathrm{I}_2 + \mathrm{I}_6. 
\end{align*}
Inserting this identity into (\ref{eq:weak-conv_aux2}) establishes (\ref{eq:bound-ah-stab}). \par 
In view of (\ref{eq:from-a-to-ah}) and (\ref{eq:bound-ah-stab}) it now remains to show that 
\[ 
\sum\limits_{j=1}^{6}\mathrm{I}_j \lesssim \norm{w}_{H^1(\Omega)} ( h^q \norm{u}_{q} + \delta).
\]
Corollary~\ref{cor:conv-tnorm} and interpolation estimates for $\InterpUnfittedZ$ yield
\begin{align*}
\mathrm{I}_1 &\lesssim  (\StabGLS{ u_h - \Ext u \circ \Phi_h }{ u_h - \Ext u \circ \Phi_h })^{1/2}
h^{-1} \norm{ (w - \InterpUnfittedZ w) \circ \Phi_h  }_{ \Omega }  \\ 
	& \lesssim  \norm{w}_{H^1(\Omega)} (h^q \norm{u}_{q} + \delta).
\end{align*}
Similarly, by Corollary~\ref{cor:conv-tnorm} and interpolation estimates  
\begin{align*}
\mathrm{I}_2 & \lesssim (\StabCIP{ u_h - \Ext u \circ \Phi_h }{ u_h - \Ext u \circ \Phi_h })^{1/2}
 \sum\limits_{\el \in \mesh} h^{\frac{1}{2}} \norm{  ( w - \InterpBackgroundMeshZ w) \circ \Phi_h  }_{ \partial \Theta_h(\el) } \\ 
& \lesssim  \norm{w}_{H^1(\Omega)} (h^q \norm{u}_{q} + \delta).
\end{align*}
Using the estimate for $N^{c}_h(u_h,\InterpUnfittedZ w \circ \Phi_h )$ from Lemma~\ref{lem:Ah-continuity} along with $ 0 = \jump{ u }_{\Gamma} =  \jump{ \Ext u }_{\Gamma} = \jump{ \Ext u \circ \Phi_h }_{\Gamma_h}$ and Corollary~\ref{cor:conv-tnorm} as well as $H^1$-stability of $\InterpUnfittedZ$ we obtain 
\[
\mathrm{I}_3 +  \mathrm{I}_4 \lesssim  \norm{ \InterpUnfittedZ w \circ \Phi_h  }_{ s_h^{\ast} }  \tnorm{ (u_h - \Ext u \circ \Phi_h, z_h) } \lesssim \norm{w}_{H^1(\Omega)} (h^q \norm{u}_{q} + \delta).
\]
From Corollary~\ref{cor:conv-tnorm} and interpolation estimates 
\begin{align*}
	\mathrm{I}_6 & \lesssim (\StabNablaNormal{ u_h - \Ext u \circ \Phi_h }{ u_h - \Ext u \circ \Phi_h })^{1/2}  \norm{  ( w - \InterpBackgroundMeshZ  w) \circ \Phi_h }_{\frac{1}{2},h,\Gamma_h } \\
& \lesssim \norm{w}_{H^1(\Omega)} (h^q \norm{u}_{q} + \delta).
\end{align*}
Finally, the perturbations are bounded using continuity of the data extension and stability of the interpolation:
\[
\mathrm{I}_5 \lesssim \norm{ \delta f_h }_{\BrokenOmegah} \norm{  \InterpUnfittedZ w \circ \Phi_h  }_{\Omega} \lesssim \delta \norm{w}_{H^1(\Omega)}. 
\]
This concludes the argument.
\end{proof} 
Finally, we can deduce convergence in the target domain. To this end, we use Corollary~\ref{cor:conditional_stabiliy_estimate}  
which renders the resulting error estimate sensitive to the stability properties of the continuous problem.
\begin{theorem}[$L^2$-error estimate in $B$]\label{thm:L2B-error-estimate}
	Let $u \in H^{q+1}(\BrokenOmega) \cap W^{3,\infty}(\BrokenOmega) $ be the exact solution of (\ref{eq:PDE_bulk})-(\ref{eq:measurements}). 
	Let $(u_h,z_h) \in \curvedFes{\Theta}^{\Gamma} \times \curvedFes{\Theta}^{0} $ be the solution of (\ref{eq:discr_stab_var_form}).
Assume that $ \omega \subset B \subset \Omega$ such that $B \setminus \omega$ does not touch the boundary of $\Omega$.
Then there exists a $\tau \in (0,1)$ such that 
\[
 \norm{ u - u_h \circ \Phi_h^{-1}  }_{ B } \lesssim  h^{\tau q} \left(  \norm{u}_{q} + h^{-q} \delta  \right).
\]
\end{theorem}
\begin{proof}
As explained above, to apply the conditional stability estimate from Corollary~\ref{cor:conditional_stabiliy_estimate} we first have to resolve the issue that $u_h \circ \Phi_h^{-1} \notin H^1(\Omega) $.
To correct for jumps across the interface, we add the function $\varphi$ as defined in (\ref{eq:PDE_jump_correction}) to $u_h \circ \Phi_h^{-1}$, that is 
\[
\bar{u}_h = ( \bar{u}_{h,1},  \bar{u}_{h,2}) := ( u_{h,1} \circ \Phi_h^{-1} + \varphi,  u_{h,2} \circ \Phi_h^{-1} ). 
\]
Using $\Phi_h(\Gamma_h) = \Gamma$ and that $\varphi|_{\Gamma} = - \jump{u_h \circ \Phi_h^{-1} }_{\Gamma}$ we obtain
\begin{align*}
\bar{u}_{h,2}|_{\Gamma} = u_{h,2}|_{\Gamma_h} = u_{h,1}|_{\Gamma_h} - \jump{ u_h  }_{\Gamma_h} 
= u_{h,1}  \circ \Phi_h^{-1}|_{\Gamma} -  \jump{ u_h  \circ \Phi_h^{-1} }_{\Gamma} 
 = \bar{u}_{h,1}|_{\Gamma}. 
\end{align*}
It follows that $\bar{u}_h \in H^1(\Omega)$ and hence Corollary~\ref{cor:conditional_stabiliy_estimate} yields  
\begin{align}
& \norm{ u - u_h \circ \Phi_h^{-1}  }_{ B }  \leq \norm{ u - \bar{u}_h }_{ B} + \norm{\varphi}_{\Omega_1}  \nonumber \\
& \lesssim h^{q} \norm{u}_{q} + \delta + \left( \norm{ u -  \bar{u}_h }_{\omega} + \norm{ \bar{r}_h }_{ H^{-1}(\Omega) } \right)^{\tau} \left( \norm{ u -  \bar{u}_h }_{\Omega} + \norm{ \bar{r}_h }_{ H^{-1}(\Omega) } \right)^{ 1 - \tau} \label{eq:cond-sta-appl}, 
\end{align}
	where  $ \bar{r}_h := \calL (u -  \bar{u}_h) $ and Lemma~\ref{lem:corr_IF_jumps} was applied to estimate $ \norm{\varphi}_{\Omega_1}$. 
We have to estimate the terms on the right hand side of (\ref{eq:cond-sta-appl}). \par 
From Lemma~\ref{lem:corr_IF_jumps} and Lemma~\ref{lem:weak-conv} we obtain for any $w \in H^1_0(\Omega)$ that 
\begin{align*}
a( u - \bar{u}_h ,w  )  &=  -(\mu_1 \nabla \varphi, \nabla w)_{\Omega_1} + (\rho_1 \varphi, w )_{\Omega_1} + a( u - u_h \circ \Phi_h^{-1} , w ) \\ 
	& \lesssim   (h^{q } \norm{u}_{q} + \delta) \norm{w}_{H^1(\Omega)}.
\end{align*}		
Hence, 
\[
\norm{ \bar{r}_h }_{ H^{-1}(\Omega) } =: \sup_{w \in H^1_0(\Omega)}  \frac{a( u - \bar{u}_h ,w  ) }{ \norm{w}_{H^1(\Omega)} } 
\lesssim 
h^{q } \norm{u}_{q} + \delta. 
\]
Using the assumption $\Phi_h|_{\omega} = \ID$, Corollary~\ref{cor:conv-tnorm} and Lemma~\ref{lem:corr_IF_jumps} we have 
\[
\norm{u - \bar{u}_h }_{\omega} \lesssim \norm{ u - u_h \circ \Phi_h^{-1}}_{\omega} + \norm{\varphi}_{\Omega_1}   \sim \norm{ u \circ \Phi_h  - u_h  }_{\omega} + \norm{\varphi}_{\Omega_1}  
\lesssim h^{q } \norm{u}_{q} + \delta. 
\]
By Lemma~\ref{lem:norm-equiv} (a) and Lemma~\ref{lem:corr_IF_jumps}:  
\begin{align*}
\norm{u - \bar{u}_h }_{\Omega}  \lesssim \norm{ u - u_h \circ \Phi_h^{-1}}_{\BrokenOmega} + \norm{\varphi}_{\Omega_1}  \lesssim \norm{ \Ext u \circ \Phi_h - u_h  }_{\BrokenOmegah} + h^{q } \norm{u}_{q} + \delta.
\end{align*}
To control the remaining term, we use the definition of $\StabTikh{\cdot}{\cdot}$ and Corollary~\ref{cor:conv-tnorm}: 
\[
	\norm{ \Ext u \circ \Phi_h - u_h  }_{\BrokenOmegah} \lesssim h^{-q} \StabTikh{ \Ext u \circ \Phi_h - u_h }{ \Ext u \circ \Phi_h - u_h  } \lesssim \norm{u}_{q} + h^{-q} \delta.
\]
Inserting these estimates into (\ref{eq:cond-sta-appl}) yields the claim:
\begin{align*}
	\norm{ u - u_h \circ \Phi_h^{-1}  }_{ B } & \lesssim  h^{q} \norm{u}_{q} + \delta + \left(h^q \left[ \norm{u}_{q} + h^{-q} \delta \right] \right)^{\tau}  \left(  \norm{u}_{q} + h^{-q} \delta  \right)^{1 - \tau} \\
	& \lesssim h^{\tau q} \left(  \norm{u}_{q} + h^{-q} \delta  \right ).
\end{align*}
\end{proof}


\section{Numerical experiments}\label{section:numexp}
We implemented the method in \texttt{ngsxfem} \cite{ngsxfem2021} - an Add-on which enriches the finite element library \texttt{NGSolve} \cite{JS97,JS14} to allow for the use of unfitted discretizations. 
A selection of numerical experiments will be presented to expose strengths and weaknesses of our approach. 
Reproduction material for the presented experiments is avalaible at \texttt{zenodo}\footnote{\url{https://doi.org/10.5281/zenodo.13328144}} \cite{BP23} in the form of a docker image. 
\par 
Throughout the experiments we use the levelset functions
\begin{equation}\label{eq:lset_numexp}
\phi =  \norm{x}_{\ell}  - 1, \qquad \norm{x}_{\ell} := \left(\sum\limits_{j=1}^d (x_j)^{\ell}\right)^{1/{\ell}}, \qquad \text{for } x \in \mathbb{R}^d,   
\end{equation}
for $\ell \in \{2,4\}$ to represent the geometry.
We work on a sequence of quasi-uniform simplicial meshes which are not adapted to the interface. Linear systems are solved using a sparse direct solver. \par 
Our investigation will focus on the influence of the interface on the numerical solution of the unique continuation problem. Recall that the  
H\"older exponent $\tau$ and the constant $C$ in the conditional stability estimate (\ref{eq:cond_stab}) depend in an unknown manner on the interface and 
the coefficients of the PDE. They also depend on certain convexity properties of the data and target domain as has been thoroughly 
investigated in reference \cite{BNO19} for the constant coefficient Helmholtz equation. Here we will only briefly touch on this aspect since our work 
is mainly devoted to the study of the material interface.

\subsection{Pure diffusion in dimension $d=2$}\label{ssection:numexp-pure-diffusion}

We will start with a purely diffusive problem, i.e.\ $\rho_i = 0, i=1,2$, since it is well-known that unique continuation for 
the Helmholtz case can be particularly challenging, see e.g. \cite{BNO19,BDE21,BP22}. 
We consider the levelset function given in (\ref{eq:lset_numexp}) for $\ell = 4$, where the objective is to continue the solution given in $\omega = [-0.5,0.5]^2$ across the interface into $B = [-1.25,1.25]^2$. A sketch of the geometrical configuration is provided in Figure~\ref{fig:squares}.
We will consider noise-free data, i.e.\ $\delta = 0$ in (\ref{eq:delta-def}), sampled from the exact solution 
\begin{equation}\label{eq:refsol-diffusion}
u =  \begin{cases}  
   \frac{1}{\sqrt{2}} \left( 1 + \pi  \frac{ \mu_1 }{ \mu_2 } \right) - \cos\left( \frac{\pi \norm{x}_4^4 }{ 4 }  \right)  & \text{ in } \Omega_1, \\ 
    \frac{ \mu_1 \pi  }{ \mu_2 \sqrt{2} }  \norm{x}_4     & \text{ in } \Omega_2, 
\end{cases}
\end{equation}

\begin{figure}[htbp]
\centering
\begin{subfigure}[b]{0.3\textwidth}
\centering
\includegraphics[scale=0.9]{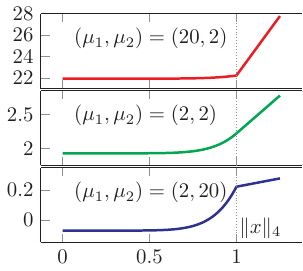}
\subcaption{ Reference solution.  }%
\label{fig:nonconvex-refsol}
\end{subfigure}
\begin{subfigure}[b]{0.3\textwidth}
\centering
\includegraphics[scale=0.135]{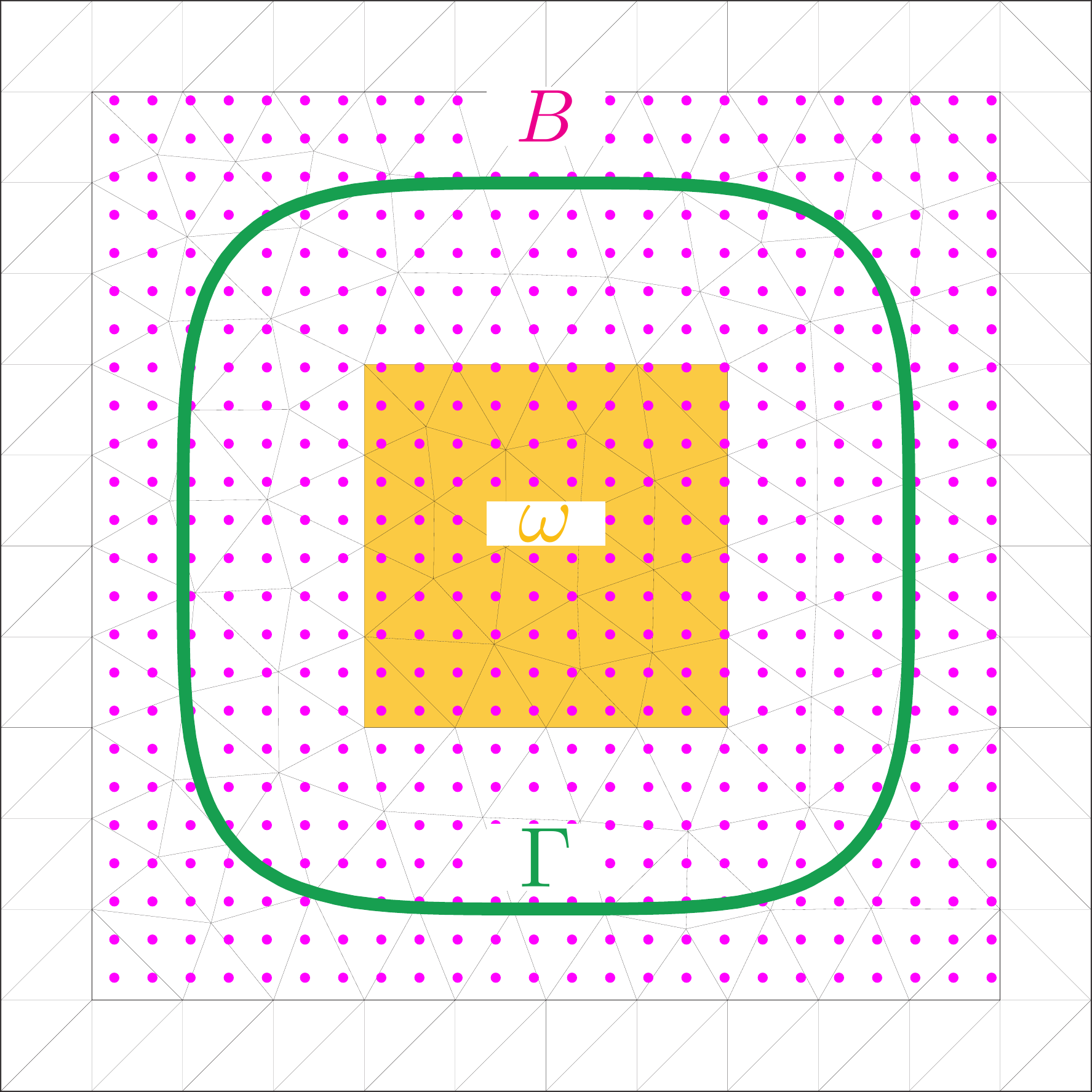}
\subcaption{Setup: Pure diffusion.}%
\label{fig:squares}
\end{subfigure}
\begin{subfigure}[b]{0.3\textwidth}
\centering
\includegraphics[scale=0.135]{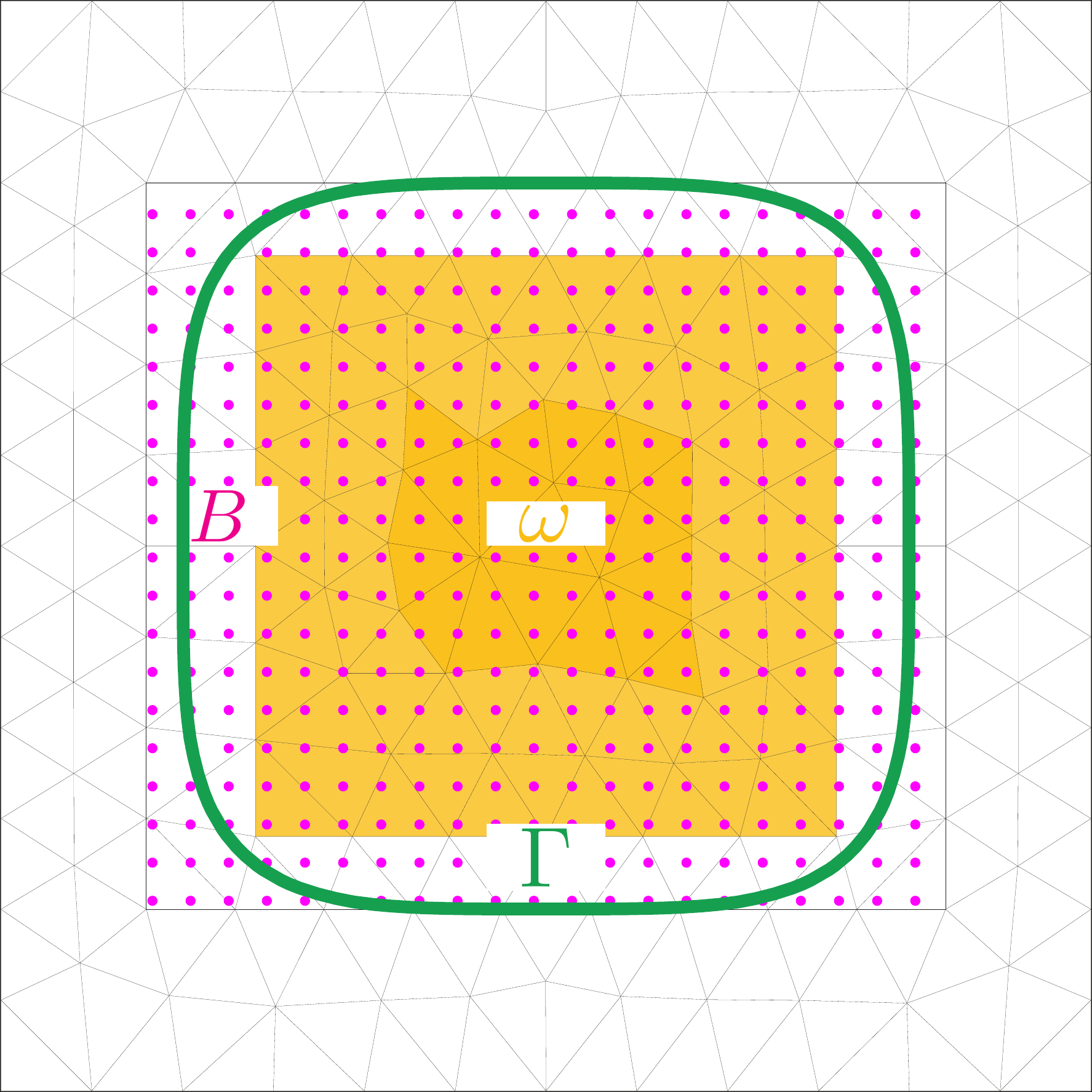}
\subcaption{Setup: Helmholtz.}%
\label{fig:squares-easy}
\end{subfigure}
	\caption{Figure (A) shows the reference solution (\ref{eq:refsol-diffusion}) in the purely diffusive case as a function of $\norm{x}_4$ for different levels of the contrast. Figure (B) and (C) display a sketch of the geometrical setup for the experiments of Section \ref{ssection:numexp-pure-diffusion}, respectively \ref{sssection:Helmholtz-contrast}. }
\label{fig:nonconvex-diffusion}
\end{figure}
which is shown in Figure~\ref{fig:nonconvex-refsol} as a function of $\norm{x}_4$ for different 
levels of the contrast. Note that a kink occurs at the interface for $\mu_1 \neq \mu_2$. \par
\begin{figure}[htbp]
\centering
\includegraphics[width=\textwidth]{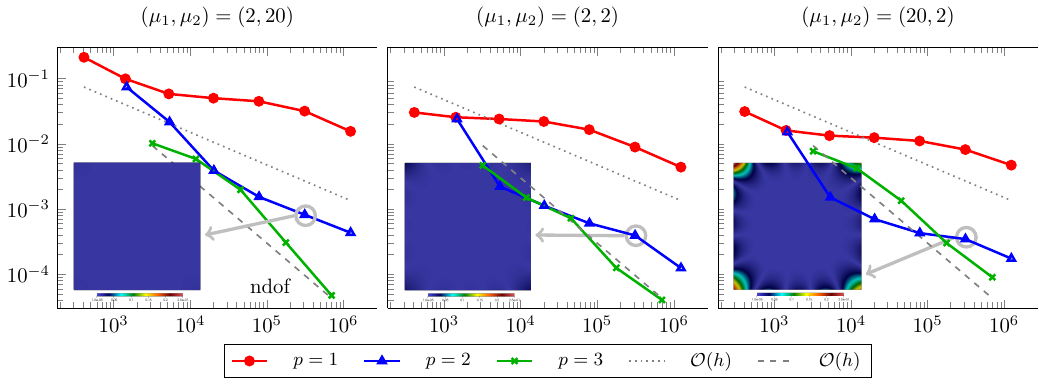}
	\caption{Dependence of the relative $L^2$-error in $B$ on the contrast for a pure diffusion problem in the geometry shown in Figure~\ref{fig:squares}.
	In the insets the absolute error for the cases indicated by the arrows is shown. Here we used $q=p$. On the $x$-axis the number of degrees of freedom `ndof' is displayed.} 
\label{fig:non-convex-diffusion}
\end{figure}
The relative $L^2$ errors in $B$ for different levels of the contrast are shown in Figure~\ref{fig:non-convex-diffusion}.
Here we have set the weights in the numerical flux to $\kappa_1 = \mu_2 / (\mu_1 + \mu_2)$ 
and $\kappa_2 = \mu_1 / (\mu_1 + \mu_2)$, which is a popular choice in the literature \cite{D03,BZ06,ESZ09} for contrast problems. 
We observe that changing the contrast by one order of magnitude seems to have little influence on the convergence rate in 
the purely diffusive regime. The setup with $(\mu_1,\mu_2) = (20,2)$ appears to be slightly more challenging than the others as 
the absolute error is observed to be high near the corners of the domain.
Nevertheless, we still obtain reasonably accurate solutions in the target domain after a few refinement steps. 
\par 
As an additional experiment let us investigate the importance of the stabilization terms. 
In numerical experiments it is common practice to rescale these terms by some positive constants 
to optimize the preasymptotic convergence behavior. 
The stabilization terms $\StabCIP{\cdot}{\cdot}$ and $\StabGLS{\cdot}{\cdot}$  are a standard ingredient 
of methods based on the framework from \cite{B13}.
Parameter studies on the optimal choice of the corresponding scaling parameters can be found in the literature, see e.g.\ \cite{B13,BHL18a,N20,BP22}.
Here, we focus on the novel terms which appeared in this work, i.e.\ 
the choice of $\gamma_{\mathrm{IF}}$ for the interface stabilization in (\ref{eq:IF-stab-comb}) 
and the choice of $\alpha_2$ in (\ref{eq:def-Tikh}) which we introduced to counter propagation of geometric errors.
The dependence of the relative $L^2$-error on the size of these parameters is displayed in Figure~\ref{fig:non-convex-diffusion-Stab}.
This experiment has been carried out on the second finest mesh used for the previous convergence study.
We discuss the results of the different stabilizations separately:
\begin{figure}[htbp]
\centering
\includegraphics[width=\textwidth]{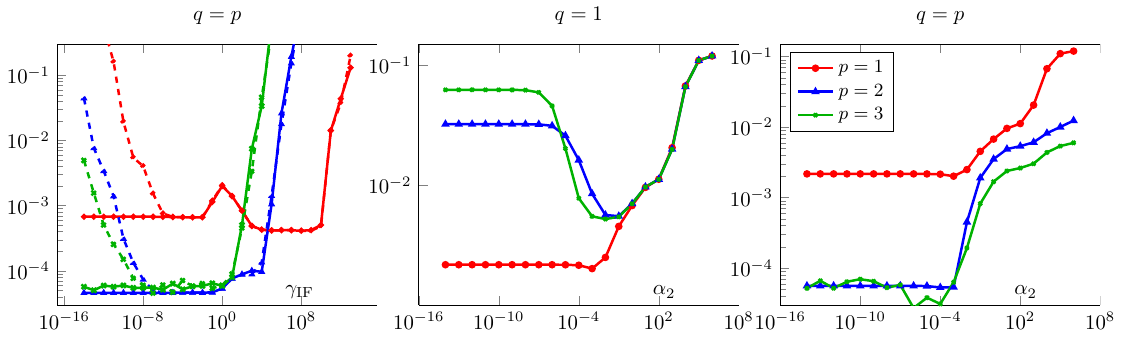}
	\caption{Dependence of the relative $L^2$-error in $B$ in terms of the choice of stabilization parameters. We consider a pure diffusion problem with $(\mu_1,\mu_2) = (20,2)$ in the geometry shown in Figure~\ref{fig:squares}. The solid lines in the left plot show the results for the method as defined and analyzed in this paper, while dashed lines display the results for an alternative method in which the term $N_h^c$ in the bilinear form is omitted.} 
\label{fig:non-convex-diffusion-Stab}
\end{figure}
\begin{itemize}
	\item The solid lines in the left plot of Figure~\ref{fig:non-convex-diffusion-Stab} show that our method appears to be stable and accurate even when the stabilization on the interface is completely omitted, i.e.\  $\gamma_{\mathrm{IF}}=0$. This is in fact thanks to the presence of the Nitsche term $N_h^c$ in the bilinear form $A_h$, see (\ref{eq:Nitsche-terms-def})-(\ref{def:A-bil}). If the latter term is omitted (dashed lines), the error increases dramatically as $\gamma_{\mathrm{IF}}$ goes to zero. 
Note that dropping $N_h^c$ means to sacrifice adjoint consistency of the Nitsche formulation, which seems to be of importance here.  
\item If the geometry is sufficiently resolved ($q=p$) it appears that $\alpha_2$ can be set to zero as well. This is reasonable since we introduced this stabilization to counter geometrical errors that are of order $\mathcal{O}(h^q)$. However, when the geometry is underresolved (central plot with $q=1$) these geometric errors become significant and we indeed observe that the error in the target domain increases noticably when $\alpha_2$ is chosen too small and $p > 1$. For $p=1$ on the other hand,  we can always set $\alpha_2 =0$ as the discrete Poincar\'{e} inequality given in \cite[Lemma 2]{BHL18a} ensures that the gradient weighted with $h$ is already controlled by the other stabilization terms.
\end{itemize}

\subsection{Helmholtz problem in dimension $d=2$}\label{ssection:numexp-Helmholtz}

Let us now tackle the Helmholtz problem while keeping the levelset function from Section~\ref{ssection:numexp-pure-diffusion}.
Here, we set $\rho_i = k_i^2 $ where $k_i >0$ is the wavenumber in the $i$-th subdomain and use the reference solution 
\begin{equation}\label{eq:refsol-oscillatory}
u =  \begin{cases}  
    C_1 \cos( k_1 \norm{x}_4^4 ) + C_2   & \text{ in } \Omega_1, \\ 
     \sin( k_2 \norm{x}_4^4  )  & \text{ in } \Omega_2, 
\end{cases}
\end{equation}
for 
\begin{equation}\label{eq:refsol_Helmholtz_consts} 
C_1 := - \frac{k_2 \mu_2}{  k_1 \mu_1} \frac{ \cos(k_2) }{ \sin(k_1)  }  \text{ and } C_2 := \sin(k_2) - C_1 \cos(k_1). 
\end{equation}

\subsubsection{Dependence on wavenumber and contrast ratio}\label{sssection:Helmholtz-contrast}
For the Helmholtz case we consider the geometrical configuration shown in Figure~\ref{fig:squares-easy}. 
This setup is slighly less challenging than the one considered in the purely diffusive case as the data domain 
has been extended to $\omega = [-0.8,0.8]^2$ whereas the target domain was reduced to $B = [-1.1,1.1] \times [-1.0,1.0]$.
The reason for this simplification becomes immediately clear when comparing the plots of the absolute errors for the Helmholtz case
shown in the insets of Figure~\ref{fig:ball-4-norm-squares-Helmholtz-contrast-rev} with the ones for the purely diffusive case 
given in Figure~\ref{fig:non-convex-diffusion}. Increasing the wavenumber $k_2$ in $\Omega_2$ by a factor of six increases the 
error close to the boundary of the domain enormously. Whereas convergence rates of\footnote{Here we use $p=q$.} $\mathcal{O}(h^{  \tau (q+1) })$ 
with $\tau$ close to one are observed for $k_2 = 1$, the rate does not exceed second order for $k_2 = 6$. 
The stability is even so poor that polynomial order $p=3$ does not perform better than order $p=2$. 
Notice that the performance is equally poor for a contrast ratio of ten, i.e.\ for $(\mu_1,\mu_2) = (2,20)$ as for the case
of no contrast $(\mu_1,\mu_2) = (2,2)$. The realative errors are even a bit lower in the former case, presumably because an increase of $\mu_2$
reduces the effective wavenumber in subdomain $\Omega_2$. In conclusion, this experiment indicates that for the considered setup the wavenumber 
outside the data domain is the decisive factor determining the stability of the problem.

\begin{figure}[htbp]
\centering
\includegraphics[width=\textwidth]{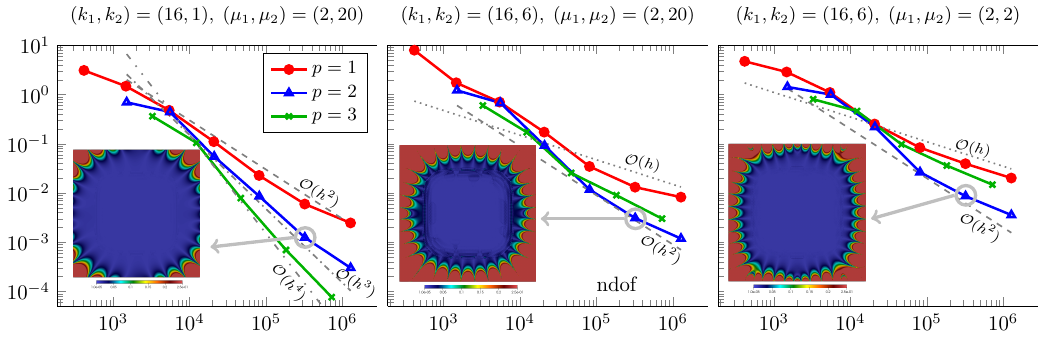}
\caption{Dependence of the relative $L^2$-error in $B$ on the coefficients for a Helmholtz problem in the geometry shown in Figure~\ref{fig:squares-easy}.
In the insets the absolute error for the cases indicated by the arrows is shown. Here we used $q=p$.  } 
\label{fig:ball-4-norm-squares-Helmholtz-contrast-rev}
\end{figure}

\subsubsection{Convex case}\label{sssection:Helmholtz-convex}

As mentioned before, the stability of unique continuation problems with constant coefficients is known to be linked to the geometry of the data and target domains, see e.g.\ numerical experiments in \cite{BNO19,BDE21,BP22}. 
Error estimates that are robust in the wavenumber have been proven in reference \cite{BNO19} when the target domain is contained in the convex hull of the data domain. We consider such a case here by setting
\begin{align*}
 \Omega = [-1.5,1.5]^2, \quad 
\quad B =  \Omega \setminus  [-1.5, 1.5] \times [1.25,1.5], \quad
\omega = B \setminus [-1.25,1.25]^2
\end{align*}
and using the same levelset function as before. 
The geometrical configuration and the coarsest mesh for our convergence studies are shown in Figure~\ref{fig:convex-geom}.
Note that we cannot expect the results from \cite{BNO19} to carry over directly to our setting since they apply to a problem with constant coefficients only.
Here we consider the case of the Helmholtz equation with jumping coefficients and
 reference solution (\ref{eq:refsol-oscillatory}). The inset of Figure~\ref{fig:convex-Helmholtz-noise} shows a plot of this function for parameters 
$k_1 = 16, k_2 = 2, \mu_1 = 1, \mu_2 = 2 $. 
\begin{figure}[htbp]
\centering
\begin{subfigure}[b]{.55\textwidth}
\centering
\includegraphics[width=1.1\textwidth]{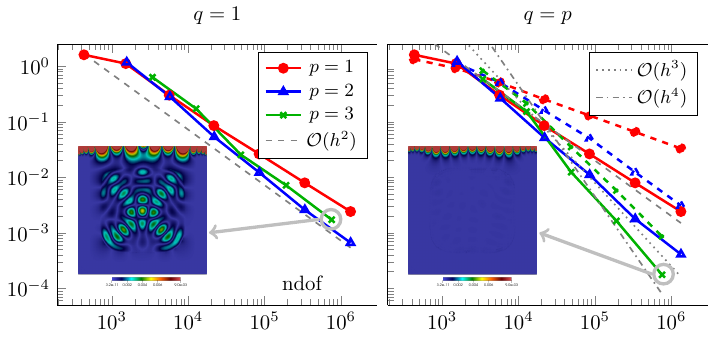}
\subcaption{Relative error for $(k_1,k_2) = (16,2), (\mu_1,\mu_2) = (1,2)$.}
\label{fig:convex-relerr-exact-data}
\end{subfigure}
\quad
\begin{subfigure}[b]{.4\textwidth}
\centering
\includegraphics[scale=0.14]{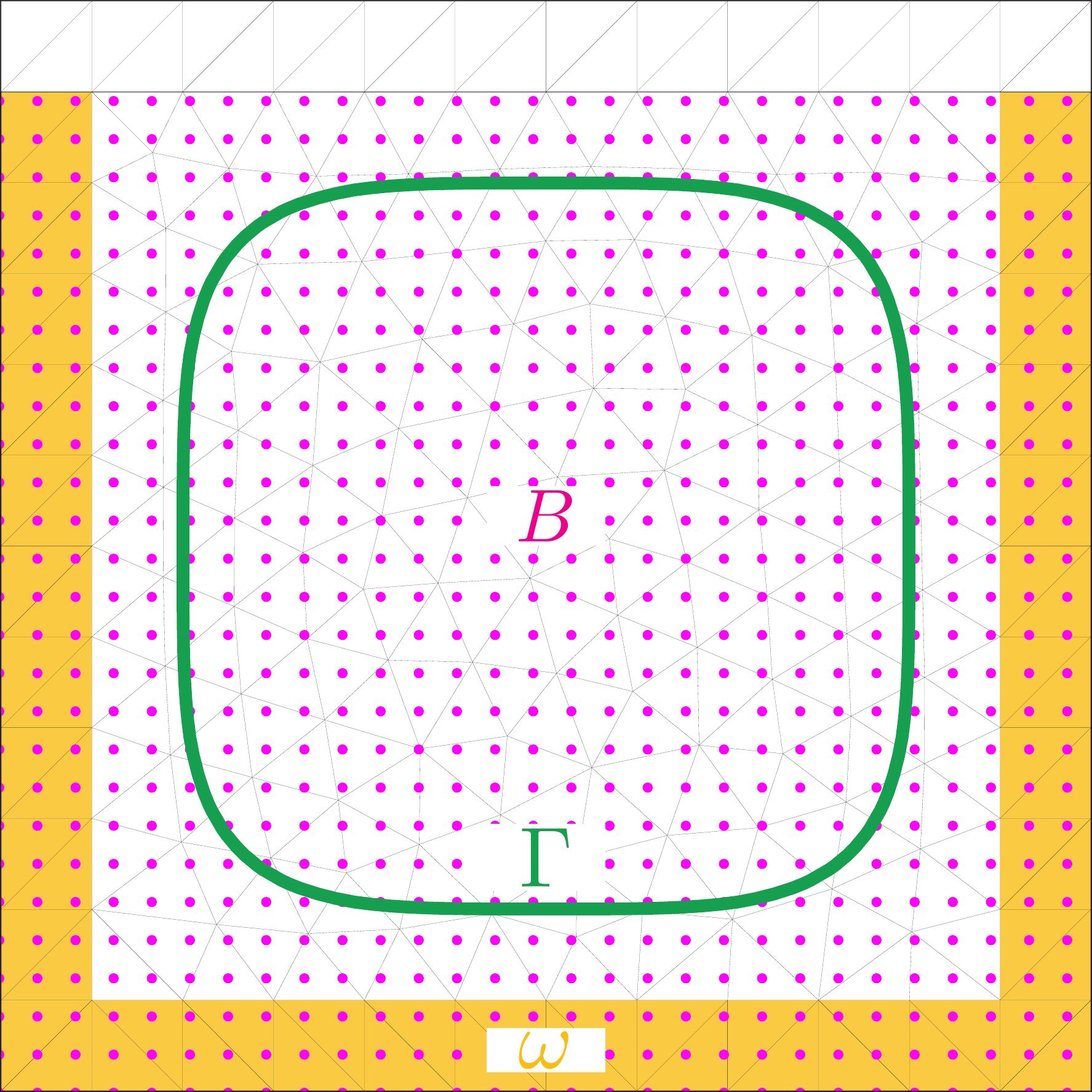}
\subcaption{Geometry.}
\label{fig:convex-geom}
\end{subfigure}
\begin{subfigure}[b]{.99\textwidth}
\centering
\includegraphics[width=.95\textwidth]{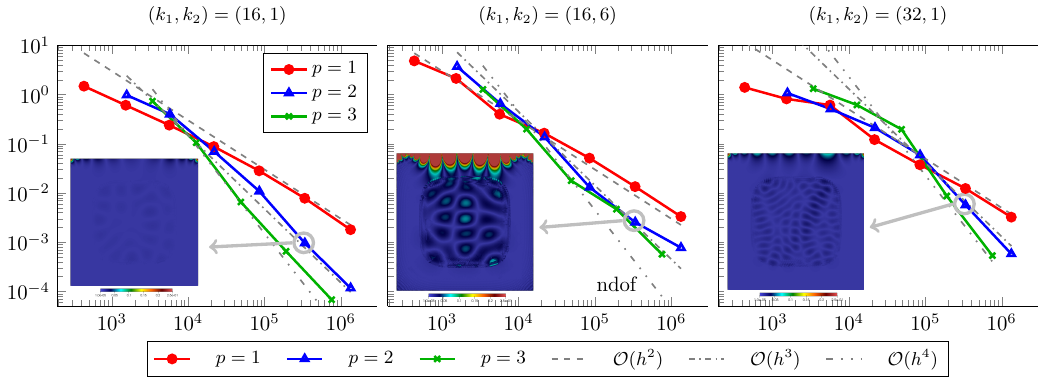}
\subcaption{Dependence of relative error on the wavenumber for $(\mu_1,\mu_2) = (2,20)$ using $q=p$.}
\label{fig:ball-4-norm-convex-Helmholtz-wavenumber}
\end{subfigure}
\caption{ Results for a Helmholtz problem using exact data. 
Solid and dashed lines in (A) show the relative $L^2$, respectively $H^1$-semi error, in the target domain $B$. In the insets the absolute error for the cases indicated by the arrows is shown. } 
\label{fig:convex-Helmholtz}
\end{figure}
\paragraph{Influence of geometry approximation}
In Figure~\ref{fig:convex-relerr-exact-data} the relative $L^2$-errors in the target domain under mesh refinement for a poorly resolved geometry ($q=1$) are compared with 
the ones based on a high-order approximation ($q=p$).  
For $q=1$ the convergence for any polynomial degree $p$ is limited by the geometrical error, i.e.\
$\dist(\Gamma, \Gamma_h) \lesssim \mathcal{O}(h^{2})$, introduced by the piecewise linear 
approximation of the interface. 
Interestingly the absolute error displayed in the inset of Figure~\ref{fig:convex-relerr-exact-data} (left) 
looks like a spurious resonance and a strong pollution on part of the boundary where no data is given can be observed. 
When increasing the resolution of the geometry to $q=p$ (central plot in Figure~\ref{fig:convex-Helmholtz}) these effects are reduced significantly.
We observe a convergence rate of $\mathcal{O}(h^{\tau q})$ in the $H^1$-semi norm with $\tau$ close to one and $\mathcal{O}(h^{  \tau (q+1) })$ in the $L^2$-norm similar 
as in the experiment shown in left panel of Figure~\ref{fig:ball-4-norm-squares-Helmholtz-contrast-rev}.
This is better than predicted by our theoretical results, which in particular do not guarantee convergence in the $H^1$-semi norm at all. 
This is because the conditional stability estimate (\ref{eq:cond_stab}) merely yields
control of the $L^2$-norm in the target domain. 
If such an estimate giving control over the $H^1$-norm was available, as e.g.\ derived in 
\cite[Corollary 3]{BNO19} for the Helmholtz equation \textit{without interfaces} in a specific convex geometric setting, then inspection of 
the proof of Theorem~\ref{thm:L2B-error-estimate} shows that rates of  $\mathcal{O}(h^{\tau q})$ in 
the $H^1$-norm would follow. \par 
\paragraph{Influence of wavenumber} 
Figure~\ref{fig:ball-4-norm-convex-Helmholtz-wavenumber} shows the effect of varying the wavenumber $(k_1,k_2)$ while keeping $(\mu_1,\mu_2)$ fixed. 
For $(k_1,k_2) = (16,1)$ we observe a rate of $\mathcal{O}(h^{\tau (q+1)})$ with $\tau$ close to one. 
Increasing $k_1$ to $32$ introduces a preasymptotic regime (up to about $10^5$ degrees of freedom) 
in which a reduced rate is observed but does not change the asymptotic rate.
However, increasing $k_2$ to $6$ appears to impact the asymptotic rate negatively and leads to a large increase of the absolute error at the top of the domain.
It seems that spurious modes from the top of the domain creep into the convex hull and are reflected at the interface. 
Even though we still observe a rate of about $\mathcal{O}(h^3)$ for $p=q=3$ when measuring in the convex hull which is significantly better 
than in the central plot of Figure~\ref{fig:ball-4-norm-squares-Helmholtz-contrast-rev}, where we sought to continue the solution outside the convex hull of the data,
the full robustness as for the constant coefficient problem \cite{BNO19} is not valid. We expect that for variable coefficient problems it will be necessary to strengthen the convexity condition describing the geometry of the target domain in which robust unique continuation is possible. In particular, in the considered setting this condition has to depend on the constrast of the wavenumber.

\paragraph{Perturbed data}
Let us now consider perturbed data of strength $\delta =  \tilde{\delta}_p  h^{p-\theta}$ for some $\theta \in \{0,1,2\}$ and with $\tilde{\delta}_p $ independent of $h$. In the implementation these perturbations are realized by populating the degrees of freedom of a finite element function with random noise and subsequent normalization. 
Setting $q=p$ the error estimate of Theorem~\ref{thm:L2B-error-estimate} predicts the following behavior:
\begin{equation}\label{eq:L2error-noise-theta}
\norm{ u - u_h \circ \Phi_h^{-1}  }_{ B } \lesssim h^{\tau p} \left( \norm{u}_{p} + \tilde{\delta}_p h^{-\theta} \right).
\end{equation}
Two different regimes can be distinguished:
\begin{figure}[htbp]
\centering
\includegraphics[width=\textwidth]{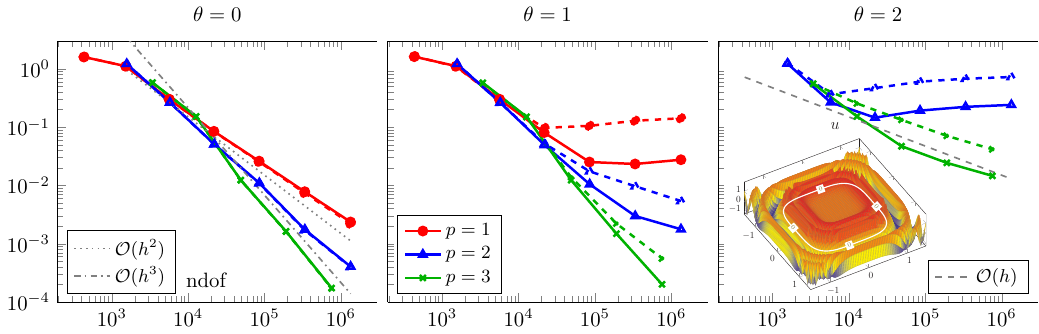}
\caption{ Results for a Helmholtz problem with $(k_1,k_2) = (16,2), (\mu_1,\mu_2) = (1,2)$ in the geoemetry of Figure~\ref{fig:convex-geom} (B) using perturbed data with $\delta =  \tilde{\delta}_p  h^{p-\theta}$.   
	Solid lines: $\tilde{\delta}_p \in [1,8,28]$ for $p \in [1,2,3]$, dashed line $\tilde{\delta}_p \in [5,24,80]$. All experiments use $q=p$.}
\label{fig:convex-Helmholtz-noise}
\end{figure}
\begin{itemize}
\item As long as $ \norm{u}_{p} h^{\theta} > \tilde{\delta}_p $ the first term in (\ref{eq:L2error-noise-theta}) will dominate. Then we expect to see the same convergence behavior as in Figure~\ref{fig:convex-relerr-exact-data} with exact data.
\item When $h$ is small enough such that $ \norm{u}_{p} h^{\theta} < \tilde{\delta}_p $, then the second term takes over and the rate deteriorates to $\mathcal{O}(h^{\tau p-\theta})$.
\end{itemize}
The numerical results shown in Figure~\ref{fig:convex-Helmholtz-noise} confirm this behavior. 
The refinement level at which the transition from one regime to the other occurs can be shifted
by scaling the coefficient $\tilde{\delta}_p$ of the noise as comparison of the solid and dashed 
lines in Figure~\ref{fig:convex-Helmholtz-noise} demonstrates.

\subsection{Fitted vs Unfitted methods in dimension $d=3$}\label{ssection:fitted-vs-unfitted}
Finally, we would like to provide a comparison between the unfitted method proposed in this work and a more traditional approach in 
which the mesh is fitted to the interface. The purpose of this comparison is to show that both methods yield similar results and are affected 
in the same way when the stability of the continuous problem deteriorates. In the fitted case we use the method from \cite{BNO19} which 
is easily generalized to allow for high-order finite elements and variable coefficients, see also \cite{BP22} for an application to an elastodynamics problem.  
\begin{figure}[htbp]
\centering
\includegraphics[width=\textwidth]{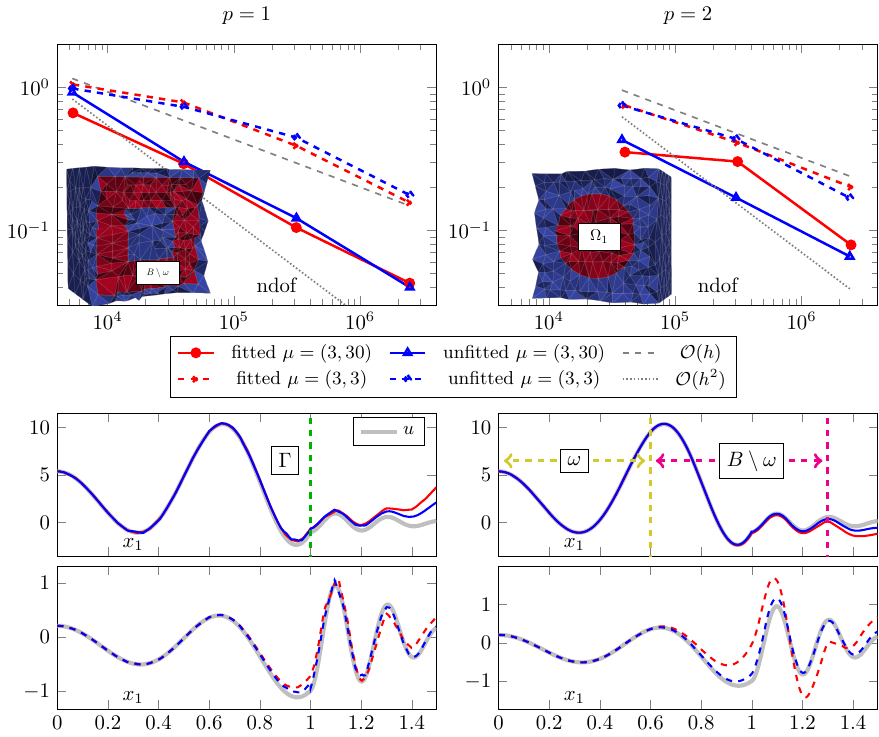}
\caption{ The upper panel compares the relative $L^2$ error in the target domain for the fitted and unfitted method 
for two different choices of $(\mu_1,\mu_2)$. The lower panel shows the continuous solution and numerical approximations as a function of $(x_1,0,0)$ evaluated 
on the finest mesh considered for the respective polynomial degree.
Note that the left panel presents the results for $p=1$, whereas the right panel shows $p=2$.
We use $p=q$ for the unfitted method and accordingly employ classical curved elements of order $p$ for the fitted method to improve the approximation of the geometry.
	} 
\label{fig:ball-3d-fitted-unfitted-contrast}
\end{figure}
For the geometry let us consider the unit cube $\Omega = [1.5,1.5]^3$ with data domain $\omega = [-0.6,0.6]^3$ and target domain $B = [-1.3,1.3]^3$. The interface is 
chosen to be the unit ball, i.e.\ levelset function (\ref{eq:lset_numexp}) with $\ell = 2$. A mesh which is fitted to 
this interface is displayed in the right inset of Figure~\ref{fig:ball-3d-fitted-unfitted-contrast}, whereas an unfitted mesh 
can be seen on the left. As in the previous sections we consider the case of the Helmholtz equation and use the following reference solution 
\[
u =  \begin{cases}  
    (C_1 \cos( k_1 r ) + C_2) \exp(-(r^2-1)^2)    & \text{ in } \Omega_1, \\ 
    \sin( k_2 r ) \exp(-(r^2-1)^2)   & \text{ in } \Omega_2, 
\end{cases}
\]
where $ r = \norm{x}_2$ and with $C_1$ and $C_2$ as in (\ref{eq:refsol_Helmholtz_consts}). 
We consider the fixed wavenumber $(k_1,k_2) = (10,30)$ and the two cases $(\mu_1,\mu_2) = (3,30)$ and $(\mu_1,\mu_2) = (3,3)$. 
Recall from Section~\ref{ssection:numexp-Helmholtz} that decreasing $\mu_2$ was observed to render the problem more challenging to 
solve. A plot of the reference solution as a function of the first coordinate, i.e.\ $r=\abs{x_1}$, is provided in the bottom panel of Figure~\ref{fig:ball-3d-fitted-unfitted-contrast}. \par 
The numerical results for piecewise affine and quadratic approximation are presented in the same figure.
In the upper panel of the figure we can clearly see that the fitted and unfitted method (red vs. blue lines) yield very similar convergence rates. 
Decreasing $\mu_2$ from $30$ to $3$ (solid vs. dashed lines) increases the error for both methods by the same amount. This indicates that the fitted and 
unfitted approach display the same sensitivity w.r.t.\ deteriorating stability of the continuous problem. 
Hence, the unfitted method is as suited as the fitted method for solving the unique continuation problem. In fact, the qualitative approximation 
of the function obtained with the unfitted method can in some cases even be better than the one obtained with the fitted method as 
the lower panel of Figure~\ref{fig:ball-3d-fitted-unfitted-contrast} demonstrates. \par 
As an outstanding issue for both approaches we mention the immense computational cost of solving the arising linear system with a direct solver, 
which limits the number of considered refinement levels. Actually, for the method with $p=2$ it is not clear whether the 
asymptotic convergence rate has been reached yet. Iterative solution approaches are hampered by the ill-conditioning of the linear systems 
which reflects the ill-posed character of the corresponding continuous problems. Recently, several interesting approaches to construct efficient preconditioners 
\cite{MNN17,AVBG17,ES09} for related problems have been proposed whose suitability for our setting would be an interesting question to investigate in future work.

\section{Conclusion}\label{section:conclusion} 
We have investigated an unfitted FEM for unique continuation across an interface. 
To establish error estimates in the unfitted setting, we had to introduce additional stabilization
terms on the interface and a weak Tikhonov penalization of the gradient in the bulk. 
The benefit of the latter has been observed in numerical experiments featuring geometrically 
underresolved regimes in which the distance between $\Gamma$ and $\Gamma_h$ is fairly large. 
At the analytical level we were able to understand the need for this particular stabilization by 
adopting a technique introduced in \cite{L_CMAME_2016,LR17} which allows for explicit control and analysis of geometric errors. It would be interesting to investigate whether our method and its analysis could be combined with other approaches for geometry approximation and numerical integration on unfitted geometries which have been proposed in the literature. In principle, any technique that can reduce the geometric error down to the level of the discretization error in a robust\footnote{In particular, guaranteeing positiveness of quadrature weights.}, accurate and (desirably) efficient way should be suitable.   \par
In numerical experiments we focused on the influence of the heterogeneity in the coefficients on our ability to extend the solution beyond the data domain. 
For Helmholtz problems we observed that the size of the wavenumber, particular outside the data domain, is the decisive factor that determines the 
stability of the problem. Contrasts in the diffusion coefficients of about one order of magnitude, on the other hand, appear to have a much smaller 
influence on the stability. In fact, increasing the contrast in the $\mu_i$ was even observed to be beneficial in some case, see e.g.\ Figure~\ref{fig:ball-3d-fitted-unfitted-contrast}. Moreover, we observed in Section~\ref{sssection:Helmholtz-convex} that to some extent an increased stability in parts of the convex hull of the data domain 
appears to hold. However, the full robustness inside the convex hull as shown in \cite{BNO19} for the homogeneous medium does not directly carry over to the interface problem.

\appendix

\section{Conditional stability estimate}\label{A:cond-stab}

From \cite[Theorem 1.1]{CW20} we get the following result: 
Let $u \in H^1(\Omega)$ be a solution of (\ref{eq:PDE_bulk}) where $f \in H^{-1}(\Omega)$ 
is of the form 
\begin{equation}\label{eq:f_form}
\langle f, v \rangle = (f_0,v)_{\Omega} + (F,\nabla v)_{\Omega} \text{  with  } \norm{f_0}_{ \Omega } + \norm{F}_{ \Omega} \leq \varepsilon, 
\end{equation}
where we use the common notation $\langle \cdot, \cdot \rangle$ to denote the duality bracket on $H^{-1} \times H_0^1$.
Then there exists $\bar{r} > 0$ such that if $ 0 < r_1 < r_2 < r_3 < \bar{r}$, $x_0 \in \Omega, \dis(x_0,\partial \Omega) > r_3$,  
there exist $C > 0, 0 < \tau < 1$ such that 
\begin{equation}\label{eq:three_ball_ieq}
\norm{u}_{ B_{r_2}(x_0) } \leq C \left( \norm{u}_{ B_{r_1}(x_0)} + \varepsilon \right)^{\tau}
\left( \norm{u}_{  B_{r_3}(x_0) } + \varepsilon \right)^{1-\tau}. 
\end{equation}
Here, $B_{r_j}(x_0)$ denotes the ball around $x_0$ with radius $r_j, j=1,2,3$.
We note that any $f \in H^{-1}(\Omega)$ with $\norm{f}_{H^{-1}(\Omega)} \leq \varepsilon /2$ can be written in the form (\ref{eq:f_form}).
Indeed, let $\phi \in H^1_0(\Omega) $ be the solution of the auxiliary problem  
\[ -\Delta \phi + \phi = f \in H^{-1}(\Omega).
\]
Then clearly,
\[
\langle f, v \rangle = (\phi,v)_{\Omega} + (\nabla \phi,\nabla v)_{\Omega} 
\]
for all $v \in H_0^1(\Omega)$.
So we can take $f_0 = \phi$ and $F = \nabla \phi$. 
Further, 
\[
\norm{f}_{ H^{-1}(\Omega) } 
= \sup_{v: \norm{v}_{ H^1(\Omega) = 1}  } \langle f, v \rangle 
= \sup_{v: \norm{v}_{ H^1(\Omega) = 1}  }  (\phi,v)_{\Omega} + (\nabla \phi,\nabla v)_{\Omega} 
	= \norm{\phi}_{ H^1(\Omega) }.
\]
Hence, 
\[
\norm{f_0}_{ \Omega } + \norm{F}_{ \Omega} 
\leq 2 \norm{\phi}_{ H^1(\Omega) }
= 2 \norm{f}_{ H^{-1}(\Omega) } 
	\leq \varepsilon.
\]
Thus (\ref{eq:three_ball_ieq}) yields
\[
\norm{u}_{ B_{r_2}(x_0) } \leq C \left( \norm{u}_{ B_{r_1}(x_0)} + \norm{\calL u}_{H^{-1}(\Omega)} \right)^{\tau}
\left( \norm{u}_{  B_{r_3}(x_0) } + \norm{\calL u}_{H^{-1}(\Omega)}  \right)^{1-\tau}. 
\]
The general case for $\omega \subset B \subset \Omega$ such that $B \setminus \omega$ does not touch the boundary stated in Corollary~\ref{cor:conditional_stabiliy_estimate} then follows by using a covering argument, see \cite[Section 5]{ARRV09} or also \cite{R91}. In this argument one essentially reaches a point in the target domain $B$ by using a sequence of balls. If the point in question happens to be contained in the convex hull of the data domain, then several paths to reach it are possible which have similar stability properties. On the other hand, for a point outside the convex hull only one path with the best stability properties exists.

\section{Proofs involving the isoparametric mapping}\label{B:geom}

\begin{ProofOf}{Lemma~\ref{lem:Phi-HigherOrderBounds}} 
\begin{enumerate}[label=(\alph*)] 
\item From the construction of the mesh transformation, see \cite[Section 3]{LR17}, we recall that  \[
\Psi - \Theta_h  
 =   \begin{cases}  
   \Psi^{\Gamma} - \Theta_h^{\Gamma}   & \text{ on } \Cutdom, \\ 
   \ExtMeshTrafo^{\partial \Cutdom } ( \Psi^{\Gamma} - \Theta_h^{\Gamma} ) & \text{ on }  \CutdomExt \setminus \Cutdom, \\
	 0 & \text{ on } \Omega \setminus \CutdomExt. 
\end{cases}
\]	
Here, $\Psi^{\Gamma}$ and $\Theta_h^{\Gamma}$ are local mappings on the cut elements $\Cutdom$ 
which are extended to a slightly larger domain $\CutdomExt$ by a certain extension operator $\ExtMeshTrafo^{\partial \Cutdom }$, which according to \cite[Theorem 3.11]{LR17} fulfills
\begin{equation}\label{eq:bounds-ExtMeshTrafo}
\max_{ T \in  \CutelExt \setminus \Cutel} \norm{ D^n \ExtMeshTrafo^{\partial \Cutdom }  w } _{\infty,T} 
\lesssim \max_{F \in \mathcal{F}(\partial  \Cutdom) } \sum\limits_{r=n}^{q+1} h^{r-n} \norm{D^r w }_{\infty,F}, \quad n=2,\ldots,q+1. 
\end{equation}
Here, $\mathcal{F}(\partial  \Cutdom)$ denotes the set of all edges ($d=2$) or faces ($d=3$) in $\partial  \Cutdom $ and  $\CutelExt \setminus \Cutel$ the set of elements convering the extended region $\CutdomExt \setminus \Cutdom$.
We also have from \cite[Lemma 3.7]{LR17} the estimate 
\begin{equation}\label{eq:bound-Psi-Theta}
\sum\limits_{r=0}^{q+1} h^r \max_{ T \in \Cutel} \norm{D^r (\Psi^{\Gamma} - \Theta_h^{\Gamma})}_{\infty,T} \lesssim h^{q+1}     
\end{equation}
which already proves the claim (a) on $\Cutdom$. It remains to consider 
\begin{align*}
& \max_{ T \in  \CutelExt \setminus \Cutel} \norm{D^l \ExtMeshTrafo^{\partial \Cutdom } ( \Psi^{\Gamma} - \Theta_h^{\Gamma} ) } \lesssim  \max_{F \in \mathcal{F} (\partial  \Cutdom) }  \sum\limits_{r=l}^{q+1} h^{r-l}  \norm{D^r ( \Psi^{\Gamma} - \Theta_h^{\Gamma} )   }_{\infty,F}   \\
& \lesssim h^{-l}  \max_{T \in \Cutel }  \sum\limits_{r=0}^{q+1} h^{r}  \norm{D^r ( \Psi^{\Gamma} - \Theta_h^{\Gamma} )   }_{\infty,T} \lesssim h^{q+1-l}. 
\end{align*}
where (\ref{eq:bounds-ExtMeshTrafo}) and (\ref{eq:bound-Psi-Theta}) have been employed. 
For passing from the facet to the element we also used that the involved functions are continuous on the elements so that the $L^{\infty}$-norm coincides with the supremum norm, see e.g.\ \cite[Chapter 6.1]{F99}.
\item Recall that $\Phi - \ID = (\Psi - \Theta_h) \Theta_h^{-1}$. In \cite[Lemma 2]{L_GUFEMA_2017} 
the bounds 
\[
 \norm{D^l \Theta_h^{-1} }_{\infty,\Theta_h(T)} \lesssim 1, \quad
 \text{ for } l \in \{1, \ldots, q+1\} 
\]
have been shown (for $l \in \{0,1\}$ the bounds actually hold globally, see proof of \cite[Lemma 5.5.]{LR17}). 
Hence, in view of part (a) the claim follows from the product rule.
\end{enumerate}
\end{ProofOf}

\begin{ProofOf}{Lemma \ref{lem:normal_tan_proj}} 
We will only prove the results on $\Gamma$ as the statements on $\Gamma_h$ follow analogously.
\begin{enumerate}[label=(\alph*)] 
\item From the reverse triangle inequality we obtain for $x \in \Gamma$ 
\begin{align*}
	\abs{ \norm{ \dPhi^{T}(x) \normal_{\Gamma}(x) }_{2} - 1 } 
	= \abs{ \norm{ \dPhi^{T}(x) \normal_{\Gamma}(x) }_{2 } - \norm{ \normal_{\Gamma}(x) }_{2}  }  
\leq  \norm{  (\dPhi^{T}(x) - I)  \normal_{\Gamma}(x)  }_{2}. 
\end{align*}
Since $\Theta_h(\mesh)$ covers $\Omega$, there exists an element $\el$ such that $x \in \Theta_h(\el)$.
As $\dPhi^{T}$ is continuous on $\Theta_h(\el)$ we obtain:
\[
\norm{  (\dPhi^{T}(x) - I)  \normal_{\Gamma}(x)  }_{2} 
		\lesssim \norm{  \dPhi^{T} - I  }_{ \infty, \Theta_h(T)}
\lesssim h^q,
\]
where (\ref{eq:Bound-Phi-Grad}) was used in the last step.
As a result, the lower bound 
\[  \norm{ \dPhi^{T} \normal_{\Gamma} }_{ 2 } \geq 
1 -  \abs{ \norm{ \dPhi^{T} \normal_{\Gamma} }_{ 2 } - 1 } 
 \gtrsim 1 		\]
for $h$ sufficiently small follows.
\item Using (\ref{eq:normal-trafo-Phi}), (\ref{eq:Bound-Phi-Grad}) and (a) we obtain 
\begin{align*}
& \norm{ \normal_{\Gamma}  - \normal_{\Gamma_h} \circ \Phi_h^{-1} }_2 
	= \frac{1}{ \norm{ \dPhi^{T} \normal_{\Gamma} }_{2} } \norm{  \normal_{\Gamma}  \norm{ \dPhi^{T} \normal_{\Gamma} } -  \dPhi^{T}   \normal_{\Gamma} }_2  \\ 
& \lesssim \norm{ \normal_{\Gamma} \left( \norm{ \dPhi^{T} \normal_{\Gamma} }_2 - 1 \right) + \left( I - \dPhi^{T} \right) \normal_{\Gamma}  }_{2}  
  \lesssim  \abs{ \norm{ \dPhi^{T} \normal_{\Gamma} }_{2 } - 1 } + \norm{ I - \dPhi^{T} }_2  \lesssim h^q.
\end{align*}
\item From the result established in (b) we obtain 
\begin{align*}
\norm{ \Ptan{\Gamma_h} \circ \Phi_h^{-1} - \Ptan{ \Gamma } }_{2}  
	&= \norm{  ( \normal_{\Gamma_h} \circ \Phi_h^{-1} - \normal_{ \Gamma } )  (\normal_{\Gamma_h} \circ \Phi_h^{-1} )^T +  \normal_{ \Gamma } ( \normal_{ \Gamma_h } \circ \Phi_h^{-1} - \normal_{ \Gamma }  )^T  }_{2} \\ 
	& \lesssim \norm{ \normal_{\Gamma_h} \circ \Phi_h^{-1} - \normal_{ \Gamma }  }_{2} +  \norm{ ( \normal_{\Gamma_h} \circ \Phi_h^{-1} - \normal_{ \Gamma }  )^T  }_{2} \lesssim h^q.  
\end{align*}
\end{enumerate}
\end{ProofOf}

\begin{ProofOf}{Lemma \ref{lem:norm-equiv}} 
\begin{enumerate}[label=(\alph*)]
\item See \cite[proof of Lemma 5.10]{LR17}.
\item According to the chain rule 
\begin{align*}
\partial_{y_{\nu}} \partial_{y_{\mu}} \tilde{v}(y)
& = \sum\limits_{j=1}^{d} \sum\limits_{k=1}^{d} (\partial_{j} \partial_{k} v) \circ \Phi_h(y) \partial_{ y_{\nu}} [ \Phi_h ]_{j}(y)  \partial_{ y_{\mu}} [ \Phi_h ]_{k}(y) \\ 
& + \sum\limits_{k=1}^{d} (\partial_{k}v) \circ \Phi_h(y) \partial_{ y_{\nu}} \partial_{ y_{\mu}} [ \Phi_h ]_{k}(y),
\end{align*}
where $ [ \Phi_h ]_{j}, j=1,\ldots,d$ denote the components of $\Phi_h$.
In view of $\Theta_h = \Phi_h^{-1} \circ \Psi$ the transformation formula for integrals yields
\begin{align*}
& \sum\limits_{ \el \in \mesh^{i} }  \int\limits_{ \Theta_h(T) } \abs{ \partial_{y_{\nu}} \partial_{y_{\mu}} \tilde{v}  }^2 \; \dY
= \sum\limits_{ \el \in \mesh^{i} }  \int\limits_{ \Psi(T) } \abs{ ( \partial_{y_{\nu}} \partial_{y_{\mu}} \tilde{v} ) \circ \Phi_h^{-1}(x)  }^2 \det( \dPhi^{-1}(x)) \; \dX \\ 
	& \lesssim \sum\limits_{ \el \in \mesh^{i} }  \sum\limits_{j,k} \int\limits_{ \Psi(T) } \abs{ (\partial_{j} \partial_{k} v (x) }^2 \abs{ \partial_{\nu} [\Phi_h]_{j}(\Phi_h^{-1}(x)) \partial_{\mu} [\Phi_h]_{k}(\Phi_h^{-1}(x)) }^2  \; \dX \\ 
& + \sum\limits_{ \el \in \mesh^{i} } \sum\limits_{k=1}^{d} \int\limits_{ \Psi(T) } \abs{ \partial_{k} v (x) }^2 \abs{ \partial_{\nu} \partial_{\mu} [\Phi_h]_{k}(\Phi_h^{-1}(x)) }^2 \; \dX \\
& \lesssim \sum\limits_{ \abs{\alpha} \leq 2  } \sum\limits_{ \el \in \mesh^{i} }  \int\limits_{ \Psi(T) } \abs{ D_x^{\alpha} v  }^2 \; \dX,
\end{align*}
where we applied the bounds for the second derivatives from Lemma~\ref{lem:Phi-HigherOrderBounds} (b) on  $\Phi_h^{-1} \circ \Psi(T) = \Theta_h(T)  $ and (\ref{eq:DPhi_estimates}). 
\end{enumerate}
\end{ProofOf}

\begin{ProofOf}{Lemma \ref{lem:Deriv-IF}} 
We recall from \cite[eq. (3.40)]{LR17} that the measures $\dS_{\Gamma} $ on $\Gamma$ and $\dS_{\Gamma_h}$ on $\Gamma_h$ are related as follows:  
\begin{equation}\label{eq:Gamma-measure-trafo-Phi}
\dS_{\Gamma} = \det(\dPhi) \norm{ \dPhi^{-T} \normal_{\Gamma_h }  }_{2}  \dS_{\Gamma_h}, \quad 
\dS_{\Gamma_h} = \det(\dPhi^{-1}) \norm{ \dPhi^{T} \normal_{\Gamma }  }_{2}  \dS_{\Gamma}. 
\end{equation}
\begin{enumerate}[label=(\alph*)] 
\item We will only show the result for the gradient. Using $(\nabla v) \circ \Phi_h = \dPhi^{-T} \nabla \tilde{v}$ and the transformation rule (\ref{eq:Gamma-measure-trafo-Phi}) for the measure yields 
\begin{align*}
\int\limits_{\Gamma} \abs{\nabla v}^2 \dS_{\Gamma} 
	&= \int\limits_{\Gamma_h} \abs{ \dPhi^{-T} \nabla \tilde{v} }^2 \det(\dPhi) \norm{ \dPhi^{-T} \normal_{\Gamma_h }  }_{2} \dS_{\Gamma_h} 
\lesssim \int\limits_{\Gamma_h} \abs{ \nabla \tilde{v}}^2 \dS_{\Gamma_h}, 
\end{align*} 
where the last inequality follows from\footnote{Note that (\ref{eq:DPhi_estimates}) makes a statement about $L^{\infty}$-norms in the volume whereas the integral is defined on a hypersurface. To pass from one to the other, we write the latter as sum of integrals over $\Theta_h( \Gamma_{\el} ), \el \in \Cutel$. Using  $\Phi_h \in [C^{q+1}( \Theta_h(\mesh))]^d$ then allows to estimate the $L^{\infty}$-norms on $\Theta_h(\Gamma_{\el})$ by the volumetric norms on  $\Theta_h(\el)$. } equation (\ref{eq:DPhi_estimates}). 
The reverse estimate is obtained in the same way by transforming from $\Gamma_h$ to $\Gamma$ and appealing to equations (\ref{eq:Bound-Phi-Grad})-(\ref{eq:DPhi_estimates}). 
The situation is similar in (b) and (c) where we will likewise only show one of the directions.
\item Using $\dPhi^T \nabla v = (\nabla \tilde{v}) \circ \Phi_h^{-1} $ and the transformation rule (\ref{eq:Gamma-measure-trafo-Phi}) gives 
\begin{align*}
& \int\limits_{\Gamma_h} \norm{ \nabla \tilde{v} \cdot \normal_{ \Gamma_h } }_2^2 \dS_{\Gamma_h	} 
= \int\limits_{\Gamma} \norm{ \dPhi^{T} \nabla v \cdot ( \normal_{ \Gamma_h } \circ \Phi_h^{-1}) }_2^2 \det( \dPhi^{-1} ) \norm{ \dPhi^T \normal_{\Gamma }  }_2   \dS_{\Gamma } \\ 
	  \lesssim & \int\limits_{\Gamma} \left[ \norm{ \dPhi^{T} \nabla v \cdot ( \normal_{ \Gamma_h } \circ \Phi_h^{-1} - \normal_{\Gamma} ) }_2 + \norm{ (\dPhi^{T} - I ) \nabla v \cdot \normal_{ \Gamma }   }_2 \right]^2 \abs{ \det( \dPhi^{-1} ) } \norm{ \dPhi^T \normal_{\Gamma }  }_2   \dS_{\Gamma } \\ 
& + \int\limits_{\Gamma} \norm{ \nabla v \cdot \normal_{ \Gamma } }_2^2 \abs{ \det( \dPhi^{-1} )} \norm{ \dPhi^T \normal_{\Gamma }  }_2  \dS_{\Gamma } \\ 
& \lesssim h^{2q}  \norm{ \nabla v }_{\Gamma}^2 
	+ \norm{ \nabla v \cdot \normal_{ \Gamma  }  }_{ \Gamma }^2,
\end{align*} 
where Lemma~\ref{lem:normal_tan_proj} (b), equation (\ref{eq:Bound-Phi-Grad}) and (\ref{eq:DPhi_estimates}) were employed in the last step.
\item Writing $ \Ptan{\Gamma_h} \circ \Phi_h^{-1} \dPhi^{T} = \Ptan{\Gamma_h} \circ \Phi_h^{-1} (\dPhi^{T} - I ) + (\Ptan{\Gamma_h} \circ \Phi_h^{-1} - \Ptan{\Gamma}) + \Ptan{\Gamma}$ we deduce the claim from Lemma~\ref{lem:normal_tan_proj} (c), equation (\ref{eq:Bound-Phi-Grad}) and (\ref{eq:DPhi_estimates}):  
\begin{align*}
	& \int\limits_{ \Gamma_h } \norm{ \nabla_{\Gamma_h} \tilde{ v} }_2^2 \dS_{\Gamma } 
= \int\limits_{ \Gamma } \norm{  \Ptan{\Gamma_h} \circ \Phi_h^{-1} \dPhi^{T} \nabla v }_2^2  \det( \dPhi^{-1} ) \norm{ \dPhi^T \normal_{\Gamma } }_2  \dS_{\Gamma } \\ 
	& \lesssim \int\limits_{ \Gamma } \left[ \norm{  \Ptan{\Gamma_h} \circ \Phi_h^{-1} ( \dPhi^{T} - I) \nabla v }_2 + \norm{ ( \Ptan{\Gamma_h} \circ \Phi_h^{-1} -  \Ptan{\Gamma} ) \nabla v }_2 \right]^2 \abs{ \det( \dPhi^{-1} )} \norm{ \dPhi^T \normal_{\Gamma } }_2  \dS_{\Gamma } \\
  & + \int\limits_{ \Gamma } \norm{ \Ptan{\Gamma} \nabla v  }_2^2    \abs{ \det( \dPhi^{-1} )} \norm{ \dPhi^T \normal_{\Gamma } }_2 \dS_{\Gamma }  \\ 
	& \lesssim h^{2q}  \norm{ \nabla v }_{\Gamma}^2 + 
\norm{ \nabla_{\Gamma} v  }_{ \Gamma }^2. 
\end{align*}
\end{enumerate}
\end{ProofOf}

\begin{ProofOf}{Lemma \ref{lem:pullback}} 
\begin{enumerate}[label=(\alph*)] 
\item Let $ \gamma(s) := f( s \Phi_h^{-1}(x) + (1-s) x )$ for $s \in [0,1]$ and $x \in M$. We have 
\begin{align*}
f \circ \Phi_h^{-1}(x) - f(x) = \int\limits_{0}^{1} \gamma^{\prime}(s) \mathrm{d}s
& = \int\limits_{0}^{1} (\nabla f( s \Phi_h^{-1}(x) + (1-s) x) , \Phi_h^{-1}(x) - x ) \; \mathrm{d}s \\ 
	& \lesssim \norm{ \nabla f }_{\infty,U} \norm{ \Phi_h^{-1}(x) - x }_2. 
\end{align*}
Hence, we obtain from Lemma~\ref{lem:Trafo-Grad-bounds} that 
\[
\norm{ f \circ \Phi_h^{-1} - f }_M \lesssim \sqrt{\abs{M}}  \norm{\nabla f}_{\infty,U} \norm{ \Phi_h^{-1} - \ID }_{\infty,\Omega} \lesssim h^{q+1} \sqrt{\abs{M}} \norm{\nabla f}_{\infty,U}. 
\]
\item We have
\begin{align*}
& \norm{ \nabla ( f \circ \Phi_h^{-1} -f ) }_M = \norm{ \dPhi^{-T} \nabla f \circ \Phi_h^{-1} - \nabla f }_{M } \\ 
& \lesssim \norm{  (\dPhi^{-T} - I) [ (\nabla f \circ \Phi_h^{-1} - \nabla f ) + \nabla f   ] + (\nabla f \circ \Phi_h^{-1} - \nabla f  ) }_{ M }  \\
& \lesssim \left( \norm{\dPhi^{-T} - I}_{\infty,M} + 1 \right) \norm{ \nabla f \circ \Phi_h^{-1} - \nabla f }_{M} + \norm{ \dPhi^{-T} - I }_{ \infty,M} \norm{ \nabla{f}}_M \\     
& \lesssim h^{q+1} \sqrt{ \abs{M} } \norm{f}_{W^{2,\infty}(U)} + h^q \norm{\nabla f}_M,
\end{align*}
where part (a) was employed to estimate the first term and we used (\ref{eq:DPhi_estimates}) to obtain the bound for the second term.
\item By the chain rule we have 
\begin{align}
& \partial_{x_\nu} \partial_{ x_{\mu} } (f \circ \Phi_h) = 
(\partial_{\nu} \partial_{\mu} f) \circ \Phi_h +  
(\partial_{\nu} \partial_{\mu} f) \circ \Phi_h ( \partial_{x_{\mu}} [ \Phi_h ]_{ \mu} - 1 ) \nonumber \\ 
& (\partial_{\nu} \partial_{\mu} f) \circ \Phi_h ( \partial_{x_{\nu}} [ \Phi_h ]_{ \nu} - 1 ) \partial_{x_{\mu}} [ \Phi_h ]_{ \mu} 
+ \sum\limits_{ (i,j) \neq (\mu,\nu) } (\partial_{j} \partial_{i} f) \circ \Phi_h \partial_{x_{\nu}} [ \Phi_h ]_{ j } \partial_{x_{\mu}} [ \Phi_h ]_{ i } \nonumber \\
& + \sum\limits_{i} (\partial_{i} f) \circ \Phi_h \partial_{x_{\nu}} \partial_{x_{\mu}} [\Phi_h]_{i},
\label{eq:pullback-proof-snd-deriv-chain-rule}
\end{align}
where $[\Phi_h]_{i}$ denotes the $i$-th component of $\Phi_h$. 
According to Lemma~\ref{lem:Trafo-Grad-bounds} we have 
\[
\norm{ \partial_{x_j} [ \Phi_h ]_{i} }_{ \infty, \Omega } \lesssim h^q \text{ for } i \neq j, \quad
		\norm{ \partial_{x_j} [ \Phi_h ]_{j} -1 }_{ \infty, \Omega }  \lesssim h^q \text{ for } j=1,\ldots,d. 
\]
Combining this with the bounds for the second derivatives from Lemma~\ref{lem:Phi-HigherOrderBounds} on $M = \Theta_h(T)$ yields 
\begin{align*}
& \norm{ \partial_{x_\nu} \partial_{ x_{\mu} } (f \circ \Phi_h) - \partial_{x_\nu} \partial_{ x_{\mu}  } f  }_{ M } \\ 
& \lesssim 
 \norm{ (\partial_{x_\nu} \partial_{ x_{\mu} } f) \circ \Phi_h - \partial_{x_\nu} \partial_{ x_{\mu}  } f  }_{ M } + \sqrt{\abs{M}} \left( h^q \norm{f}_{ W^{2,\infty}(V) } + h^{q-1} \norm{f}_{ W^{1,\infty}(V) }  \right)  \\ 
& \lesssim h^{q-1} \sqrt{ \abs{M} } \norm{ f }_{W^{3,\infty}(V)},
\end{align*}
where part (a) was once again applied in the last step.
\item If we precompose equation (\ref{eq:pullback-proof-snd-deriv-chain-rule}) with $\Phi_h^{-1}$ we obtain
\begin{align*}
	& \partial_{x_{\nu}} \partial_{x_{\mu}} (f \circ \Phi_h) \circ \Phi_h^{-1}  - \partial_{x_{\nu}} \partial_{x_{\mu}} f =  (\partial_{\nu} \partial_{\mu} f) ( \partial_{x_{\nu}} [ \Phi_h ]_{ \nu} \circ \Phi_h^{-1} - 1 ) \partial_{x_{\mu}} [ \Phi_h ]_{ \mu} \circ \Phi_h^{-1}    \\
& + (\partial_{\nu} \partial_{\mu} f) ( \partial_{x_{\mu}} [ \Phi_h ]_{ \mu} \circ \Phi_h^{-1} - 1 )
+ \sum\limits_{ (i,j) \neq (\mu,\nu) } (\partial_{j} \partial_{i} f)  \partial_{x_{\nu}} [ \Phi_h ]_{ j } \circ \Phi_h^{-1} \partial_{x_{\mu}} [ \Phi_h ]_{ i } \circ \Phi_h^{-1} \\ 
	&  + \sum\limits_{i} (\partial_{i} f)  \partial_{x_{\nu}} \partial_{x_{\mu}} [\Phi_h]_{i} \circ \Phi_h^{-1}.
\end{align*}
Since the term  $\partial_{x_{\nu}} \partial_{x_{\mu}} f$ is now on the left hand side of the equation, it is no longer necessary to appeal to part (a) to estimate this term. 
We can directly apply Lemma~\ref{lem:Trafo-Grad-bounds} to estimate the first order derivatives of $\Phi_h$ and Lemma~\ref{lem:Phi-HigherOrderBounds} to obtain bounds for the second order derivatives on $\Phi_h^{-1} \circ \Psi(T) = \Theta_h(T)$. Hence, 
\begin{align*}
  \norm{ \partial_{x_{\nu}} \partial_{x_{\mu}} (f \circ \Phi_h) \circ \Phi_h^{-1}  - \partial_{x_{\nu}} \partial_{x_{\mu}} f  }_{ \Psi(T) }
\lesssim
   h^{q} \norm{f}_{ H^2( \Psi(T)) } +  h^{q-1 } \norm{ \nabla f}_{  \Psi(T) }
\end{align*}
follows.
\end{enumerate}
\end{ProofOf}

\section*{Acknowledgements}
We would like to thank Mihai Nechita for valuable input on the interface finite element formulation. 

\bibliographystyle{plain}

\bibliography{biblio}

\end{document}